\documentclass[11pt, oneside]{article}
\usepackage[margin=1in]{geometry}

\usepackage{graphicx}	
\usepackage{amssymb, amsmath, amsfonts, amsthm,url,enumitem}

\newtheorem{theorem}{Theorem}[section]
\newtheorem{lemma}[theorem]{Lemma}
\newtheorem{corollary}[theorem]{Corollary}

\newcommand{\ignore}[1]{}

\newcommand{\matwo}[4]{\begin{pmatrix} #1 & #2 \\ #3 & #4 \end{pmatrix}}

\newcommand{\vecthree}[3]{\begin{pmatrix} #1 \\ #2 \\ #3 \end{pmatrix}}

\newcommand{\R}{\mathbb{R}}

\newcommand{\N}{\mathbb{N}}

\newcommand{\FF}{\mathbb{F}}

\newcommand{\RR}{\mathbb R}
\newcommand{\CC}{\mathbb C}

\newcommand{\ZZ}{\mathbb Z}
\newcommand{\SSS}{\ensuremath{\mathbb S}}

\newcommand{\pts}{\mathcal P}
\newcommand{\lines}{\mathcal L}
\newcommand{\flats}{\mathcal F}
\newcommand{\BI}{\boldsymbol 1}

\newcommand{\cyl}{\mathcal C}

\newcommand{\Spin}[1]{\text{\textnormal{Spin}}(#1)}

\newcommand{\Pin}[1]{\mbox{Pin}(#1)}
\newcommand{\Spun}[1]{\text{\textnormal{Spun}}(#1)}
\newcommand{\SO}[1]{\mbox{SO}(#1)}
\newcommand{\SE}[1]{\mbox{SE}(#1)}
\newcommand{\Ort}[1]{\mbox{O}(#1)}

\newcommand{\flatss}{{\mathcal F}_\pts}
\newcommand{\flatsu}{{\mathcal F}_U}
\newcommand{\vb}{{\bf V}}
\newcommand{\ib}{{\bf I}}
\newcommand{\jb}{{\bf J}}
\newcommand{\rank}{\mathrm{rank\,}}

\def\eps{{\varepsilon}}

\newcommand{\parag}[1]{\vspace{2mm}

\noindent{\bf #1} }

\begin{document}

\title{A Reduction for the Distinct Distances Problem in  $\RR^d$}
\author{Sam Bardwell--Evans\thanks{California Institute of Technology, Pasadena, CA, USA. Supported by Caltech's Summer Undergraduate Research Fellowships (SURF) program.
{\sl sam.bardwell.evans@gmail.com}. }
\and
Adam Sheffer\thanks{Department of Mathematics, Baruch College, City University of New York, NY, USA.
{\sl adamsh@gmail.com}. Supported by NSF grant DMS-1710305. Corresponding author. 55 Lexington Ave, New York, NY 10010, room 6-291.
}}
\maketitle

\begin{abstract}
We introduce a reduction from the distinct distances problem in $\RR^d$ to an incidence problem with $(d-1)$-flats in $\RR^{2d-1}$.
Deriving the conjectured bound for this incidence problem (the bound predicted by the polynomial partitioning technique) would lead to a tight bound for the distinct distances problem in $\RR^d$.
The reduction provides a large amount of information about the $(d-1)$-flats, and a framework for deriving more restrictions that these satisfy.

Our reduction is based on introducing a Lie group that is a double cover of the special Euclidean group.
This group can be seen as a variant of the Spin group, and a large part of our analysis involves studying its properties.
\end{abstract}


\section{Introduction}

The \emph{Erd\H os distinct distances problem} is a main problem in Discrete Geometry, which asks for the minimum number of distinct distances spanned by a set of $n$ points in $\RR^2$.
That is, denoting the distance between two points $p,q\in \RR^2$ as $|pq|$, we wish to find $\min_{|\pts|=n} |\{|pq| :\, p,q\in \pts \}|$.

In 1946, Erd\H os \cite{erd46} observed that a $\sqrt{n} \times \sqrt{n}$ section of the integer lattice $\ZZ^2$ spans $\Theta(n/\sqrt{\log n})$ distinct distances (this observation is an immediate corollary of a number theoretic result of Landau and Ramanujan).
Erd\H os conjectured that no set of $n$ points in $\RR^2$ spans an asymptotically smaller number of distinct distances.
Proving that every set of $n$ points in $\RR^2$ spans $\Omega(n/\sqrt{\log n})$ distinct distances turned out to be a difficult problem, to have a deep underlying theory, and to have strong connections to several other parts of mathematics.

After over 60 years and many works on the distinct distances problem, Guth and Katz \cite{GK15} proved that every set of $n$ points in $\RR^2$ spans $\Omega(n/\log n)$ distinct distances.
Their proof involves studying properties of polynomials, partly by using tools from Algebraic Geometry.
This work began a new era of polynomial methods in Discrete Geometry.

Already in his 1946 paper, Erd\H os observed that a $n^{1/d}\times n^{1/d} \times \cdots \times n^{1/d}$ section of the integer lattice $\ZZ^d$ spans $\Theta(n^{2/d})$ distinct distances.
He then conjectured that this construction is asymptotically best possible, in the sense that every set  of $n$ points in $\RR^d$ spans $\Omega(n^{2/d})$ distinct distances.
When the Guth--Katz paper first appeared, it seemed that similar techniques might solve the distinct distance problem in $\RR^d$.
However, over six years have passed and no new results were obtained for this problem.
Before the new era of polynomial methods, Solymosi and Vu \cite{SV08} derived a lower bound for the number of distinct distances in $\RR^d$.
This bound was obtained by an induction on the dimension $d$.
The current best bounds for distinct distances in $\RR^d$ are obtained by using this induction, with the planar distinct distances theorem as the induction basis.
For example, this implies that every $n$ points in $\RR^3$ determine $\Omega^*(n^{3/5})$ distinct distances.\footnote{In the $\Omega^*(\cdot)$-notation we ignore polylogarithmic factors.}

The proof of the planar distinct distances theorem reduces the problem into a point-line incidence problem in $\RR^3$ (based on a previous work by Elekes and Sharir \cite{ES11}), and then solves the incidence problem by using polynomial methods.
Specifically, given a finite set of lines $\lines$ in $\RR^d$ and a positive integer $k$, we say that a point in $\RR^d$ is $k$-\emph{rich} if it is contained in at least $k$ lines of $\lines$.
The planar distinct distances theorem was reduced to the following problem.
\begin{theorem}[Guth and Katz \cite{GK15}] \label{eq:GKRichPnts}
Let $\lines$ be a set of $n$ lines in $\RR^3$ such that no point of $\RR^3$ is contained in more than $\sqrt{n}$ lines of $\lines$.
Moreover, every plane,  hyperbolic paraboloid, or single-sheeted hyperboloid contains $O(\sqrt{n})$ lines of $\lines$.
Then for every $k\ge 2$, the number of $k$-rich points is $O\left(\frac{n^{3/2}}{k^2}+\frac{n}{k}\right)$.
\end{theorem}

It is possible to imitate the reduction of the planar distinct distances problem in higher dimensions.
However, already for distinct distances in $\RR^3$ this leads to an incidence problem with somewhat involved varieties that are difficult to study.
For example, it is not clear how to bound the number of varieties that can be contained in a hyperplane.

The main contribution of this paper is a more involved reduction that leads to a simpler incidence problem.
It is significantly easier to establish properties of the varieties in this problem.
We refer to $k$-dimensional planes in $\RR^d$ as $k$-\emph{flats}.
Let $\SSS^d$ be the hypersphere in $\RR^{d+1}$ that is centered at the origin and of radius 1.

\begin{theorem} \label{th:DDreduction}
The problem of deriving a lower bound on the minimum number of distinct distances spanned by $n$ points in $\RR^d$ can be reduced to the following problem:
\vspace{-2mm}

\begin{quotation}
Let $\flats$ be a set of $n$ distinct $(d-1)$-flats in $\RR^{2d-1}$, such that every two flats intersect in at most one point,
every point of $\RR^{2d-1}$ is contained in $O(\sqrt{n})$ flats of $\flats$, and every hyperplane in $\RR^{2d-1}$ contains $O(\sqrt{n})$ of these flats.
Find an upper bound on the number of $k$-rich points, for every $2\le k = O(n^{1/d+\eps})$ (for some $\eps>0$).
\end{quotation}
\vspace{-2mm}

Deriving the bound $O\left(\frac{n^{(2d-1)/d}}{k^{2+\eps}}\right)$ for the number of $k$-rich points would yield the conjectured lower bound of $\Omega(n^{2/d})$ distinct distances.
\end{theorem}

{\bf Remarks.}
(i) Using our methods, we obtained the same reduction for the case where the points are on the hypersphere $\SSS^d$ rather than in $\RR^d$.
Since the paper is already rather long and technical, we decided not to include the proof of this case. \\
(ii) For $\alpha\ge 0$, deriving the bound $O\left(\frac{n^{\alpha+(2d-1)/d}}{k^{2+\eps}}\right)$ for the number of $k$-rich points would yield a lower bound of $\Omega(n^{2/d-2\alpha})$ distinct distances. \\
(iii) The $\eps$ in the bound $k = O(n^{1/d+\eps})$ comes from an incidence bound of Solymosi and Tao \cite{ST12}.
It is conjectured that this $\eps$ can be removed from the bound of \cite{ST12}, and this would immediately remove the $\eps$ from the restriction on $k$.\\
(iv) Usually a bound on the number of $k$-rich points also includes an extra term of the form $n/k$, which dominates the bound when $k$ is large.
Since we are only interested in small values of $k$, this extra term is not relevant in our case.
\vspace{1.5mm}

The incidence problem stated in Theorem \ref{th:DDreduction} is false without including additional restrictions.
That is, one can obtain point-flat constructions in $\RR^{2d-1}$ that have too many incidences.
Using our framework, one can overcome this issue by deriving many additional restrictions for the incidence problem.
We do not mention such restrictions in Theorem \ref{th:DDreduction}, since it is not clear what are the natural restrictions for deriving the incidence bounds.
Instead, in Section \ref{sec:FlatsR5} we demonstrate how our framework can be used to derive additional restrictions, considering specifically the case of distinct distances in $\RR^3$.

The problem of distinct distances in $\RR^3$ leads to an incidence problem with 2-flats in $\RR^5$.
In Section \ref{sec:FlatsR5}, we bound the number of such 2-flats that can be contained in a constant-degree variety in $\RR^5$. We also bound the number of 2-flats that can have a one-dimensional intersection with a constant-degree two-dimensional variety in $\RR^5$.
Some of these results are conditional on having a good distinct distances bound for points on constant-degree surfaces in $\RR^3$.
Thus, to obtain the conjectured distinct distances bound in $\RR^3$, it is possible that one would first need to derive a distinct distances bound for the special case of points on a surface in $\RR^3$.
Currently, such bounds are known for planes, spheres, and two-sheeted hyperboloids (for example, see \cite{Tblog11}).
However, we are far from having this bound for arbitrary constant-degree surfaces in $\RR^3$.
For more details, see Sharir and Solomon \cite{SharirSolo16}.

The current best bounds for incidences with varieties in $\RR^d$ are obtained by the \emph{polynomial partitioning} technique (for example, see \cite{FPSSZ17,ST12}).
We can efficiently apply this technique to incidences with $(d-1)$-flats in $\RR^{2d-2}$, but the case of $(d-1)$-flats in $\RR^{2d-1}$ seems to be just beyond the current capabilities.
There is a simple way to estimate the bounds that the polynomial  partitioning technique is expected to yield after overcoming the current issues.\footnote{This is done by bounding the number of incidences in the cells while ignoring the incidences on the partition itself. See for example \cite[Chapter 8]{ShefferBook}. }
In the case of $(d-1)$-flats in $\RR^{2d-1}$, the expected incidence bound is $m_{k} = O\left(\frac{n^{(2d-1)/d}}{k^{(3d-2)/d}}+ \frac{n}{k}\right)$.
Note that this is the incidence bound required in Theorem \ref{th:DDreduction} to obtain a tight bound for the distinct distances problem in $\RR^d$.

Theorem \ref{th:DDreduction} states three restrictions on the set of flats $\flats$: the maximum number of flats incident to a common point, the maximum number of flats contained in a common hyperplane, and the size of the intersection of any two flats.
Our framework can be used to obtain additional information about the flats of $\flats$.
In particular, before obtaining the set $\flats$ of $(d-1)$-flats in $\RR^{2d-1}$, we get a set $\lines$ of $\binom{d}{2}$-flats in $\RR^{\binom{d+1}{2}}$.
To move to the space $\RR^{2d-1}$ from the statement of Theorem \ref{th:DDreduction}, we intersect $\lines$ with a generic $(2d-1)$-flat.
In Section \ref{sec:structure} we describe the exact structure of the flats of $\lines$ (that is, the equations that define each flat).
This structure can be used to obtain additional properties of the flats of $\lines$, and thus of the flats of $\flats$.
It is currently unclear what additional properties would be needed to solve the resulting incidence problem, but given such properties it seems reasonable that our techniques would lead to the derivation of the corresponding restrictions.

The inspiration for this work came from a blog post of Tao \cite{Tblog11}.
Tao states that he wrote this post ``to record some observations arising from discussions with Jordan Ellenberg, Jozsef Solymosi, and Josh Zahl.''
The post describes a reduction from the problem of finding a lower bound for the number of distinct distances spanned by points on the sphere $\SSS^2$ to Theorem \ref{eq:GKRichPnts}.
It also shows how the case of distinct distances in $\RR^2$ can be viewed as a scaling limit of the case of distinct distances in $\SSS^2$.
This is an alternative way to reduce the planar distinct distances problem to a point-line incidence problem in $\RR^3$.
While the original reduction can be seen as based on the Lie group $\SE{2}$, the reduction in the blog post is based on the Lie group $\Spin{3}$ (a brief introduction to these groups can be found in Section \ref{sec:Perlim}).
A more direct approach to distinct distances on $\SSS^2$ was presented by Rudnev and Selig \cite{RS16}.

To derive a reduction from the distinct distances problem in $\RR^d$ we introduce a variant of the group $\Spin{d}$, which we denote $\Spun{d}$.
While $\Spin{d}$ is a double cover of $\SO{d}$, the group $\Spun{d}$ is a double cover of $\SE{d}$.
A large part of our analysis deals with studying properties of $\Spun{d}$.

In Section \ref{sec:Perlim} we briefly describe several Lie groups that we rely on.
In Section \ref{sec:SpunD} we introduce the group $\Spun{d}$ and study its structure.
In Section \ref{sec:DDR3} we derive Theorem \ref{th:DDreduction} for the special case of distinct distances in $\RR^3$.
We present this case separately since it is simpler to prove and provides more intuition about what is happening in the proof.
In Section \ref{sec:DistancesRd} we extend the analysis from Section \ref{sec:DDR3} to any dimension.
Finally, in Section \ref{sec:structure} we derive the defining equations of the flats of $\lines$, as stated above.

\parag{Acknowledgments.}
We would like to thank Joshua Zahl for many discussions that eventually led to Section 7.
We wish to thank William Ballinger and Dmitri Gekhtman for several helpful discussions.
We would also like to thank the anonymous referees for helping to improve a previous draft of this work.

\section{Preliminaries: Lie groups} \label{sec:Perlim}

In our analysis we rely on a specific family of Lie groups.
In this section we briefly introduce these groups and some of their properties.
In Section \ref{sec:SpunD} we will introduce our own Lie group and study it in more detail.

Given a point $p\in \RR^d$, we denote by $\|p\|$ the standard Euclidean norm of $p$.
Given two points $p,q\in \RR^d$, we denote by $|pq|$ the Euclidean distance between them (that is, $\|p-q\|$).

\parag{Groups of rigid motions.}
A \emph{rigid motion} (or \emph{isometry}) of $\RR^d$ is a transformation $T:\RR^d \to \RR^d$ that preserves Euclidean distances.
That is, for every $v,u\in \RR^d$ we have that $|uv|=|T(u)T(v)|$.
It is well known that every rigid motion of $\RR^d$ is a combination of translations, rotations, and reflections.
A rigid motion is said to be \emph{proper} if it is a combination of translations and rotations.
In $\RR^2$, a rigid motion is proper if and only if for every three points $a,b,c\in \RR^d$, the path $a\to b \to c$ forms a right turn if and only if $T(a)\to T(b) \to T(c)$ forms a right turn (that is, if the rigid motion is \emph{orientation preserving}).
A similar definition exists in higher dimensions.
The \emph{Special Euclidean group} of $\RR^d$, denoted $\SE{d}$, is the group of proper rigid motions of $\RR^d$ under the operation of composition.

The \emph{Orthogonal group} $\Ort{d}$ is the group of rigid motions of $\RR^d$ that fix the origin.
It consists of the rotations around the origin and the reflections about a hyperplane incident to the origin.
Equivalently, we can think of $\Ort{d}$ as the set of rigid motions that take $\SSS^{d-1}$ to itself.
The \emph{Special Orthogonal group} $\SO{d}$ is the group of \emph{proper} rigid motions of $\RR^d$ that fix the origin (equivalently, of proper rigid motions that take $\SSS^{d-1}$ to itself).
It consists of the rotations around the origin.
Note that $\SO{d}$ is a subgroup of both $\Ort{d}$ and $\SE{d}$.

For any unproved claims in the the remainder of this section, see \cite[Sections 1.2--1.4]{Gallier08}.

\parag{Clifford algebras.}
A \emph{Clifford algebra} is defined with respect to a vector space and to a symmetric bilinear form.
We only define a special case of this algebra: the Clifford algebra associated with $\RR^d$ and the Euclidean norm.
This is a real unitary algebra $C\ell_d$ with a linear map $i:\RR^d \to C\ell_d$ that satisfies the following two properties.
For every $v\in \RR^d$, we have $i(v)^2 = -\|v\|^2\cdot \BI$, where $\BI$ is the multiplicative identity element of $C\ell_d$.
Moreover, if $A$ is a real algebra and $f:\RR^d \to A$ is a linear map satisfying $f(v)^2 = -\|v\|^2 \cdot \BI$ for all $v\in \R^d$, then there exists an algebra homomorphism $\phi: C\ell_d \to A$ such that $f = \phi\circ i$.
It can be shown that the algebra $C\ell_d$ is unique up to an isomorphism.

We now present a more constructive definition of the Clifford algebra $C\ell_d$ (the definition that we will actually rely is in the following paragraph).
For a vector space $V$, we denote by $V^{\otimes k}$ the $k$-fold tensor product of $V$ with itself.
Consider the direct sum $\bigoplus_{k\in \N}\left(\R^d\right)^{\otimes k}$, and let $\mathcal{I}$ be the ideal in this tensor algebra that is generated by all elements of the form $v\otimes v + \|v\|^2\cdot\BI$.
Then we can write $C\ell_d$ as the quotient
\[ \bigoplus_{k\in \N}\left(\R^d\right)^{\otimes k} / \mathcal{I}. \]
Let $j:\R^n \to \bigoplus_{k\in \N}\left(\R^d\right)^{\otimes k}$ be the natural injection, and let $\pi: \bigoplus_{k\in \N}\left(\R^d\right)^{\otimes k} \to \bigoplus_{k\in \N}\left(\R^d\right)^{\otimes k}/\mathcal{I}$ be the natural quotient map.
Then the linear map associated with $C\ell_d$ is the composition $\pi \circ j$.

For our purposes, it would be more intuitive to think of the Clifford algebra $C\ell_d$ as follows.
Let $e_1,\ldots,e_d$ denote the image of an element of the standard basis of $\R^n$ under the map $i$.
When dealing with tensor products of elements of $C\ell_d$, we will write $xy$ instead of $x\otimes y$.
Note that $C\ell_d$ is a $2^d$-dimensional real vector space with basis $\BI, e_1, \ldots, e_d, e_1e_2,\ldots, e_1e_d, e_2e_3,\ldots, e_1\cdots e_d$ (that is, the $2^d$ subsets of $\{e_1,\ldots,e_d\}$).
Recalling the definition of $\mathcal{I}$, we note that the Clifford algebra satisfies $e_j^2 = -\BI$ for every $1\le j \le d$.
Moreover, a simple argument shows that $e_je_k = -e_ke_j$ for every $1\le j,k \le d$ with $j\neq k$.
This explains why in the basis of $C\ell_d$ we do not have combinations of elements $e_1,\ldots,e_d$ where some $e_k$ repeats more than once.

Let $\alpha: C\ell_d \to C\ell_d$ be the automorphism satisfying $\alpha(\BI) = \BI$ and $\alpha(e_j) = -e_j$ for all $j$.
Let $t: C\ell_d \to C\ell_d$ be the anti-automorphism satisfying $t(xy) = t(y)t(x)$, $t(e_j) = e_j$ for all $j$, and $t(\BI) = \BI$.
For example, we have $\alpha(e_1 + e_1e_2) = -e_1 + e_1e_2$ and $t(e_1+e_1e_2) = e_1 + e_2e_1 = e_1-e_1e_2$.
It can be shown that the functions $\alpha$ and $t$ are uniquely defined.
We define the \emph{conjugate} of $x\in C\ell_d$ as $\overline{x}=\alpha(t(x))=t(\alpha(x))$.
We also define the norm $N(x) = x\overline{x}$.
Returning to the above example, we have $\overline{e_1 + e_1e_2} = -e_1 - e_1e_2$ and $N(e_1 + e_1e_2) = 2\cdot \BI$.
Note that for every $v\in \RR^d$ and $x = i(v)$ we have $\overline{x} = -x$, which in turn implies $N(x) = \|v\|^2$.

We are especially interested in elements $x\in C\ell_d$ that satisfy $\alpha(x)i(v)x^{-1} \in i(\RR^d)$ for every $v\in \RR^d$.
One advantage of working with such elements is that their norm is well behaved.
\begin{lemma} \label{le:Norm}
\hspace{2mm}\\
(i) Let $x\in C\ell_d$ satisfy $\alpha(x)i(v)x^{-1} \in i(\RR^d)$ for every $v\in \RR^d$.
Then $N(x)= r \cdot \BI$ for some $r\in \RR$. \\
(ii) Consider a second element $y\in C\ell_d$ that satisfies $\alpha(y)i(v)y^{-1} \in i(\RR^d)$ for every $v\in \RR^d$.
Then $N(xy) = N(x)N(y)$. \\
(iii) Let $x=i(u)$ for $u\in \RR^d$.
Then $\alpha(x)i(v)x^{-1} \in i(\RR^d)$ for every $v\in \RR^d$.
\end{lemma}
\begin{proof}
(i) See \cite[Proposition 1.8]{Gallier08}.

(ii) By part (i) of the lemma, $N(y)= r \cdot \BI$ for some $r\in \RR$, so $N(y)$ commutes with everything in $C\ell_d$.
This implies
\[ N(xy)= xy\overline{xy} = xy\overline{y}\hspace{0.6mm}\overline{x} = xN(y)\overline{x} = x\overline{x}N(y) = N(x)N(y). \]

(iii) See \cite[Proposition 1.6]{Gallier08}.
\end{proof}

Rather than working with all of $C\ell_d$, we will rely on the subalgebra
\[ C\ell_d^0 = \{x\in C\ell_d :\, \alpha(x) = x\}. \]
This is the $2^{d-1}$-dimensional subspace of $C\ell_d$ generated by the elements of the basis of $C\ell_d$ that are the product of an even number of distinct $e_j$'s.
Similarly, we set $C\ell_d^1 = \{x\in C\ell_d :\, \alpha(x) = -x\}$.
This is a $2^{d-1}$-dimensional subspace (not a subalgebra), and is generated by the elements of the basis of $C\ell_d$ that are the product of an odd number of distinct $e_j$'s.

\parag{Spin groups.}
Denote the multiplicative groups of $C\ell_d$ and $C\ell_d^0$ as $C\ell_d^\times$ and ${C\ell_d^{0\times}}$, respectively.
We define the Lie groups
\begin{align}
\Pin{d} &= \{ x\in C\ell_d^{\times} \ :\, N(x) = \BI \quad \text{ and } \quad \alpha(x)i(v)x^{-1} \in i(\R^d) \ \text{ for every } v\in \R^d\}, \nonumber \\
\Spin{d} &= \{ x\in C\ell_d^{0\times} \ :\, N(x) = \BI \quad \text{ and } \quad xi(v)x^{-1} \in i(\R^d) \ \text{ for every } v\in \R^d\}. \label{eq:SpinDef}
\end{align}
Note that in the definition of $\Spin{d}$ we can replace $xi(v)x^{-1}$ with $\alpha(x)i(v)x^{-1}$, since $x=\alpha(x)$ for every $x\in C\ell_d^{0\times}$.

An equivalent definition for $\Pin{d}$ is the set of elements that can be written as $i(v_1) i(v_2)\cdots i(v_k)$, where $v_1,\ldots,v_k\in \SSS_{d-1}$ (and $k$ is not fixed).
Similarly, an equivalent definition of $\Spin{d}$ is the set of elements that can be written as $i(v_1) i(v_2)\cdots i(v_k)$, where $v_1,\ldots,v_k\in \SSS_{d-1}$ and $k$ is even.

For $\gamma \in \Pin{d}$ and $v\in \R^d$, we denote the group action of $\gamma$ on $p$ as $p^\gamma$.
This group action is $v^\gamma = i^{-1}(\alpha(\gamma) i(v) \gamma^{-1})$.
Notice that $i$ is injective when considered as a function from $\RR^d$ to $i(\RR^d)$.
When $v\in \R^d$ we have $\alpha(\gamma) i(v) \gamma^{-1}\in i(\RR^d)$, so $v^\gamma = i^{-1}(\gamma i(v) \gamma^{-1})$ is well defined.

By Lemma \ref{le:Norm}, any $\gamma\in \Pin{n}$ satisfies
\[ N(\alpha(\gamma) i(v) \gamma^{-1}) = N(\alpha(\gamma)) N(i(v)) N(\gamma^{-1}) = N(i(v)) = \|v\|^2 \cdot \BI. \]
That is, the transformation of $\RR^d$ induced by the action of $\gamma$ preserves the Euclidean norm, and is thus in $\Ort{d}$.
Letting $\rho: \Pin{d} \to \Ort{d}$ be defined by $\rho(x)(v) = i^{-1}(\alpha(x)i(v)x^{-1})$, we get that $\rho$ is surjective with kernel $\{\BI, -\BI\}$.
That is, $\Pin{d}$ is a double cover of $\Ort{d}$.
In the special case where $\gamma = i(w) \in i(\R^d)\subseteq \Pin{d}$, the action of $\rho(\gamma)$ corresponds to a reflection of $\RR^d$ about the hyperplane orthogonal to $w$ and incident to the origin.

The restricted transformation $\rho: \Spin{d} \to \SO{d}$ is also surjective with kernel $\{\BI,-\BI\}$.
For some intuition, recall that the composition of two reflections about hyperplanes incident to the origin is a rotation centered at the origin.
Thus, the tensor product of two elements of $i(\R^d)$ corresponds to a rotation in $\Spin{d}$.
Similarly, the composition of rotations around the origin is a rotation around the origin.

\ignore{% ------------------------------------------------------------------
By the above, if we wish to consider the intersections of subsets of $\SO{d}$, we may instead consider the intersections of their preimages under $\rho$.
This will be our main use of $\Spin{d}$.
Specifically, when we study distinct distances using rotations, it will be more convenient to think of the rotations as being in $\Spin{d}$ rather than in $\SO{d}$.
} %---------------------------------------- end ignore ----------------------

A proof of the following lemma can be found in \cite[Section 1.4]{Gallier08}.

\begin{lemma} \label{le:C4InR4}
Let $d\le 5$ and let $x\in C\ell_d^0$ satisfy $N(x) = \BI$.
Then for every $v\in \R^d$ we have $xi(v)x^{-1} \in i(\RR^d)$.
\end{lemma}
The claim of Lemma \ref{le:C4InR4} is false for $d\ge 6$.
Combining this lemma with the definition in \eqref{eq:SpinDef} yields the following result.
\begin{corollary}\label{co:SpinDefLowD}
For $d\le 5$, we have
\[ \Spin{d} = \{ x\in C\ell_d^{0\times} \ :\, N(x) = \BI \}. \]
\end{corollary}

We will also rely on the following observation.
\begin{lemma} \label{le:AntiCommute}
If $u,v \in \RR^d$ are orthogonal vectors then $i(u)i(v) = -i(v)i(u)$.
\end{lemma}
\begin{proof}
We set $u' = u/\|u\|$ and $v' = v/\|v\|$, so that $u',v' \in \SSS^{d-1}$.
Since $u$ and $v$ are orthogonal, so are $u'$ and $v'$.
Thus, there exists $\gamma\in \Spin{d}$ that corresponds to a rotation taking $e_1$ to $u'$ and $e_2$ to $v'$.
Since $e_1e_2 = -e_2e_1$, we have
\[ \gamma e_1\gamma^{-1}\gamma e_2\gamma^{-1} = -\gamma e_2\gamma^{-1}\gamma e_1\gamma^{-1} \quad \text{ which implies }  \quad  i(u')i(v') = -i(v')i(u').\]
The assertion of the lemma is obtained by multiplying both sides by $\|u\|\cdot \|v\|$.

The above argument holds for $d\ge 3$.
When $d=2$, there might not exist $\gamma\in \Spin{d}$ that takes $e_1$ to $u'$ and $e_2$ to $v'$.
In that case we can consider instead $\gamma\in \Spin{d}$ that takes $e_1$ to $v'$ and $e_2$ to $u'$
\end{proof}

\section{The group $\Spun{d}$}\label{sec:SpunD}

In this section we introduce the group $\Spun{d}$.
We first construct a variant $X_d$ of the Clifford algebra $C\ell_d$.
Consider the direct sum $\bigoplus_{k\in \N}\left(\RR^{d+2}\right)^{\otimes k}$, and let $\mathcal{I}$ be the ideal in this tensor algebra that is generated by
\begin{align*}
&\left\{e_j \otimes e_k+e_k \otimes e_j :\ 1 \le 1\le j,k \le d+1 \right\}  \bigcup \left\{e_{d+2} \otimes e_{d+2} \right\} \\
& \hspace{60mm} \bigcup \left\{e_j\otimes e_j+\BI, e_{d+2}\otimes e_j - e_j \otimes e_{d+2} :\ 1 \le 1\le j \le d+1 \right\}.
\end{align*}
Then we write $X_d$ as the quotient
\[ \bigoplus_{k\in \N}\left(\R^{d+2}\right)^{\otimes k} / \mathcal{I}. \]

For brevity we write $e_ie_j$ instead of $e_i \otimes e_j$.
Let $i: \RR^{d}\to X_d$ be the linear map that takes the standard basis elements of $\RR^{d}$ to $e_1,\hdots, e_{d}$, respectively.
Let $\alpha: X_d \to X_d$ be the automorphism satisfying $\alpha(\BI) = \BI$, $\alpha(e_j) = -e_j$ for every $1\le j \le d+1$, and $\alpha(e_{d+2})=e_{d+2}$.
Let $t: X_d \to X_d$ be the anti-automorphism satisfying $t(xy) = t(y)t(x)$, $t(e_j) = e_j$ for every $1\le j\le d+2$, and $t(\BI) = \BI$.
For example, when $d=4$ we have
\begin{align*}
\alpha(e_3e_4 + e_1e_5e_6 + e_2e_6) &= e_3e_4 + e_1e_5e_6 - e_2e_6, \\
t(e_3e_4 + e_1e_5e_6 + e_2e_6) &= -e_3e_4 - e_1e_5e_6 + e_2e_6.
\end{align*}

For every $x\in X$, we define the \emph{conjugate} of $x$ as $\overline{x}=\alpha(t(x))=t(\alpha(x))$, and the norm of $x$ as $N(x) = x\overline{x}$.
Note that for every $x = i(v) \in i(\RR^{d})$ we have $\overline{x} = -x$, which in turn implies $N(x) = \|v\|^2\cdot \BI$.
%We also consider the subalgebra $X_d^0 = \{ x\in X_d :\, \alpha(x) = x\}$ and the subspace $X_d^1 = \{x \in X_d \setst \alpha(x) = -x\}$.

We define the standard basis of $X_d$ to consist of $\BI$ and of the tensor products of any number of distinct elements from $\{e_1,\ldots,e_{d+2}\}$.
It is not difficult to verify that this set generates $X_d$ and is linearly independent.
Let $Z_d^0\subset X_d^0$ be the subspace generated by $\BI$ and by products of an even number of elements from $\{e_1,e_2,\ldots,e_d,e_{d+1}e_{d+2}\}$ (note that $e_{d+1}e_{d+2}$ is a single element).

For $1\le k \le d+1$, note that the subspace of $X_k$ generated by $\BI$ and by products of distinct elements from $\{e_1,\ldots,e_{k}\}$ is a subalgebra of $X_d$.
This subalgebra is isomorphic to the Clifford algebra $C\ell_{k}$, and we thus refer to it as $C\ell_{k}$.
With this notation, the above definition of the norm $N(\cdot)$ of $X_d$ generalizes the definition of a norm in $C\ell_{k}$.
Similarly, the subalgebra $C\ell_{d}^0$ is contained in $Z_d^0$.
%This allows us to consider the group $\Spin{d}$ as being contained in $Y_d^{0}$.
%Similarly, we may write $i(v)\in X_d$ for any $v\in \RR^d$, thinking of $i(v)\in X_d$ as contained in the subspace generated by $\{e_1,\ldots,e_{d}\}$.

We are now ready to define our variant of $\Spin{d}$.
\begin{align}
\Spun{d} = \big\{z\in Z_d^0 :\, N(z) = \BI &\mbox{ and for every } v\in \RR^{d} \text{ there exists } w\in \RR^d \nonumber \\
& \text{ such that } z \left(e_{d+2} i(v) + e_{d+1}\right)\overline{z} = e_{d+2}i(w) + e_{d+1}  \big\}. \label{eq:SpunDef}
\end{align}

We will prove that $\Spun{d}$ is indeed a group and a double cover of $\SE{d}$.
But first  we give a brief intuition for the definition in \eqref{eq:SpunDef}.
We can think of this definition as extending $\Spin{d}$ with the two extra elements $e_{d+1}$ and $e_{d+2}$.
The addition of $e_{d+1}$ leads us to the group $\Spin{d+1}$, which is a double cover of $\SO{d+1}$.
The role of $e_{d+2}$ is to imitate a scaling limit argument as in \cite{Tblog11}.
We think of $e_{d+2}$ as a small $\eps>0$, or as the restriction to a small disc on $\SSS^{d}$.
As $\eps$ approaches zero, this disc behaves more like a flat so $\Spin{d+1}$ becomes more similar to $\SE{d}$.

Note that for every $\gamma \in \Spin{d}$ we have that $\gamma e_{d+1} = e_{d+1}\gamma$, and similarly for $e_{d+2}$.

\begin{theorem} \label{th:SpunGroup}
The set $\Spun{d}$ is a group under the product operation of $X_d$.
Moreover, the inverse of every $x\in \Spun{d}$ is $\overline{x}$.
\end{theorem}
\begin{proof}
We first show that for every $x,y\in \Spun{d}$ we have $xy\in \Spun{d}$.
Indeed, note that
\[ N(xy) = xy\overline{y} \hspace{0.6mm} \overline{x} = xN(y)\overline{x} = x\overline{x} = N(x) = \BI. \]
Moreover, for every $v\in \RR^d$ there exist $u,w\in \RR^d$ such that
\begin{align*}
xy \left(e_{d+2} i(v) + e_{d+1}\right)\overline{xy} &= x (y(e_{d+2} i(v) + e_{d+1})\overline{y})\overline{x} = x (e_{d+2}i(w) + e_{d+1})\overline{x} = e_{d+2}i(u) + e_{d+1}.
\end{align*}

Since the product operation of $X_d$ is clearly associative and $\BI$ is the identity element, it remains to prove that every $x\in \Spun{d}$ has an inverse in $\Spun{d}$.
We will prove that $x\overline{x} = \overline{x}x = \BI$ and that $\overline{x}\in \Spun{d}$.
Fix $x\in \Spun{d}$, and write $x = \gamma_1 + e_{d+1}e_{d+2}\gamma_2$ where $\gamma_1 \in C\ell_{d+1}^0$ and $\gamma_2 \in C\ell_d^1$.
Since $x\overline{x}=N(x)=\BI$, we have
\[ \BI = x\overline{x} = (\gamma_1 + e_{d+1}e_{d+2}\gamma_2) (\overline{\gamma_1} + e_{d+1}e_{d+2}\overline{\gamma_2}) = \gamma_1\overline{\gamma_1} + e_{d+2}(\gamma_1e_{d+1}\overline{\gamma_2}+e_{d+1}\gamma_2\overline{\gamma_1}).\]
By comparing the parts that do not involve $e_{d+2}$ on both sides of the equation, we get $\gamma_1\overline{\gamma_1} = \BI$.
By comparing the parts that contain $e_{d+2}$, we get $\gamma_1e_{d+1}\overline{\gamma_2} = -e_{d+1}\gamma_2\overline{\gamma_1}$.

Since $x\in \Spun{d}$, for every $v\in \RR^d$ there exists $w \in \RR^d$ such that $x(e_{d+2} i(v) + e_{d+1})\overline{x} = e_{d+2}i(w)+e_{d+1}$.
In particular, there exists $w_0 \in \RR^d$ such that $xe_{d+1}\overline{x} = x(e_{d+2}\cdot i(0) + e_{d+1})\overline{x} = e_{d+2}i(w_0) + e_{d+1}$.
Fixing $v,w\in \RR^d$ as defined above and setting $u=w-w_0$ gives
\[ e_{d+2}xi(v) \overline{x} = x(e_{d+2} i(v) + e_{d+1})\overline{x} - xe_{d+1}\overline{x} = e_{d+2}i(w)+e_{d+1} - (e_{d+2}i(w_0) + e_{d+1}) = e_{d+2}i(u).\]

We also have
\[ e_{d+2}xi(v) \overline{x} = e_{d+2}(\gamma_1 + e_{d+1}e_{d+2}\gamma_2)i(v) (\overline{\gamma_1} + e_{d+1}e_{d+2}\overline{\gamma_2}) = e_{d+2}\gamma_1 i(v) \overline{\gamma_1}.\]
Combining the above gives $\gamma_1 i(v) \overline{\gamma_1} = i(u) \in i(\RR^d)$.
That is, $\gamma_1 i(v) \overline{\gamma_1} \in i(\RR^d)$ for every $v \in \RR^d$.

We have
\begin{align}
e_{d+2}i(w_0) + e_{d+1} = xe_{d+1}\overline{x} &= \left(\gamma_1 + e_{d+1}e_{d+2}\gamma_2\right)e_{d+1} \left(\overline{\gamma_1} + e_{d+1}e_{d+2}\overline{\gamma_2}\right) \nonumber \\
&= \gamma_1e_{d+1} \overline{\gamma_1} + e_{d+2}\left(\gamma_1e_{d+1}e_{d+1}\overline{\gamma_2} + e_{d+1}\gamma_2e_{d+1}\overline{\gamma_1} \right). \label{eq:w0GroupLemma}
\end{align}
By again comparing the terms that do not involve $e_{d+2}$ we get $\gamma_1e_{d+1} \overline{\gamma_1} = e_{d+1}$.

Since $\gamma_1 i(v) \overline{\gamma_1} \in i(\RR^d)$ for every $v\in \RR^d$ and $\gamma_1e_{d+1} \overline{\gamma_1} = e_{d+1}$, we get that $\gamma_1 i(v') \overline{\gamma_1} \in i(\RR^{d+1})$ for every $v'\in \RR^{d+1}$.
Combining this with $N(\gamma_1) = \gamma_1\overline{\gamma_1} = \BI$ implies that $\gamma_1 \in \Spin{d+1}$.
In particular, the inverse of $\gamma_1$ is $\overline{\gamma_1}$.
Multiplying the above equation $\gamma_1e_{d+1} \overline{\gamma_1} = e_{d+1}$ by $\gamma_1$ from the right leads to $\gamma_1e_{d+1} = e_{d+1}\gamma_1$.
Since $\gamma_1$ commutes with $e_{d+1}$ we have $\gamma_1 \in C\ell_{d}^0$, which in turn implies $\gamma_1 \in \Spin{d}$.

We next wish to show that $\overline{x}x =\BI$.
Since $x\overline{x}=\BI$, it suffices to prove that $x\overline{x} = \overline{x}x$, or equivalently
\begin{align*}
\gamma_1\overline{\gamma_1} + e_{d+2}\left(\gamma_1e_{d+1}\overline{\gamma_2} + e_{d+1}\gamma_2\overline{\gamma_1}\right) = \overline{\gamma_1}\gamma_1 + e_{d+2}\left(e_{d+1}\overline{\gamma_2}\gamma_1 + \overline{\gamma_1}e_{d+1}\gamma_2\right).
\end{align*}

Since the inverse of $\gamma_1$ is $\overline{\gamma_1}$, we have that $\gamma_1\overline{\gamma_1} = \BI = \overline{\gamma_1}\gamma_1$.
It remains to prove that
\begin{align*}
\gamma_1e_{d+1}\overline{\gamma_2} + e_{d+1}\gamma_2\overline{\gamma_1} = e_{d+1}\overline{\gamma_2}\gamma_1 + \overline{\gamma_1}e_{d+1}\gamma_2.
\end{align*}
Since $e_{d+1}$ commutes with $\gamma_1$ and $\overline{\gamma_1}$, this equation becomes $\gamma_1\overline{\gamma_2} + \gamma_2\overline{\gamma_1} = \overline{\gamma_2}\gamma_1 + \overline{\gamma_1}\gamma_2$.

Above we proved that $\gamma_1e_{d+1}\overline{\gamma_2} = -e_{d+1}\gamma_2\overline{\gamma_1}$.
Since $e_{d+1}$ commutes with $\gamma_1$ and $\overline{\gamma_1}$, we get $\gamma_1\overline{\gamma_2} = -\gamma_2\overline{\gamma_1}$.
Multiplying by $\overline{\gamma_1}$ from the left and by $\gamma_1$ from the right gives $\overline{\gamma_2}\gamma_1 = -\overline{\gamma_1}\gamma_2$.
Combining these two inequalities leads to the required equation $\gamma_1\overline{\gamma_2} + \gamma_2\overline{\gamma_1} = 0\cdot \BI =\overline{\gamma_2}\gamma_1 + \overline{\gamma_1}\gamma_2$.
We conclude that $\overline{x}x =\BI$.
That is, $x^{-1} = \overline{x}$.

To complete the proof of the lemma we need to show that $\overline{x} \in \Spun{d}$.
We already know that $N(\overline{x})=\overline{x}x =\BI$.
It remains to prove that for every $v\in \RR^d$ there exists $w\in \RR^d$ such that $\overline{x} \left(e_{d+2}i(v) + e_{d+1}\right) x = e_{d+2}i(w) + e_{d+1}$.

By considering the coefficients of $e_{d+2}$ in \eqref{eq:w0GroupLemma}, we get
\[i(w_0) = \gamma_1e_{d+1}e_{d+1}\overline{\gamma_2} + e_{d+1}\gamma_2e_{d+1}\overline{\gamma_1} = \gamma_2\overline{\gamma_1} -\gamma_1\overline{\gamma_2}. \]
Multiplying by $\gamma_1$ from the right and by $\overline{\gamma_1}$ from the left gives $\overline{\gamma_1} i(w_0)\gamma_1 = \overline{\gamma_1}\gamma_2 -\overline{\gamma_2}\gamma_1$.
Since $\overline{\gamma_1}\in \Spin{d}$, there exists $w_1\in \RR^d$ such that $\overline{\gamma_1} i(w_0)\gamma_1=i(w_1)$.

We have that
\begin{align*}
\overline{x}e_{d+1}x &=  \left(\overline{\gamma_1} + e_{d+1}e_{d+2}\overline{\gamma_2}\right)e_{d+1} \left(\gamma_1 + e_{d+1}e_{d+2}\gamma_2\right) = e_{d+1} + e_{d+2}\left(\overline{\gamma_2}\gamma_1-\overline{\gamma_1}\gamma_2\right) = e_{d+1} - e_{d+2}i(w_1).
\end{align*}

Since $\overline{\gamma_1}\in \Spin{d}$, for every $v\in \RR^d$ there exists $u\in \RR^d$ such that $\overline{\gamma_1}i(v)\gamma_1 = i(u)$.
By combining the above, we get
\begin{align*}
\overline{x} \left(e_{d+2}i(v) + e_{d+1}\right) x &= e_{d+2}\overline{x}i(v) x + \overline{x} e_{d+1} x \\
&= e_{d+2}\left(\overline{\gamma_1} + e_{d+1}e_{d+2}\overline{\gamma_2}\right)i(v)\left(\gamma_1 + e_{d+1}e_{d+2}\gamma_2\right) + e_{d+1} - e_{d+2}i(w_1) \\
&= e_{d+2}\left(\overline{\gamma_1}i(v)\gamma_1 - i(w_1)\right) + e_{d+1} = e_{d+2}i(u-w_1)+e_{d+1}.
\end{align*}
\end{proof}

Now that we established the $\Spun{d}$ is a group, we start to study its structure.

\begin{lemma} \label{le:SpunDesc}
We have
\emph{$\Spun{d} = \{ \gamma\left(\BI+e_{d+1}e_{d+2}i(v)\right) \, :\, \gamma \in \Spin{d}, \ v\in \RR^d\}$}.
Every element of \emph{$\Spun{d}$} corresponds to a unique pair \emph{$(\gamma,v)\in \Spin{d} \times \RR^d$}.
\end{lemma}
\begin{proof}
For arbitrary $\gamma \in \Spin{d}$ and $v\in \RR^d$, we set $x =  \gamma\left(\BI+e_{d+1}e_{d+2}i(v)\right)$.
Then
\begin{align*}
N(x) =  \gamma\left(\BI+e_{d+1}e_{d+2}i(v)\right) (\overline{\BI+e_{d+1}e_{d+2}i(v)}) \overline{\gamma} &= \gamma\left(\BI+e_{d+1}e_{d+2}i(v)\right) \left(\BI-e_{d+1}e_{d+2}i(v)\right)\overline{\gamma} \\
&=\gamma(\BI - e_{d+1}e_{d+2}i(v)+e_{d+1}e_{d+2}i(v))\overline{\gamma} =\gamma \overline{\gamma}=\BI.
\end{align*}

Since $\gamma \in \Spin{d}$, there exists $w_1\in \RR^d$ such that $\gamma i(v)\overline{\gamma} = i(w_1)$.
For every $u\in \RR^d$ there exists $w_2\in \RR^d$ such that $\gamma i(u)\overline{\gamma} = i(w_2)$, so
\begin{align*}
x(&e_{d+2}i(u)+e_{d+1})\overline{x} = e_{d+2}(xi(u)\overline{x}) + x e_{d+1} \overline{x} \\
&=  e_{d+2}\gamma\left(\BI+e_{d+1}e_{d+2}i(v)\right) i(u)\left(\BI-e_{d+1}e_{d+2}i(v)\right)\overline{\gamma} + \gamma\left(e_{d+1} + e_{d+2}i(v)\right)\left(\BI-e_{d+1}e_{d+2}i(v)\right)\overline{\gamma} \\
&= e_{d+2}\gamma i(u)\overline{\gamma} +\gamma(e_{d+1}+2e_{d+2}i(v))\overline{\gamma} = e_{d+2} i(w_2+2w_1) +e_{d+1} \in \left(e_{d+2} i(\RR^d)+e_{d+1}\right).
\end{align*}
We conclude that  $\{ \gamma \left(\BI+e_{d+1}e_{d+2}i(v)\right)  \, :\, \gamma \in \Spin{d}, v\in \RR^d\} \subseteq \Spun{d}$.

For the other direction, consider an element $x \in \Spun{d}$, and recall from the proof of Theorem \ref{th:SpunGroup} that $x = \gamma_1 + e_{d+1}e_{d+2}\gamma_2$ for some $\gamma_1 \in \Spin{d}$ and $\gamma_2\in C\ell_{d}^1$.
By definition, there exists $w_0\in \RR^d$ such that $xe_{d+1}\overline{x} = x(e_{d+2}i(0) + e_{d+1}\overline{x} = e_{d+2}i(w_0)+e_{d+1}$.
In the proof of Theorem \ref{th:SpunGroup} it is also shown that $i(w_0) = \gamma_2\overline{\gamma_1} -\gamma_1\overline{\gamma_2}$ and that $\gamma_1\overline{\gamma_2} = -\gamma_2\overline{\gamma_1}$.
Together these imply $\gamma_2\overline{\gamma_1} = i(w_0)/2$.
Since $\gamma_1\in \Spin{d}$, it has the inverse $\overline{\gamma_1}$.
Thus, there exists $w_1\in \RR^d$ such that
\[ \gamma_2 = (\gamma_2 \overline{\gamma_1}) \gamma_1=  i(w_0)\gamma_1/2 = \gamma_1\overline{\gamma_1}i(w_0)\gamma_1/2 = \gamma_1i(w_1)/2. \]

We conclude that $x = \gamma_1(\BI + e_{d+1}e_{d+2}i(w_1)/2)$ where $\gamma_1\in \Spin{d}$.
That is, $\Spun{d} \subseteq \{ \gamma \left(\BI+e_{d+1}e_{d+2}i(v)\right) \, :\, \gamma \in \Spin{d}, v\in \RR^d\}$.

Note that $\gamma$ is uniquely determined by $x$, since it is exactly the part of $x$ that does not involve $e_{d+2}$.
Once $\gamma$ is fixed, there is a unique $v\in \RR^d$ that satisfies $\gamma e_{d+1}e_{d+2}i(v) = e_{d+1}e_{d+2}\gamma_2$.
That is, the pair $(\gamma,v)$ is uniquely determined.
\end{proof}

Recall that every transformation of $\SE{d}$ can be seen as a translation followed by a rotation, which is a pair in $\SO{d}\times \RR^d$.
Lemma \ref{le:SpunDesc} states that every element of $\Spun{d}$ corresponds to a unique pair of $\Spin{d}\times \RR^d$.
Since $\Spin{d}$ is a double cover of $\SO{d}$, we are starting to see why $\Spun{d}$ is a double cover of $\SE{d}$.
The following result proves this property, and provides a variant of the homomorphism $\rho: \Spin{d} \to \SO{d}$ defined above.

\begin{theorem} \label{th:SpunToSE}
For every $d$ there exists a surjective group homomorphism \emph{$\rho: \Spun{d} \to \SE{d}$} with $\ker(\rho) = \{-1,1\}$.
That is, \emph{$\Spun{d}$} is a double cover of \emph{$\SE{d}$}.
\end{theorem}
\begin{proof}
Let $a \in \RR^d$ be a fixed point.
By Lemma \ref{le:SpunDesc}, any element $x\in \Spun{d}$ can be written as $\gamma_x\left(\BI + e_{d+1}e_{d+2}i(v_x)\right)$ where $\gamma_x \in \Spin{d}$ and $v_x\in \RR^d$ are uniquely determined.
We set $p_x = \gamma_x i(a+2v_x)\overline{\gamma_x}$ and note that $p_x \in i(\RR^d$).
We have
\begin{equation} \label{eq:xToRho}
\gamma_x + \frac{1}{2}e_{d+1}e_{d+2}\left(p_x\gamma_x - \gamma_x i(a)\right) = \gamma_x + \frac{1}{2}e_{d+1}e_{d+2}\cdot 2 \gamma_xi(v_x) =  \gamma_x (\BI + e_{d+1}e_{d+2}i(v_x)) = x.
\end{equation}

Thus, for any $x\in \Spun{d}$ there exist unique $p_x\in i(\RR^d)$ and $\gamma_x \in \Spin{d}$ such that $x = \gamma_x + \frac{1}{2}e_{d+1}e_{d+2}\left(p_x\gamma_x - \gamma_x i(a)\right)$.
We now rely on this observation to define the map $\rho: \Spun{d} \to \SE{d}$.
For any $v\in \RR^d$, denote the translation of $\RR^d$ by $v$ as $v^+\in \SE{d}$.
Similarly, for any $p\in i(\RR^d)$, we set $p^+=(i^{-1}(p))^+\in \SE{d}$.
As stated in Section \ref{sec:Perlim}, there is a unique $\Gamma_x\in \SO{d}$ that corresponds to $\gamma_x$.
We set
\[ \rho(x) = p_x^+\circ \Gamma_x \circ (-a)^+. \]
Note that $p_x$ and $\Gamma_x$ are uniquely determined by $x$.
Recalling that $a$ is fixed, we conclude that the map $\rho(\cdot)$ is well-defined.

For every $p \in i(\RR^d)$ and $\gamma \in \Spin{d}$, by setting $i(u) = \frac{1}{2} \left(\overline{\gamma}p\gamma -  i(a) \right)\in i(\RR^d)$ we get
\[ \gamma \left(\BI + e_{d+1}e_{d+2}i(u)\right) = \gamma + \frac{1}{2}e_{d+1}e_{d+2}\left(p\gamma - \gamma i(a)\right). \]

Combining this with \eqref{eq:xToRho} and with Lemma \ref{le:SpunDesc} implies that for every $p \in i(\RR^d)$ and $\gamma \in \Spin{d}$ there exists $x\in \Spun{d}$ such that $p=p_x$ and $\gamma=\gamma_x$.
Every transformation $M\in \SE{d}$ can be written as $(M(a))^+ \circ R \circ (-a)^+$ for some transformation $R\in \SO{d}$.
Indeed, note that for any $R\in \SO{d}$ the map $(M(a))^+ \circ R \circ (-a)^+$ takes $a$ to $M(a)$, so we just need to choose the $R$ that rotates the space properly around $a$.
We conclude that $\rho$ is surjective.

For $x,y\in \Spun{d}$, we now consider how the product $xy\in \Spun{d}$ behaves.
Since $\overline{\gamma_y}\in \Spin{d}$, there exists $v_z\in \RR^d$ such that $i(v_z) = \overline{\gamma_y} i(v_x)\gamma_y$.
Then
\begin{align*}
xy &= \gamma_x\left(\BI + e_{d+1}e_{d+2}i(v_x)\right) \gamma_y\left(\BI + e_{d+1}e_{d+2}i(v_y)\right) \\
&= \gamma_x\left(\gamma_y + e_{d+1}e_{d+2}i(v_x)\gamma_y\right) \left(\BI + e_{d+1}e_{d+2}i(v_y)\right) \\
&= \gamma_x\gamma_{y}\left(\BI + e_{d+1}e_{d+2}i(v_z)\right) \left(\BI + e_{d+1}e_{d+2} i(v_y)\right) = \gamma_x\gamma_{y}\left(\BI + e_{d+1}e_{d+2}i(v_z+v_y)\right).
\end{align*}

This implies that $\gamma_{xy} = \gamma_x\gamma_y$ and that $v_{xy} = v_z+v_y$.
This in turn implies that $p_{xy} = \gamma_{xy}i(a + 2v_z+2v_y)\overline{\gamma_{xy}}$.
We are now ready to verify that $\rho$ is a group homomorphism.
Note that the action of $\gamma_{xy}$ is first performing the action of $\gamma_y$ and then the action of $\gamma_x$.
That is, $\Gamma_{xy} = \Gamma_x  \circ \Gamma_y$.
For the same reason we have
\begin{align*}
\rho(x)\rho(y) &=  p_x^+\circ \Gamma_x \circ (-a)^+  \circ p_y^+\circ \Gamma_y \circ (-a)^+= p_x^+\circ \Gamma_x \circ (p_y-i(a))^+ \circ \Gamma_y \circ (-a)^+ \\
&= p_x^+ \circ (\gamma_x(p_y-i(a))\overline{\gamma_x})^+ \circ \Gamma_x  \circ \Gamma_y \circ (-a)^+ \\
&=  (p_x +\gamma_x(p_y-i(a))\overline{\gamma_x})^+ \circ \Gamma_{xy} \circ (-a)^+ \\
&= \left(\gamma_x\left(i(a+2v_x)+\gamma_y i(a+2v_y)\overline{\gamma(y)}-i(a)\right)\overline{\gamma_x}\right)^+ \circ \Gamma_{xy} \circ (-a)^+ \\
&= \left(\gamma_x \gamma_y(2\overline{\gamma_y}i(v_x)\gamma_y+ i(a+2v_y))\overline{\gamma_x\gamma_y}\right)^+ \circ \Gamma_{xy} \circ (-a)^+ \\
&= (\gamma_{xy}i(a +2v_{xy})\overline{\gamma_{xy}})^+ \circ \Gamma_{xy} \circ (-a)^+ = p_{xy}^+ \circ \Gamma_{xy} \circ (-a)^+ = \rho(xy).
\end{align*}

It remains to find the kernel of the homomorphism $\rho$.
Let $I$ be identity element of $\SE{d}$ and let $x\in \Spun{d}$ satisfy $\rho(x)=p_x^+\circ \Gamma_x \circ (-a)^+=I$.
That is, $p_x^+\circ \Gamma_x = a^+$.
The composition of a rotation and a translation cannot be a translation, so $\Gamma_x$ is the identity of $\SO{d}$ and $p_x = i(a)$.
This implies that $\gamma_x \in \{-\BI,\BI\} \subset \Spin{d}$.
Combining the above with \eqref{eq:xToRho} gives
\[ \ker(\rho) = \rho^{-1}(I) = \left\{-\BI + \frac{1}{2}e_{d+1}e_{d+2}\left(-i(a) + i(a)\right), \BI + \frac{1}{2}e_{d+1}e_{d+2}\left(i(a) - i(a)\right)\right\} = \{-\BI,\BI\}. \]
\end{proof}

Let $\tau: \{e_{d+2}i(v) + e_{d+1} :\, v\in \RR^n\} \to \RR^n$ be the map defined by $\tau(e_{d+2}i(v) + e_{d+1}) = v$.
The following lemma studies the behaviour of the homomorphism $\rho$ from Theorem \ref{th:SpunToSE}.

\begin{lemma} \label{le:RhoDef}
For every $w\in \RR^d$ and $x\in \Spun{d}$,
\[ \rho(x)(w) = \tau(x(e_{d+2}i(w)+e_{d+1})\overline{x}). \]
\end{lemma}
\begin{proof}
By Lemma \ref{le:SpunDesc}, every $x\in \Spun{d}$ can be written as $\gamma_x\left(\BI + e_{d+1}e_{d+2}i(v_x)\right)$ for some $\gamma_x \in \Spin{d}$ and $v_x\in \RR^d$.
Recalling that $p_x = \gamma_xi(a + 2v_x)\overline{\gamma_x}$, we have
\begin{align*}
x(e_{d+2}i(w)+e_{d+1})\overline{x} &= \gamma_x\left(\BI + e_{d+1}e_{d+2}i(v_x)\right) (e_{d+2}i(w)+e_{d+1}) \left(\BI - e_{d+1}e_{d+2}i(v_x)\right) \overline{\gamma_x} \\
&= \gamma_x e_{d+2}i(w)\overline{\gamma_x} + \gamma_x\left(\BI + e_{d+1}e_{d+2}i(v_x)\right) e_{d+1} \left(\BI - e_{d+1}e_{d+2}i(v_x)\right) \overline{\gamma_x} \\
&= e_{d+2} \gamma_x i(w)\overline{\gamma_x} + \gamma_x\left(e_{d+1} + e_{d+1}e_{d+2}i(v_x)e_{d+1} - e_{d+1}e_{d+1}e_{d+2}i(v_x)\right) \overline{\gamma_x} \\
&= e_{d+2} \gamma_x (i(w) +2i(v_x))\overline{\gamma_x} + \gamma_x e_{d+1}\overline{\gamma_x} \\
&= e_{d+2} \big(p_x+ \gamma_x (i(w) -i(a))\overline{\gamma_x}\big) + e_{d+1}.
\end{align*}

That is, the operation of $\tau(x(e_{d+2}i(w)+e_{d+1})\overline{x})$ can be seen as first translating $w$ by $-a$, then performing the rotation of $\gamma_x\in \Spin{d}$, and finally translating by $p_x$.
This is exactly the operation $\rho(x) = p_x^+\circ \Gamma_x \circ (-a)^+$.
\end{proof}

For $w\in \RR^d$ and $x\in \Spun{d}$, we write $w^x = \rho(x)(w) = \tau(x(e_{d+2}i(w)+e_{d+1})\overline{x})$.

\subsection{The sets $T_{ap}$}

Given points $a,p\in \RR^d$, we define
\begin{equation} \label{eq:TapDef}
T_{ap} = \left\{x \in \Spun{d} :\ a^x =p \right\}.
\end{equation}
That is, $T_{ap}$ is the set of elements of $\Spun{d}$ that correspond to a proper rigid motion of $\RR^d$ that takes $a$ to $p$.
In this section we study the structure of $T_{ap}$.
We begin by presenting a relatively simple description of this set.

\begin{lemma} \label{le:TapBasicRd}
For any $a,p\in \RR^d$, we have
\[ T_{ap} = \left\{\gamma + \frac{1}{2}e_{d+1}e_{d+2}\left(i(p)\gamma - \gamma i(a)\right) :\, \gamma \in \Spin{d} \right\}. \]
\end{lemma}
\begin{proof}
Let $x\in \left\{\gamma + \frac{1}{2}e_{d+1}e_{d+2}\left(i(p)\gamma - \gamma i(a)\right) :\, \gamma \in \Spin{d} \right\}$.
That is, there exists $\gamma_x \in \Spin{d}$ such that $x = \gamma_x + \frac{1}{2}e_{d+1}e_{d+2}\left(i(p)\gamma_x - \gamma_x i(a)\right)$.
We get that
\begin{align*}
N(x) &= x\overline{x} = \left(\gamma_x + \frac{1}{2}e_{d+1}e_{d+2}\left(i(p)\gamma_x - \gamma_x i(a)\right)\right)\left(\overline{\gamma_x} + \frac{1}{2}e_{d+1}e_{d+2}\left( i(a)\overline{\gamma_x} - \overline{\gamma_x} i(p)\right)\right) \\
&= \gamma_x \overline{\gamma_x} + \frac{1}{2}e_{d+1}e_{d+2}\left(\gamma_x i(a)\overline{\gamma_x} -  i(p) + i(p) - \gamma_x i(a)\overline{\gamma_x}\right) =\BI.
\end{align*}

For every $u\in \RR^d$ we have
\begin{align}
x (&e_{d+2}i(u) + e_{d+1}) \overline{x} \nonumber \\
&= \left(\gamma_x + \frac{1}{2}e_{d+1}e_{d+2}\left(i(p)\gamma_x - \gamma_x i(a)\right)\right) (e_{d+2}i(u) + e_{d+1}) \left(\overline{\gamma_x} + \frac{1}{2}e_{d+1}e_{d+2}\left( i(a)\overline{\gamma_x} - \overline{\gamma_x} i(p)\right)\right) \nonumber \\
&=\gamma_x (e_{d+2}i(u) + e_{d+1}) \overline{\gamma_x} +  \frac{1}{2}e_{d+2} \big(e_{d+1}\left(i(p)\gamma_x - \gamma_x i(a)\right)e_{d+1} \overline{\gamma_x} + \gamma_x e_{d+1} e_{d+1}\left( i(a)\overline{\gamma_x} - \overline{\gamma_x} i(p)\right)\big) \nonumber \\
&= e_{d+2} \gamma_xi(v)\overline{\gamma_x} + e_{d+1} + \frac{1}{2}e_{d+2} \big(\left(i(p) - \gamma_x i(a)\overline{\gamma_x}\right) -  \left( \gamma_xi(a)\overline{\gamma_x} -  i(p)\right)\big)  \nonumber \\
&= e_{d+2} \left(\gamma_xi(v-a)\overline{\gamma_x} +i(p)\right) + e_{d+1} \in \left(e_{d+2} i(\RR^d)+e_{d+1}\right). \label{eq:SpunActionRd}
\end{align}

By combining the above, we get that $x\in \Spun{d}$.
From \eqref{eq:SpunActionRd} we obtain
\begin{align*}
x (e_{d+2}i(a)+e_{d+1}) \overline{x} = e_{d+2} \left(\gamma_xi(a-a)\overline{\gamma_x} +i(p)\right) + e_{d+1} = e_{d+2}i(p)+e_{d+1}.
\end{align*}
Since the action of $x$ takes $a$ to $p$, we have that $x\in T_{ap}$.
This in turn implies
\[ \left\{\gamma + \frac{1}{2}e_{d+1}e_{d+2}\left(i(p)\gamma - \gamma i(a)\right) :\, \gamma \in \Spin{d} \right\}\subseteq T_{ap}. \]

For the other direction, consider $y\in T_{ap} \subset \Spun{d}$.
By Lemma \ref{le:SpunDesc}, there exist $\gamma_y \in \Spin{d}$ and $v_y\in \RR^d$ such that $y= \gamma_y\left(\BI+e_{d+1}e_{d+2}i(v_y)\right)$.
We also know that
\begin{align*}
e_{d+2}i(p) + e_{d+1} &= y(e_{d+2}i(a)+e_{d+1})\overline{y} \\
&= \gamma_y\left(\BI+e_{d+1}e_{d+2}i(v_y)\right) (e_{d+2}i(a)+e_{d+1}) \left(\BI-e_{d+1}e_{d+2}i(v_y)\right)\overline{\gamma_y} \\
&= \gamma_y\left(e_{d+2}i(a)+e_{d+1}+e_{d+1}e_{d+2}i(v_y)e_{d+1}-e_{d+1}e_{d+1}e_{d+2}i(v_y)\right)\overline{\gamma_y} \\
&= e_{d+2}\gamma_y\left(i(a)+2i(v_y)\right)\overline{\gamma_y}+e_{d+1}.
\end{align*}

The above calculation implies that $i(p) = \gamma_y\left(i(a)+2i(v_y)\right)\overline{\gamma_y}$.
After rearranging we get $i(v_y) = (\overline{\gamma_y} i(p) \gamma_y - i(a))/2$.
We thus have
\begin{align*}
y = \gamma_y\left(\BI+e_{d+1}e_{d+2}i(v_y)\right) &= \gamma_y\left(\BI+\frac{1}{2}e_{d+1}e_{d+2}(\overline{\gamma_y} i(p) \gamma_y- i(a))\right) \\
&=  \gamma_y+\frac{1}{2}e_{d+1}e_{d+2}( i(p) \gamma_y - \gamma_y i(a)).
\end{align*}

We conclude that $T_{ap} \subseteq \left\{\gamma + \frac{1}{2}e_{d+1}e_{d+2}\left(i(p)\gamma - \gamma i(a)\right) :\, \gamma \in \Spin{d} \right\}$, which in turn implies that the two sets are identical.
\end{proof}

The following lemma provides a more geometric representation of the sets $T_{ap}$: the intersection of $\Spun{d}$ with the linear subspace.
Let
\begin{equation} \label{eq:FapDefD}
F_{ap} = \left(\BI+\frac{1}{2}e_{d+1}e_{d+2}i(p)\right)C\ell_d^0\left(\BI-\frac{1}{2}e_{d+1}e_{d+2}i(a)\right).
\end{equation}

\begin{lemma} \label{le:TapStructureRd}
For $a,p\in \RR^d$, we have $T_{ap} = F_{ap} \cap \Spun{d}$.
\end{lemma}
\begin{proof}
Let $x\in F_{ap} \cap \Spun{d}$.
Since $x\in F_{ap}$, there exists $\delta\in C\ell_d^0$ such that
\begin{equation} \label{eq:XinM}
x= \left(\BI+\frac{1}{2}e_{d+1}e_{d+2}i(p)\right)\delta\left(\BI-\frac{1}{2}e_{d+1}e_{d+2}i(a)\right) = \delta + \frac{1}{2}e_{d+1}e_{d+2}\left(i(p)\delta - \delta i(a)\right).
\end{equation}

As in the proof of Theorem \ref{th:SpunToSE}, since $x\in \Spun{d}$ there exist $\gamma_x \in \Spin{d}$ and $p_x \in i(\RR^d)$ such that $x = \gamma_x + \frac{1}{2}e_{d+1}e_{d+2}\left(p_x\gamma_x - \gamma_x i(a)\right)$.
Combining this with \eqref{eq:XinM} implies that $\gamma_x = \delta$ and $p_x = i(p)$.
By Lemma \ref{le:TapBasicRd} we get that $x\in T_{ap}$.
We conclude that $F_{ap} \cap \Spun{d} \subseteq T_{ap}$.

For the other direction, consider $x\in T_{ap}$.
By Lemma \ref{le:TapBasicRd}, there exists $\gamma \in \Spin{d}$ such that
\[ x = \gamma + \frac{1}{2}e_{d+1}e_{d+2}\left(i(p)\gamma - \gamma i(a)\right) = \left(\BI+\frac{1}{2}e_{d+1}e_{d+2}i(p)\right)\gamma\left(\BI-\frac{1}{2}e_{d+1}e_{d+2}i(a)\right) \in F_{ap}. \]
That is $T_{ap} \subseteq F_{ap}$.
By definition, we have that $T_{ap} \subset \Spun{d}$.
This implies $T_{ap} \subseteq \Spun{d} \cap F_{ap}$ and completes the proof of the lemma.
\end{proof}

\section{Distinct distances in $\RR^3$} \label{sec:DDR3}

In this section we prove Theorem \ref{th:DDreduction} for the case of $\RR^3$.
The proof is based on the $\Spun{3}$ group that was defined in Section \ref{sec:SpunD}.
We note that $C\ell_3^0$ is isomorphic to $\RR^4$ as a vector space.
Specifically, we consider the basis $\BI, e_1e_2, e_1e_3, e_2e_3$ of $C\ell_3^0$ and write
\[ x = x_1\cdot \BI + x_2e_1e_2 + x_3e_1e_3 + x_4 e_2e_3. \]

\begin{lemma} \label{le:Norm3}
For every $x \in C\ell_3^0$ we have $N(x) = \sum_{j=1}^4 x_j^2 \cdot \BI$. \\
\end{lemma}
\begin{proof}
Using the above notation
\begin{align*}
\overline{x} &= \alpha(t(x_1\cdot \BI + x_2e_1e_2 + x_3e_1e_3 + x_4e_2e_3)) = x_1\cdot \BI - x_2e_1e_2 - x_3e_1e_3 - x_4e_2e_3.
\end{align*}
This immediately implies $N(x) = x\overline{x} = x_1^2+x_2^2+x_3^2+x_4^2$.
\end{proof}

By combining Corollary \ref{co:SpinDefLowD} and Lemma \ref{le:Norm3}, we get that
\begin{equation} \label{eq:Spin3Alt}
\Spin{3} = \left\{\left(x_1,x_2,x_3,x_4\right)\in C\ell_3^0 :\, \sum_{j=1}^4 x_j^2 = 1 \right\}.
\end{equation}

We are now ready to derive our reduction for distinct distances in $\RR^3$.

\begin{theorem} \label{th:R3DD}
The problem of deriving a lower bound on the minimum number of distinct distances spanned by $n$ points in $\RR^3$ can be reduced to the following problem:
\vspace{-2mm}

\begin{quotation}
Let $\flats$ be a set of $n$ distinct 2-flats in $\RR^5$, such that every two flats intersect in at most one point,
every point of $\RR^5$ is contained in $O(\sqrt{n})$ flats of $\flats$, and every hyperplane in $\RR^5$ contains $O(\sqrt{n})$ of these flats.
Find an upper bound on the number of $k$-rich points, for every $2\le k = O(n^{1/3+\eps})$ (for some $\eps>0$).
\end{quotation}
\vspace{-2mm}

Deriving the bound $O\left(\frac{n^{5/3}}{k^{2+\eps}}\right)$ for the number of $k$-rich points would yield the conjectured lower bound of $\Omega(n^{2/3})$ distinct distances.
\end{theorem}
\begin{proof}
Let $\pts$ be a set of $n$ points in $\RR^3$.
Let $D$ denote the number of distinct distances that are spanned by $\pts$, and denote these distances as $\delta_1,\ldots, \delta_D$.
Recalling that $|uv|$ is the distance between the points $u$ and $v$, we set
\begin{equation*} %\label{eq:quadruples}
Q = \left\{(a,b,p,q)\in \pts^4: |ab| = |pq|>0 \right\}.
\end{equation*}
The quadruples of $Q$ are ordered, so $(a,b,p,q)$ and $(b,a,p,q)$ are considered
as two distinct elements of $Q$.
The proof is based on double counting $|Q|$.

For every $j \in \{1,\ldots,D\}$, let $E_j = \{(a,b)\in \pts^2 : |ab|=\delta_j \}$.
Since every ordered pair of distinct points $(a,b)\in \pts^2$ appears in exactly one set $E_j$, we have that $\sum_{j=1}^D |E_j| = n^2-n  > n^2/2$.
The Cauchy-Schwarz inequality implies
\begin{equation} \label{eq:lowerQbi}
|Q| = \sum_{j=1}^D |E_j|^2 \ge \frac{1}{D}\left(\sum_{j=1}^D |E_j| \right)^2 > \frac{n^4}{4D}.
\end{equation}

For $a,b,p,q\in \RR^3$ with $a\neq b$, we have $|ab|=|pq|$ if and only if there exists a proper rigid motion in $\SE{3}$ that takes both $a$ to $p$ and $b$ to $q$.
Thus, for every $(a,p)\in \pts^2$ we set
\[ R_{ap} = \{\gamma \in \SE{3} :\, a^{\gamma} = p\}. \]

To derive an upper bound for $|Q|$ it suffices to bound the number of quadruples $(a,b,p,q)\in \pts^4$ that satisfy $a\neq b$ and $R_{ap}\cap R_{bq} \neq \emptyset$.
Since we wish to work in $\Spun{3}$ rather than in $\SE{3}$, we recall the following definition from \eqref{eq:TapDef}.
\[ T_{ap} = \{ x \in \Spun{3} :\, a^x = p \} = \rho^{-1}(R_{ap}). \]

Recall from Theorem \ref{th:SpunToSE} that the homomorphism $\rho$ is surjective with kernel $\{\BI,-\BI\}$.
That is, for every point of $R_{ap}\cap R_{bq}$ there are two corresponding points in $T_{ap}\cap T_{bq}$.
It thus suffices to bound the number of quadruples $(a,b,p,q)\in \pts^4$ that satisfy $a\neq b$ and $T_{ap}\cap T_{bq} \neq \emptyset$.

Before getting to the more technical details of the proof, we provide a brief sketch of the rest of the proof.
We will show that $\Spun{3}$ can be embedded in $\RR^8$ as a well-behaved six-dimensional variety (see Lemma \ref{le:Spun3}).
Under this embedding, each set $T_{ap}$ is a three-dimensional variety that corresponds to an intersection of the $\Spun{3}$ variety with a four-dimensional linear subspace.
We project the $\Spun{3}$ variety in $\RR^8$ from the origin onto the hyperplane defined by $x_1=1$, and then perform a standard projection by removing the coordinates $x_1$ and $x_8$.

Combining the above projections gives a map that is a bijection between most of the $\Spun{3}$ variety and $\RR^6$.
This map takes each set $T_{ap}$ to a 3-flat in $\RR^6$, and every two such 3-flats are either disjoint or intersect in a line.
Since the map is a bijection only after removing a small part of $\Spun{3}$, we get that a quadruple $(a,p,b,q)$ is in $Q$ if and only if the two corresponding 3-flats in $\RR^6$ are contained in a common hyperplane.
By performing a generic projective transformation and then intersecting the 3-flats with a hyperplane, we obtain an incidence problem between points and 2-flats in $\RR^5$.

\parag{From $\Spun{3}$ to $\RR^6$.}
Recall that $\Spun{3}$ is contained in the eight-dimensional subspace $Z_3^0\subset X_3$ generated by $\BI,e_1e_2,e_1e_3,e_2e_3,e_1e_4e_5, e_2e_4e_5, e_3e_4e_5, e_1e_2e_3e_4e_5$.
We consider $Z_3^0$ as $\RR^8$ by mapping these basis elements to the standard basis vectors of $\RR^8$.
That is, we write $x=x_1\cdot \BI +x_2e_1e_2 + x_3e_1e_3 + x_4e_2e_3+ x_5e_1e_4e_5 +x_6e_2e_4e_5 + x_7 e_3e_4e_5 +x_8e_1e_2e_3e_4e_5$ as the point $(x_1,x_2,\ldots,x_8)\in \RR^8$.
With this notation, we study the behavior of $\Spun{3}$ as a set in $\RR^8$.
Set
\[ G = \left\{x\in \RR^8 :\, x_1x_8 - x_2x_7 + x_3x_6 - x_4x_5 =0 \right\} \quad \text{ and } \quad \cyl = \left\{x\in \RR^8 :\, x_1^2 +x_2^2+x_3^2+x_4^2 = 1\right\}. \]

\begin{lemma} \label{le:Spun3}
$\Spun{3} = G \cap \cyl$.
\end{lemma}
\begin{proof}
For every $x\in Z_3^0$ we have
\begin{align*}
\overline{x} &= \alpha(t(x_1\cdot \BI + x_2 e_1e_2 + x_3e_1e_3 +x_4e_2e_3 +x_5e_1e_4e_5 +x_6e_2e_4e_5+ x_7e_3e_4e_5 +x_8 e_1e_2e_3e_4e_5)) \\
&= x_1\cdot \BI - x_2e_1e_2 - x_3e_1e_3 - x_4e_2e_3 - x_5e_1e_4e_5 - x_6e_2e_4e_5 - x_7e_3e_4e_5 + x_8e_1e_2e_3e_4e_5,
\end{align*}
and thus
\begin{equation} \label{eq:SpunNorm}
N(x) = x\overline{x} = \left(x_1^2+x_2^2+x_3^2+x_4^2\right) \BI+ 2\left( x_1x_8 - x_2x_7 + x_3x_6-x_4x_5\right)e_1e_2e_3e_4e_5.
\end{equation}
That is, $N(x)=\BI$ if and only if $x \in \cyl \cap G$.
Combining this with \eqref{eq:SpunDef} implies that $\Spun{3} \subseteq \cyl \cap G$.

For the other direction, consider $x\in \cyl\cap G$.
By \eqref{eq:SpunNorm} we have that $N(x) = \BI$.
Note that we can write $x = \gamma_1 + e_4e_5 \gamma_2$ from some $\gamma_1 \in C\ell_{3}^0$ and $\gamma_2 \in C\ell_{3}^1$.
We then get
\[ \BI = N(x) = (\gamma_1 + e_4e_5 \gamma_2)(\overline{\gamma_1} + e_4e_5 \overline{\gamma_2}) = \gamma_1\overline{\gamma_1} + e_4e_5(\gamma_1\overline{\gamma_2} + \gamma_2\overline{\gamma_1}). \]
This implies that $N(\gamma_1)= \gamma_1\overline{\gamma_1} =\BI$ and $\gamma_1\overline{\gamma_2} = - \gamma_2\overline{\gamma_1}$.
From \eqref{eq:Spin3Alt} we get that $\gamma_1 \in \Spin{3}$.

Since $\gamma_2\overline{\gamma_1} \in C\ell_{3}^1$, there exist $u\in \RR^3$ and $\lambda \in \RR$ such that $\gamma_2\overline{\gamma_1} = i(u) +\lambda e_1e_2e_3$.
Since $\overline{e_1e_2e_3} = e_1e_2e_3$, we have $\overline{\gamma_2\overline{\gamma_1}} = -i(u) +\lambda e_1e_2e_3$.
On the other hand, we have
\[ \overline{\gamma_2\overline{\gamma_1}} = \gamma_1\overline{\gamma_2} = -\gamma_2\overline{\gamma_1} = -i(u)-\lambda e_1e_2e_3. \]
Thus, it must be that $\lambda = 0$.
This in turn implies that $\gamma_2\overline{\gamma_1} \in i(\RR^3)$ and $\gamma_1\overline{\gamma_2} = - \gamma_2\overline{\gamma_1}\in i(\RR^3)$.

For every $v\in \RR^3$, we have
\begin{align*}
x\left(e_5i(v) +e_4\right)\overline{x} &= (\gamma_1 + e_4e_5 \gamma_2)\left(e_5i(v) +e_4\right)(\overline{\gamma_1} + e_4e_5 \overline{\gamma_2})\\
&= \gamma_1 e_5i(v) \overline{\gamma_1} + (\gamma_1 + e_4e_5 \gamma_2)e_4(\overline{\gamma_1} + e_4e_5 \overline{\gamma_2})\\
&= e_5\gamma_1 i(v) \overline{\gamma_1} + \gamma_1e_4\overline{\gamma_1} + \gamma_1e_4e_4e_5 \overline{\gamma_2}  + e_4e_5 \gamma_2e_4\overline{\gamma_1} \\
&= e_5\left(\gamma_1i(v)\overline{\gamma_1} + \gamma_2\overline{\gamma_1} - \gamma_1\overline{\gamma_2}\right) + e_4.
\end{align*}
By the above, $\gamma_1i(v)\overline{\gamma_1} + \gamma_2\overline{\gamma_1} - \gamma_1\overline{\gamma_2}\in i(\RR^3)$.
From the definition in \eqref{eq:SpunDef}, we conclude that $x\in \Spun{3}$.
That is, $\cyl\cap G \subseteq \Spun{3}$, which in turn implies $S\cap G = \Spun{3}$.
\end{proof}

The proof of Lemma \ref{le:Spun3} also implies that $\Spun{3} = \left\{ x \in Z_3^0 :\ N(x) = \BI \right\}$.
We will not rely on this observation.

We now perform a \emph{gnomonic projection}\footnote{Recall that in a gnomonic projection we project the sphere $\SSS^2$ onto a tangent plane, by shooting rays from the center of $\SSS^2$ onto the plane.}, although with the cylindrical hypersurface $\cyl$ rather than a sphere.
Let $\pi_8: \RR^8 \to \RR^7$ be the projection defined by $\pi_8(x_1,x_2,...,x_8) = (x_2,...,x_8)$.
Let $H_0$ denote the hyperplane in $\RR^8$ defined by $x_1=0$ and let $H_1$ denote the hyperplane defined by $x_1=1$.
For each $x\in \RR^8\setminus H_0$ there exists a unique $\lambda_x \in \RR$ such that the $x_1$-coordinate of $\lambda_x x$ is 1.
We define $\pi: \RR^8\setminus H_0 \to \RR^7$ as $\pi(x) = \pi_8(\lambda_x x)$.
That is, $\pi$ projects $x$ from the origin onto $H_1$ and then removes the first coordinate of the resulting point.

For points $a,p\in \pts$, let $F_{ap}$ be defined as in \eqref{eq:FapDefD}.
Note that $F_{ap} \subset Z_3^0$ is a four-dimensional subspace of $\RR^8$.
Since $F_{ap}$ is ruled by lines incident to the origin, we have that $\pi(F_{ap} \setminus H_0) = \pi_8(F_{ap} \cap H_1)$ and $F_{ap} \not\subseteq H_1$.
In the definition \eqref{eq:FapDefD}, by taking an element of $C\ell_3^0$ with a constant term $1\cdot \BI$ we get that $F_{ap} \cap H_1\neq \emptyset$.
Since $F_{ap}$ is a 4-flat and $H_1$ is a hyperplane that intersects $F_{ap}$ without containing it, the intersection $F_{ap} \cap H_1$ is a 3-flat.
Since the restriction of $\pi_8$ to $H_1$ is linear and injective, we get that $\pi(F_{ap} \setminus H_0)$ is a 3-flat in $\RR^7$.

Note that $\cyl$ is a cylindrical hypersurface, and let $\cyl_+$ be set of points of $\cyl$ with a positive $x_1$-coordinate.
By Lemmas \ref{le:TapStructureRd} and \ref{le:Spun3} we have $T_{ap} = F_{ap} \cap \cyl \cap G$, which implies $\pi(T_{ap}\cap \cyl_+) \subseteq \pi(F_{ap}\setminus H_0)$.
By \eqref{eq:Spin3Alt} we have that $T_{00} = \Spin{3}=F_{00}\cap \cyl$.
This implies
\begin{align*}
T_{ap} &= \left(\BI+\frac{1}{2}e_4e_5i(p)\right)T_{00}\left(\BI-\frac{1}{2}e_4e_5i(a)\right)\\
&= \left(\left(\BI+\frac{1}{2}e_4e_5i(p)\right)F_{00}\left(\BI-\frac{1}{2}e_4e_5i(a)\right) \right)\bigcap \left(\left(\BI+\frac{1}{2}e_4e_5i(p)\right)\mathcal{C}\left(\BI-\frac{1}{2}e_4e_5i(a)\right)\right)\\
&\hspace{135mm}= F_{ap} \cap \mathcal{C}.
\end{align*}
Thus, for every $v\in H_1\cap F_{ap}$ there exists $r\in \RR$ such that $rv \in \cyl_+$.
That is, $\pi(F_{ap} \setminus H_0) \subseteq \pi(T_{ap} \cap \cyl_+)$, which in turn implies $\pi(T_{ap}\cap \cyl_+) = \pi(F_{ap}\setminus H_0)$.
We conclude that $\pi$ maps each set $T_{ap}\cap \cyl_+$ onto a 3-flat in $\RR^7$.

Let $g:\RR^8 \to \RR$ be the map defined by $g(x_1,...,x_8) = x_1x_8 - x_2x_7 + x_3x_6-x_4x_5$.
Note that $G = g^{-1}(0)$.
For every $v\in G$ and $r\in \RR$ we have $g(r v) = r^2 g(v)$, so $rv\in G$.
That is, $G$ is ruled by lines incident to the origin, which implies that $\pi(G\setminus H_0) = \pi_8(G\cap H_1)$.
Let $g_7: \R^7 \to \RR$ be the map defined by $g_7(x_2,...,x_8) = x_8 - x_2x_7 + x_3x_6-x_4x_5$ and note that $\pi(G\setminus H_0) = g_7^{-1}(0)$.
We set $G_7 = g_7^{-1}(0)\subset \RR^7$.
Since each $T_{ap} \cap \cyl_+ \subset \Spun{3} \subset G$, every 3-flat of the form $\pi(T_{ap} \cap \cyl_+)$ is contained in $G_7$.
Given $(x_2,\ldots,x_8)\in \RR^7$, let $x=(1,x_2,\ldots,x_8)$.
Then there exists $r\in \RR$ such that $y=r x$ is the unique point on $\cyl_+$ that satisfies $\pi(y) = (x_2,\ldots,x_8)$.
That is, the restriction of $\pi$ to $\cyl_+$ is a bijection between $\cyl_+$ and $\RR^7$.
Moreover, $\pi$ maps $G$ to $G_7$ (it is not injective in this domain) and maps each $T_{ap}\cap \cyl_+$ to a 3-flat contained in $G_7$.

Let $\pi_7: \RR^7\to \RR^6$ be the projection that is defined by $\pi_7(x_2,...,x_7,x_8) = (x_2,...,x_7)$.
Since $g_7(x_2,...,x_7,x_8) = g_7(x_2,...,x_7,x_8')$ implies $x_8 = x_8'$, the restriction of $\pi_7$ to $G_7$ is injective.
Since $\pi_7$ is linear and every 3-flat of the form $\pi(T_{ap} \cap \cyl_{+})$ is contained in $G_7$, we get that $\pi_7(\pi(T_{ap} \cap \cyl_{+}))$ is a 3-flat in $\RR^6$.
Furthermore, since both the restriction of $\pi_7$ to $G_7$ and the restriction of $\pi$ to $\cyl_{+}$ are bijections, the restriction of $\pi_7\circ \pi$ to $\cyl_{+}\cap G$ is injective.
For every $v\in G\setminus H_0$ there exists $r\in \RR$ such that $rv\in \cyl_{+} \cap G$.
That is, $\eta = \pi_7\circ \pi$ is a bijection from $G\cap \cyl_+$ to $\RR^6$.

\parag{Studying intersections of 3-flats.}
Recall from Lemma \ref{le:TapStructureRd} that $T_{ap} = F_{ap} \cap \Spun{3}$.
To study intersections of the 3-flats in $\RR^6$, we first study the intersections $F_{ap}\cap F_{bq}$.

\begin{lemma} \label{le:EmptyIntTCond}
We have that $T_{ap}\cap T_{bq} = \emptyset$ if and only if $F_{ap}\cap F_{bq} = \{0\}$.
\end{lemma}
\begin{proof}
By Lemma \ref{le:TapStructureRd}, $T_{ap} = F_{ap} \cap \Spun{3}$ and $T_{bq} = F_{bq} \cap \Spun{3}$.
Thus, $F_{ap}\cap F_{bq} = \{0\}$ immediately implies $T_{ap}\cap T_{bq} = \emptyset$.

Next, we assume that $F_{ap}\cap F_{bq} \neq \{0\}$.
For any $v\in \RR^3$ we have $\left(\BI+\frac{1}{2}e_{4}e_{5}i(v)\right) \left(\BI-\frac{1}{2}e_{4}e_{5}i(v)\right) = \BI$.
Combining this with the definition of $F_{ap}$ gives
\begin{align*}
F_{ap}\cap F_{bq} &= \left(\BI+\frac{1}{2}e_{4}e_{5}i(p)\right) \left(\BI-\frac{1}{2}e_{4}e_{5}i(p)\right)\left(F_{ap}\cap F_{bq}\right)\left(\BI+\frac{1}{2}e_{4}e_{5}i(a)\right) \left(\BI-\frac{1}{2}e_{4}e_{5}i(a)\right)\\
&= \left(\BI+\frac{1}{2}e_{4}e_{5}i(p)\right)\left(C\ell_3^0\cap F_{(b-a)(q-p)}\right)\left(\BI-\frac{1}{2}e_{4}e_{5}i(a)\right).
\end{align*}

Since $F_{ap}\cap F_{bq} \neq \{0\}$, we have that $C\ell_3^0\cap F_{(b-a)(q-p)} \neq \{0\}$.
That is, there exist $\gamma,\delta\in C\ell_3^0$ such that
\[ \gamma = \left(\BI+\frac{1}{2}e_{4}e_{5}i(q-p)\right)\delta\left(\BI-\frac{1}{2}e_{4}e_{5}i(b-a)\right).\]
By comparing the terms that do not depend on $e_5$, we get $\gamma=\delta$.
By then comparing the coefficient of $e_5$ on each side, we get  $i(q-p)\gamma = \gamma i(b-a)$.

Note that for any $x \in C\ell_3^0$, the coefficient of $\BI$ in $x\overline{x}$ is equal to the coefficient of $\BI$ in $\overline{x}x$ (this coefficient equals $\|x\|^2$ when thinking of $x$ as a point in $\RR^4$, as in the beginning of this section).
Recall that for any $s\in \RR^3$ we have $i(s)\overline{i(s)} = \overline{i(s)}i(s) = \|s\|^2 \cdot \BI$.
By taking $x = i(q-p)\gamma = \gamma i(b-a)$, we get that the coefficient of $\BI$ in $\overline{\gamma}\overline{i(q-p)} i(q-p)\gamma = \|q-p\|^2\overline{\gamma}\gamma$ is equal to the coefficient of $\BI$ in $\gamma i(b-a) \overline{i(b-a)}\overline{\gamma} = \|b-a\|^2\gamma\overline{\gamma}$.
Since the coefficients of $\BI$ in $\overline{\gamma} \gamma$ and $\gamma \overline{\gamma}$ are equal, it follows that $\|b-a\| = \|q-p\|$.
Since the vectors $b-a,q-p\in \RR^3$ have the same length, there exists a rotation $\beta \in \Spin{3}$ such that $\beta i(b-a)\beta^{-1} = i(q-p)$.
By Lemma \ref{le:TapBasicRd}, we have
\[ \beta_{ap} = \left(\BI+\frac{1}{2}e_{4}e_{5}i(p)\right) \beta \left(\BI-\frac{1}{2}e_{4}e_{5}i(a)\right) = \beta +\frac{1}{2}e_{4}e_{5}(i(p)\beta-\beta i(a))\in T_{ap}. \]

To prove that $T_{ap}\cap T_{bq}\neq \emptyset$, we show that $\beta_{ap}\in T_{bq}$.
Since $\beta_{ap}\in T_{ap}$, we have that $\beta_{ap}\in \Spun{3}$.
It remains to prove that $\beta_{ap}$ takes $b$ to $q$.
Indeed, recalling that $\beta$ takes $b-a$ to $q-p$ gives
\begin{align*}
\beta_{ap}(e_{5}i(b)+e_{4})\overline{\beta_{ap}} &= \left(\beta +\frac{1}{2}e_{4}e_{5}(i(p)\beta-\beta i(a))\right) (e_{5}i(b)+e_{4}) \left(\overline{\beta} +\frac{1}{2}e_{4}e_{5}(i(a)\overline{\beta}-\overline{\beta}i(p))\right) \\
&= \beta (e_{5}i(b)+e_{4}) \overline{\beta} + \frac{1}{2}e_{5} \left(\beta e_{4}e_{4}(i(a)\overline{\beta}-\overline{\beta}i(p)) + e_{4}(i(p)\beta-\beta i(a))e_{4}\overline{\beta}\right) \\
&= \beta (e_{5}i(b)) \overline{\beta} +e_{4} + \frac{1}{2}e_{5} \left(-(\beta i(a)\overline{\beta}-i(p)) + (i(p)-\beta i(a)\overline{\beta})\right) \\
&=  e_{5}(\beta i(b-a)\overline{\beta})+e_{4}  + e_{5}i(p) = e_{5}i(q-p)+e_{4}  + e_{5}i(p) = e_{5}i(q)+e_{4}.
\end{align*}
\end{proof}

We next study the case where $F_{ap} \cap F_{bq} \neq \{0\}$.

\begin{lemma} \label{le:FapIntersectionChar}
If $F_{ap} \neq F_{bq}$ and $F_{ap} \cap F_{bq} \neq \{0\}$, then
\begin{align*}
F_{ap} \cap F_{bq} = \left(\BI + \frac{1}{2}e_{4}e_{5}i(p)\right)\beta \cdot C\ell_{2}^0 \cdot \alpha \left(\BI - \frac{1}{2}e_{4}e_{5}i(a)\right),
\end{align*}
for any $\alpha,\beta \in \Spin{3}$ that satisfy $\alpha \frac{i(b-a)}{\|b-a\|}\alpha^{-1} = e_{3}$ and $\beta e_{3} \beta^{-1} = \frac{i(q-p)}{\|q-p\|}$.
\end{lemma}
\begin{proof}
By the assumptions and Lemma \ref{le:EmptyIntTCond}, we have that $a\neq b$ and $p\neq q$, so $\|b-a\|$ and $\|q-p\|$ are nonzero.
Thus, the definitions of $\alpha$ and $\beta$ are valid.
Let
\[ N_{ap} = \left(\BI + \frac{1}{2}e_{4}e_{5}i(p)\right)\beta \cdot C\ell_{2}^0 \cdot \alpha \left(\BI - \frac{1}{2}e_{4}e_{5}i(a)\right). \]

Since $\alpha,\beta \in C\ell_3^0$, we have $\beta \cdot C\ell_{2}^0 \cdot \alpha \subset C\ell_3^0$ so $N_{ap}\subseteq F_{ap}$.
We note that
\begin{align}
\alpha \left(\BI - \frac{1}{2}e_{4}e_{5}i(a)\right) &= \alpha \left(\BI - \frac{1}{2}e_{4}e_{5}i(a - b + b)\right) = \alpha \left(\BI - \frac{1}{2}e_{4}e_{5} i(a-b) \right)\left(\BI - \frac{1}{2}e_{4}e_{5}i(b)\right) \nonumber \\
&= \left(\alpha - \frac{1}{2}e_{4}e_{5}\alpha i(a-b)\alpha^{-1} \alpha\right)\left(\BI - \frac{1}{2}e_{4}e_{5}i(b)\right) \nonumber \\
&= \left(\BI + \|b-a\|\frac{1}{2}e_{4}e_{5}e_{3}\right)\alpha\left(\BI - \frac{1}{2}e_{4}e_{5}i(b)\right). \label{eq:FintAux1}
\end{align}
Similarly,
\begin{align}
\left(\BI + \frac{1}{2}e_{4}e_{5}i(p)\right) \beta &= \left(\BI + \frac{1}{2}e_{4}e_{5}i(p - q + q)\right) \beta = \left(\BI + \frac{1}{2}e_{4}e_{5}i(q)\right)\left(\BI + \frac{1}{2}e_{4}e_{5}i(p - q)\right) \beta \nonumber \\
&= \left(\BI + \frac{1}{2}e_{4}e_{5}i(q)\right)\beta \left(\BI + \frac{1}{2}e_{4}e_{5}\beta^{-1}i(p - q)\beta\right) \nonumber \\
&= \left(\BI + \frac{1}{2}e_{4}e_{5}i(q)\right)\beta \left(\BI - \|q-p\|\frac{1}{2}e_{4}e_{5}e_{3}\right). \label{eq:FintAux2}
\end{align}

Combining \eqref{eq:FintAux1} and \eqref{eq:FintAux2} gives
\begin{align*}
N_{ap} &= \left(\BI + \frac{1}{2}e_{4}e_{5}i(p)\right)\beta \cdot C\ell_{2}^0\cdot \alpha \left(\BI - \frac{1}{2}e_{4}e_{5}i(a)\right) \\
&= \left(\BI + \frac{1}{2}e_{4}e_{5}i(q)\right)\beta \left(\BI - \|q-p\|\frac{1}{2}e_{4}e_{5}e_{3}\right)
\cdot C\ell_{2}^0 \cdot \left(\BI + \|b-a\|\frac{1}{2}e_{4}e_{5}e_{3}\right)\alpha\left(\BI - \frac{1}{2}e_{4}e_{5}i(b)\right).
\end{align*}

By Lemma \ref{le:EmptyIntTCond}, the assumption $F_{ap} \cap F_{bq} \neq \{0\}$ implies $T_{ap} \cap T_{bq} \neq \emptyset$.
That is, there exists a rigid motion of $\SE{3}$ that takes both $a$ to $p$ and $b$ to $q$, which in turn implies that $\|b-a\| = \|q-p\|$.
Thus, for any $\gamma \in C\ell_{2}^0$ we have $ \left(\BI - \|q-p\|\frac{1}{2}e_{4}e_{5}e_{3}\right) \gamma \left(\BI + \|b-a\|\frac{1}{2}e_{4}e_{5}e_{3}\right) = \gamma$.
Combining this with the calculation above yields
\begin{align*}
N_{ap} = \left(\BI + \frac{1}{2}e_{4}e_{5}i(q)\right)\beta \cdot C\ell_{2}^0 \cdot \alpha \left(\BI - \frac{1}{2}e_{4}e_{5}i(b)\right) \subseteq F_{bq}.
\end{align*}

We conclude that $N_{ap} \subseteq F_{ap}\cap F_{bq}$.
To prove the other direction, consider $x\in F_{ap}\cap F_{bq}$.
By definition, there exist $\gamma, \gamma' \in C\ell_3^0$ such that
\begin{equation} \label{eq:XtwoDefs}
x =  \left(\BI + \frac{1}{2}e_{4}e_{5}i(p)\right)\gamma \left(\BI - \frac{1}{2}e_{4}e_{5}i(a)\right) = \left(\BI + \frac{1}{2}e_{4}e_{5}i(q)\right)\gamma' \left(\BI - \frac{1}{2}e_{4}e_{5}i(b)\right).
\end{equation}

The part of $x$ that does not involve $e_5$ needs to be identical in both definitions, so $\gamma=\gamma'$.
The part of $x$ that does involve $e_5$ also needs to be identical in both definitions, so $i(p)\gamma-\gamma i(a) = i(q)\gamma - \gamma i(b)$, or equivalently $\gamma i(b-a) = i(q-p)\gamma$.
This implies that
\[ \beta^{-1}\gamma \alpha^{-1} \alpha i(b-a)\alpha^{-1} = \beta^{-1}i(q-p)\beta \beta^{-1}\gamma \alpha^{-1}, \]
which in turn implies $\beta^{-1}\gamma \alpha^{-1}e_3 = e_3 \beta^{-1}\gamma \alpha^{-1}$.
Since $e_3$ commutes with $\beta^{-1}\gamma \alpha^{-1}$, we get that $\beta^{-1}\gamma \alpha^{-1} \in C\ell_{2}^0$.
That is, $\gamma \in \beta \cdot C\ell_{2}^0 \cdot \alpha$.
By combining this with the first equality of \eqref{eq:XtwoDefs}, we conclude that $x\in N_{ap}$ and thus that $F_{ap}\cap F_{bq} \subseteq N_{ap}$.
\end{proof}

Let $L_{ap} = \eta(T_{ap}\setminus H_0)$ be the 3-flat in $\RR^6$ that corresponds to $T_{ap}$.
Given points $a,p,b,q\in \RR^3$, we now study the intersection $L_{ap}\cap L_{bq}$.
Let $L_{apbq}=\eta\left(\langle F_{ap},F_{bq}\rangle\setminus H_0\right)$.
By comparing the definitions of $F_{ap}$ and $L_{ap}$, we note that $L_{ap}\cup L_{bq}\subset L_{apbq}$.

Note that the map $\eta(x)$ is well-defined for every point $x\in \RR^8\setminus H_0$.
Additionally, when we restrict the domain of $\eta$ to $H_1$ it becomes a linear map.
Let $\eta': \RR^8 \to \RR^6$ be the standard linear projection satisfying $\eta'(x_1,x_2,...,x_8) = (x_2,...,x_7)$.
We think of $\eta'$ as a linear extension of the restricted $\eta$ to $\RR^8$.
Denote by $\langle F_{ap},F_{bq}\rangle$ the linear subspace that is spanned by $F_{ap}$ and $F_{bq}$.

\begin{lemma} \label{le:LapbqDim}
If $T_{ap}\cap T_{bq} \not\subseteq H_0$ and $T_{ap} \neq T_{bq}$, then $L_{ap}\cap L_{bq}$ is a line.
\end{lemma}
\begin{proof}
From $T_{ap}\cap T_{bq} \not\subseteq H_0$ we have that $L_{ap}\cap L_{bq} \neq \emptyset$.
Since $L_{ap}$ and $L_{bq}$ are distinct 3-flats in $\RR^6$, their intersection is a flat of dimension between zero and two.
If $\dim\left(L_{ap}\cap L_{bq}\right) = 2$ then $\dim\left(F_{ap}\cap F_{bq} \cap H_1\right) = 2$, which in turn implies $\dim\left(F_{ap}\cap F_{bq}\right) = 3$.
This contradicts Lemma \ref{le:FapIntersectionChar} which states that $\dim\left(F_{ap}\cap F_{bq}\right) = 2$.
Thus, it remains to prove that $L_{ap}\cap L_{bq}$ is not a single point.

For any $v\in \langle F_{ap}, F_{bq} \rangle \cap H_1$, we have
\[ L_{apbq} = \eta(\langle F_{ap}, F_{bq} \rangle \setminus H_0) = \eta'(\langle F_{ap}, F_{bq} \rangle \cap H_1) = \eta'(\langle F_{ap}, F_{bq} \rangle \cap H_0) + \eta'(v). \]
This implies that
\begin{equation} \label{eq:LapbqDim}
\dim L_{apbq} = \dim \left(\langle F_{ap}, F_{bq} \rangle\cap H_0\right) - \dim \left(\langle F_{ap}, F_{bq} \rangle \cap H_0\cap \ker (\eta')\right).
\end{equation}

By Lemma \ref{le:FapIntersectionChar}, $\dim (F_{ap} \cap F_{bq}) = \dim C\ell_2^0= 2$.
Since $\dim F_{ap} = \dim F_{bq} = \dim C\ell_3^0= 4$, we have $\dim \left(\langle F_{ap}, F_{bq} \rangle\cap H_0\right) = 4+4-2-1=5$ (by definition both $F_{ap}$ and $F_{bq}$ intersect $H_0$ but are not contained in it).
Combining this with \eqref{eq:LapbqDim} leads to $\dim L_{apbq} \le 5$.
This completes the proof, since the intersection of two 3-flats in a 5-dimensional space cannot be a single point.
\end{proof}

Next, we study what happens to $L_{ap}$ and $L_{bq}$ when $T_{ap}\cap T_{bq}=\emptyset$.

\begin{lemma} \label{le:ContainingFlat}
For any $a,p,b,q\in \RR^3$, any flat in $\RR^6$ that contains $L_{ap}$ and $L_{bq}$ also contains $L_{apbq}$.
\end{lemma}
\begin{proof}
Let $W$ be a flat that contains $L_{ap}$ and $L_{bq}$.
Then there exists a linear subspace $V \subseteq \RR^6$ such that for any $w \in W$ we have $W = w + V$.
Recall that $F_{ap}\cap H_1 \neq \emptyset$.
For any $x\in \langle F_{ap},F_{bq}\rangle\cap H_1$,
\[ L_{apbq} = \eta\left(\langle F_{ap},F_{bq}\rangle\setminus H_0\right) = \eta\left(\langle F_{ap},F_{bq}\rangle\cap H_1\right) = \eta'\left(\langle F_{ap},F_{bq}\rangle\cap H_1\right) = \eta'\left(\langle F_{ap},F_{bq}\rangle\cap H_0\right) + \eta(x). \]

For $x\in F_{ap} \cap H_1$ we have that $W = \eta(x)+V$ and $L_{ap} = \eta'(x+F_{ap}\cap H_0)= \eta(x) + \eta'(F_{ap}\cap H_0)$.
Combining this with $L_{ap} \subseteq W$ gives $\eta'(F_{ap}\cap H_0) \subseteq V$.
Similarly, by taking $y\in F_{bq}\cap H_1$ we get $W = \eta(y)+V$, which in turn implies $\eta'(F_{bq}\cap H_0) \subseteq V$.
Combining the above yields $\eta'\left(\langle F_{ap},F_{bq}\rangle\cap H_0\right)\subseteq V$.
We conclude that $L_{apbq}\subseteq W$, as desired.
\end{proof}

\begin{corollary} \label{co:EmptyTIntersection}
If $T_{ap}\cap T_{bq} = \emptyset$ then no hyperplane contains both $L_{ap}$ and $L_{bq}$.
\end{corollary}
\begin{proof}
Lemma \ref{le:ContainingFlat} implies that $L_{apbq}$ is the smallest flat that contains $L_{ap}\cup L_{bq}$.
By Lemma \ref{le:EmptyIntTCond}, the assumption $T_{ap}\cap T_{bq} = \emptyset$ implies that $F_{ap} \cap F_{bq} = \{0\}$.
Since $F_{ap}$ and $F_{bq}$ are 4-flats in $Z_3^0 \cong \RR^8$ that intersect in a single point, we have $\langle F_{ap},F_{bq}\rangle = Z_3^0$.
That is, $L_{apbq} = \eta\left(\langle F_{ap},F_{bq}\rangle\setminus H_0\right) = \RR^{6}$.
 \end{proof}

We are now ready to state the connection between the distinct distances problem and the flats $L_{ap}$.
Let $Q'$ be the set of quadruples $(a,p,b,q)\in \pts^4$ such that $T_{ap}\cap T_{bq}\not\subseteq H_0$.
The following corollary is a special case of Corollary \ref{co:TapTbqIntCharRd} that we will prove in Section \ref{sec:DistancesRd}.

\begin{corollary} \label{co:TapTbqIntCharR3}
We have that $Q'\subset Q$ and $|Q'|\ge |Q|/2$.
\end{corollary}

\parag{Flats in $\RR^6$ and in $\RR^5$.}
We set
\[ \lines = \{L_{ap} :\ a,p\in \pts \text{ and } a\neq p \}.\]
Note that $\lines$ is a set of $\Theta(n^2)$ flats of dimension three in $\RR^6$.
By Corollary \ref{co:TapTbqIntCharR3}, to get an asymptotic upper bound for the number of quadruples in $Q$ it suffices to derive an upper bound for the number of quadruples $(a,p,b,q)\in \pts^4$ such that $T_{ap}\cap T_{bq} \not\subseteq \emptyset$.
By Lemma \ref{le:LapbqDim}, for every such quadruple we have that $L_{ap}\cap L_{bq}$ is a line.
On the other hand, when $T_{ap}\cap T_{bq} \subseteq H_0$ we have that $L_{ap}\cap L_{bq}=\emptyset$.
Thus, it remains to derive an upper bound on the number of pairs of flats of $\lines$ that intersect (in a line).

\begin{lemma} \label{le:3flatsRest}
(a) Every point of $\RR^6$ is contained in at most $n$ flats of $\lines$. \\
(b) Every hyperplane in $\RR^6$ contains at most $n$ flats of $\lines$.
\end{lemma}
\begin{proof}
Consider three distinct points $a,p,q\in \pts$ and note that $T_{ap}\cap T_{aq} = \emptyset$, since a rigid motion cannot simultaneously take $a$ into two distinct points.
This immediately implies part (a) of the lemma.
By Corollary \ref{co:EmptyTIntersection}, $L_{ap}$ and $L_{aq}$ cannot be in the same hyperplane, which implies part (b).
\end{proof}

Let $H_g$ be a generic hyperplane in $\RR^6$, in the sense that every 3-flat of $\lines$ intersects $H_g$ in a 2-flat, and
every line of the form $L_{ap}\cap L_{bq}$ (with $a,b,p,q\in \pts$) intersects $H_g$ at a single point.
Let $\flats = \{L_{ap} \cap H_g :\, L_{ap}\in \lines\}$ and consider $H_g$ as $\RR^5$.
Note that $\flats$ is a set of $\Theta(n^2)$ distinct 2-flats.
Every two 2-flats of $\flats$ are either disjoint or intersect in a single point.
By Lemma \ref{le:3flatsRest}, every point of $\RR^5$ is incident to at most $n$ of the 2-flats of $\flats$ and every hyperplane in $\RR^5$ contains at most $n$ of the 2-flats of $\flats$.

For every integer $k\ge 2$, let $m_k$ denote the number of points of $\RR^5$ that are contained in exactly $k$ of the 2-flats of $\flats$.
Similarly, let $m_{\ge k}$ denote the number of points of $\RR^5$ that are contained in \emph{at least} $k$ of the 2-flats of $\flats$.
Then $|Q|$ is the number of pairs of intersecting flats of $\flats$, and
\[ |Q| = \sum_{k=2}^n m_k \cdot 2\binom{k}{2} < \sum_{k=2}^n k^2 m_k = O\left( \sum_{k=1}^{\log n} 2^{2k} m_{\ge 2^k}\right). \]

If we had the bound $m_{\ge k} = O\left(\frac{n^{10/3}}{k^{2+\eps}}\right)$ for some $\eps>0$, then the above would imply  $|Q|=O(n^{10/3})$.
Combining this with \eqref{eq:lowerQbi} would imply that the points of $\pts$ span $\Omega\left(n^{2/3}\right)$ distinct distances.

An incidence result of Solymosi and Tao \cite{ST12} implies that the number of incidences between $m$ points and $n$ 2-flats in $\RR^5$, with every two 2-flats intersecting in at most one point, is $O(m^{2/3+\eps'}n^{2/3}+m+n)$ (for any $\eps'>0$).
Every incidence bound of this form has a dual formulation involving $k$-rich points (for example, see \cite[Chapter 1]{ShefferBook}).
In this case, the dual bound is: Given $n^2$ 2-flats in $\RR^5$ such that every two intersect in at most one point, for every $k\ge 2$ the number of $k$-rich points is $O\left(\frac{n^{4/(1-\eps')}}{k^{3/(1-\eps')}}+\frac{n^2}{k}\right)$.
By taking $\eps'$ to be sufficiently small with respect to $\eps$, we obtain the bound $m_{\ge k} = O\left(\frac{n^{4+\eps}}{k^{3}}+\frac{n^2}{k}\right)$.
This bound is stronger than the required bound when $k=\Omega(n^{2/3+\eps})$.
That is, it remains to consider the case where $k=O(n^{2/3+\eps})$.
This completes the proof of Theorem \ref{th:R3DD}.
\end{proof}

\section{Distinct distances in $\RR^d$} \label{sec:DistancesRd}

In this section we prove Theorem \ref{th:DDreduction} in every dimension.
While the general outline of the proof remains the same as in the proof of Theorem \ref{th:R3DD}, several steps become significantly more involved.
As before, we embed $\Spun{d}$ in a real space and then perform several projections to lower dimensional spaces.
Since Corollary \ref{co:SpinDefLowD} does not hold for $d\ge 6$, we do not have a simple description of $\Spun{d}$ as in Lemma \ref{le:Spun3}.
This leads us to study $\Spun{d}$ in a more indirect way.

Recall that $\Spun{d}$ is contained in the subspace $Z_d^0\subset X_d$ generated by $\BI$ and by products of an even number of elements from $\{e_1,e_2,\ldots,e_d,e_{d+1}e_{d+2}\}$.
Note that $Z_d^0$ has a basis of size $2^d$.
We consider $Z_d^0$ as $\RR^{2^d}$ by mapping the above basis elements to the standard basis vectors of $\RR^{2^d}$.
With this notation, we study the behavior of $\Spun{d}$ as a set in $\RR^{2^d}$.

\subsection{Studying $m$-terms} \label{ssec:mTerms}

For an even integer $m>0$, an $m$-\emph{term} of $C\ell^0_d$ is a product of $m$ distinct elements from $\{e_1,e_2,\ldots,e_d\}$ (together with a real coefficient).
Similarly, an $m$-term of $Z_d^0$ is a product of $m$ distinct elements from $\{e_1,e_2,\ldots,e_d,e_{d+1}e_{d+2}\}$ (together with a real coefficient).
In both cases a $0$-term is $\BI$ multiplied some real number.
In this section we study several basic properties of $m$-terms.
Since these are just straightforward calculations, the reader might prefer to skip this section and refer to it when necessary.

\begin{lemma} \label{le:mTermsConj}
For a fixed even $m$, let $x \in C\ell_{d}^0 \setminus \{0\cdot \BI\}$ consist entirely of $m$-terms and let $\gamma \in \Spin{d}$.
Then $\gamma x \gamma^{-1}$ also consists entirely of $m$-terms.
\end{lemma}
\begin{proof}
Let $z\in C\ell_{d}^0$ and $\gamma \in \Spin{d}$.
We think of $C\ell_{d}^0$ as $\RR^{2^d}$ and write $\|z\|$ for the Euclidean norm of $z$ in $\RR^{2^d}$.
Note that the first coordinate of $z\overline{z}$ is $\|z\|$ and so is the first coordinate of $\overline{z}z$ (since $\|z\| = \|\overline{z}\|$).
Since $z\gamma^{-1} \overline{z\gamma^{-1}}= z\gamma^{-1}\gamma \overline{z} = z\overline{z}$, by considering the first coordinate of these expressions we get that $\|z\| = \|z\gamma^{-1}\|$.
That is, multiplication by $\gamma^{-1}$ from the right is an orthogonal transformation (with respect to the Euclidean norm).
Similarly, $\overline{\gamma z}\gamma z = \overline{z}z$ implies that multiplication by $\gamma$ from the left is also an orthogonal transformation.
We conclude that the conjugation $z \to \gamma z \gamma^{-1}$ is orthogonal with respect to the Euclidean norm.
Combining this with $\gamma \BI \gamma^{-1} =\BI$ implies that $z$ and $\gamma z \gamma^{-1}$ have the same first coordinate.

For $u_1,\ldots,u_m \in \RR^d$, the product $i(u_1)\cdots i(u_m)$ cannot contain $\ell$-terms for any $\ell > m$.
Moreover, for any $m$-term $e_{k_1}e_{k_2}\cdots e_{k_m}$ we have that $\gamma e_{k_1}e_{k_2}\cdots e_{k_m} \gamma^{-1} = \gamma e_{k_1}\gamma^{-1} \gamma e_{k_2}\gamma^{-1} \hdots \gamma e_{k_m} \gamma^{-1}$.
 This implies that $\gamma x \gamma^{-1}$ cannot contain $\ell$-terms for any $\ell > m$.

We write $\gamma x \gamma^{-1} = \delta + \delta'$, where $\delta$ consists entirely of $m$-terms and $\delta'$ consists entirely of smaller terms.
We have
\begin{align}
N(x - \gamma^{-1}\delta' \gamma) - N(x) - N(\gamma^{-1}\delta'\gamma) &= (x - \gamma^{-1}\delta' \gamma)\overline{(x - \gamma^{-1}\delta' \gamma)} - x\overline{x} - \gamma^{-1}\delta' \overline{\delta'} \gamma \nonumber \\
&=-\left(x\overline{\gamma^{-1}\delta' \gamma} + \gamma^{-1}\delta' \gamma \overline{x}\right). \label{eq:NormsExpression}
\end{align}

For any $y,z \in C\ell_{d}^0$, the first coordinate of $y\overline{z}$ is the dot product of $y$ and $z$ as vectors in $\RR^{2^d}$.
Since $\gamma^{-1}\delta' \gamma$ consists entirely of $\ell$-vector terms with $\ell < m$, the first coordinate of \eqref{eq:NormsExpression} is zero.
Since conjugation by $\gamma$ preserves the first coordinate, we have that the first coordinate of $(x - \gamma^{-1}\delta' \gamma)\overline{(x - \gamma^{-1}\delta' \gamma)} - x\overline{x} - \gamma^{-1}\delta' \overline{\delta'} \gamma$ is the same as the first coordinate of
\begin{align*}
&\gamma\left((x - \gamma^{-1}\delta' \gamma)\overline{(x - \gamma^{-1}\delta' \gamma)} - x\overline{x} - \gamma^{-1}\delta' \overline{\delta'} \gamma\right)\gamma^{-1} \\
&\hspace{18mm}= \gamma(x - \gamma^{-1}\delta' \gamma)\gamma^{-1} \gamma \overline{(x - \gamma^{-1}\delta' \gamma)}\gamma^{-1} - \gamma x\gamma^{-1}\gamma \overline{x}\gamma^{-1} - \delta' \overline{\delta'}\\
&\hspace{18mm} = \delta\overline{\delta} - \left(\delta\overline{\delta} +\delta'\overline{\delta'} + \delta \overline{\delta'} + \delta' \overline{\delta}\right) - \delta'\overline{\delta'} = -\left(\delta \overline{\delta'} + \delta' \overline{\delta} + 2\delta' \overline{\delta'}\right).
\end{align*}

Since $\delta$ and $\delta'$ do not have terms of the same size, the first coordinates of $\delta \overline{\delta'}$ and $\delta' \overline{\delta}$ are both zero.
This implies that first coordinate of $\delta' \overline{\delta'}$ is zero.
 Since this first coordinate equals $\|\delta'\|$, we get that $\delta' = 0$ and complete the proof.
\end{proof}

\begin{lemma} \label{le:TermsedAed}
For a fixed even $m$, let $x\in C\ell_{d-1}^0$ consist entirely of $m$-terms.
Then for every $a \in \RR^{d}$ the expression $xe_{d}(i(a)+e_{d})$ consists entirely of $m$-terms and $(m+2)$-terms.
It the $d$'th coordinate of $a$ is not $-1$, then $xe_{d}(i(a)+e_{d})$ contains at least one $m$-term.
\end{lemma}
\begin{proof}
Let $a_d$ be the $d$'th coordinate of $a$ and let $a' = a -  (0,\ldots,0,a_d)$.
We have that
\[ xe_{d}(i(a)+e_{d}) = x\left((-1-a_{d})\BI +e_{d} i(a')\right) = -(1+a_{d})x - x i(a')e_{d}. \]

Since $xi(a') \in C\ell_{d-1}$ consists entirely of $(m-1)$-terms and $(m+1)$-terms, we have that $x i(a')e_{d}$ consists entirely of $m$-terms and $(m+2)$-terms, as desired.
Since no term of $x$ contains $e_{d}$ and every term of $x i(a')e_{d}$ contains $e_{d}$, if $a_d\neq -1$ then the $m$-terms from $-(1+a_{d})x$ are nonzero and do not get canceled by other terms.
\end{proof}

\begin{lemma}\label{le:TermsSize}
For a fixed even $m$, let $z \in Z_d^0 \setminus \{0\cdot \BI\}$ contain only $m$-terms and let $a,p \in \RR^d$.
Then $(\BI+\frac{1}{2}e_{d+1}e_{d+2}i(p))z (\BI-\frac{1}{2}e_{d+1}e_{d+2}i(a))$ contains only $m$-terms and $(m+2)$-terms.
This expression contains at least one nonzero $m$-term.
\end{lemma}
\begin{proof}
We write $z = z_1 + z_2e_{d+1}e_{d+2}$ where $z_1 \in C\ell_d^0$ and $z_2 \in C\ell_d^1$.
We then have
\[ \left(\BI+\frac{1}{2}e_{d+1}e_{d+2}i(p)\right) z \left(\BI-\frac{1}{2}e_{d+1}e_{d+2}i(a)\right) = z + \frac{1}{2}e_{d+1}e_{d+2}\left(i(p)z_1-z_1i(a)\right). \]

We observe that both $i(p)z_1$ and $z_1i(a)$ contain only $(m+1)$-terms and $(m-1)$-terms, and do not contain $e_{d+1}e_{d+2}$.
This implies that $\frac{1}{2}e_{d+1}e_{d+2}\left(i(p)z_1-z_1i(a)\right)$ contains only $(m+2)$-terms and $m$-terms.
Additionally, the part of $z + \frac{1}{2}e_{d+1}e_{d+2}\left(i(p)z_1-z_1i(a)\right)$ that does not involve $e_{d+1}e_{d+2}$ is exactly $z_1$.
Thus, if $z_1 \neq 0\cdot \BI$ then we have at least one $m$-term.
If $z_1 = 0\cdot \BI$ then $z + \frac{1}{2}e_{d+1}e_{d+2}\left(i(p)z_1-z_1i(a)\right) = z$, and we again have an $m$-term.
\end{proof}

\subsection{Proof of Theorem \ref{th:DDreduction}.}
Let $\pts$ be a set of $n$ points in $\RR^d$.
Let $D$ denote the number of distinct distances that are spanned by $\pts$ and denote these distances as $\delta_1,\ldots, \delta_D$.
We set
\begin{equation*} %\label{eq:quadruples}
Q = \left\{(a,b,p,q)\in \pts^4: |ab| = |pq|>0 \right\}.
\end{equation*}
The quadruples of $Q$ are ordered, so $(a,b,p,q)$ and $(b,a,p,q)$ are considered
as two distinct elements of $Q$.
Our proof is based on double counting $|Q|$.

For every $j \in \{1,\ldots,D\}$, let $E_j = \{(a,b)\in \pts^2 : |ab|=\delta_j \}$.
Since every ordered pair of distinct points $(a,b)\in \pts^2$ appears in exactly one set $E_j$, we have that $\sum_{j=1}^D |E_j| = n^2-n  > n^2/2$.
The Cauchy-Schwarz inequality implies
\begin{equation} \label{eq:lowerQRd}
|Q| = \sum_{j=1}^D |E_j|^2 \ge \frac{1}{D}\left(\sum_{j=1}^D |E_j| \right)^2 > \frac{n^4}{4D}.
\end{equation}

For $a,b,p,q\in \RR^d$ with $a\neq b$, we have $|ab|=|pq|$ if and only if there exists a proper rigid motion in $\SE{d}$ that takes both $a$ to $p$ and $b$ to $q$.
Thus, for every $(a,p)\in \pts^2$ we set
\[ R_{ap} = \{\gamma \in \SE{d} :\, a^\gamma = p\}. \]

To derive an upper bound for $|Q|$ it suffices to bound the number of quadruples $(a,b,p,q)\in \pts^4$ that satisfy $a\neq b$ and $R_{ap}\cap R_{bq} \neq \emptyset$.
Since it would be simpler to work in $\Spun{d}$ rather than in $\SE{d}$, we recall the following definition from \eqref{eq:TapDef}.
\[ T_{ap} = \{x \in \Spun{d} :\, a^x = p\} = \rho_d^{-1}(R_{ap}). \]

\parag{From $\Spun{d}$ to $\RR^{\binom{d+1}{2}}$.}
In Section \ref{sec:DDR3} we studied the bijection $\eta$ from the set of points of $\Spun{3}$ that have a positive $x_1$-coordinate to $\RR^6$.
We now generalize this bijection to the case of $\Spun{d}$.
Let $\Spun{d}_+$ be the set of points of $\Spun{d}$ that have a positive first coordinate (the coordinate that corresponds to the coefficient of $\BI$).

Let $\pi_{1}: \RR^{2^d} \to \RR^{2^d-1}$ be the projection defined by $\pi_1(x_1,x_2,...,x_{2^d}) = (x_2,...,x_{2^d})$.
Let $H_0$ denote the hyperplane in $\RR^{2^d}$ defined by $x_1=0$ and let $H_1$ denote the hyperplane defined by $x_1=1$.
For each $x\in \RR^{2^d}\setminus H_0$ there exists a unique $\lambda_x \in \RR$ such that the $x_1$-coordinate of $\lambda_x x$ is 1.
We define $\pi: \RR^{2^d}\setminus H_0 \to \RR^{2^d-1}$ as $\pi(x) = \pi_1(\lambda_x x)$.
We think of elements of $\RR^{2^d-1}$ as corresponding to elements of $Z_d^0$, except for the coefficient of $\BI$ (which was removed by $\pi_1$).
Let $\pi':\RR^{2^d-1} \to \RR^{{d+1}\choose{2}}$ be the projection that keeps only the $\binom{d+1}{2}$ coordinates corresponding to 2-terms of $Z_d^0$.
We will see that we do not lose information of elements of $\Spun{d}_+$ by keeping only these coordinates.
Finally, let $\eta_d = \pi'\circ \pi_1$.
Note that $\eta_3$ is indeed the map $\eta$ from Section \ref{sec:DDR3}.

We first claim that the restriction of $\pi_1$ to $\Spun{d}_+$ is injective.
Indeed, assume that $\pi_1(x) = y$ for $x\in \Spun{d}_+$ and write $y= (y_2,...,y_{2^d})\in \RR^{2^d-1}$.
This implies that $\lambda_x x = (1,y_2,...,y_{2^d})$.
Since $x\in \Spun{d}_+$, we have that $N(x)=x\overline{x} = \BI$ and thus $N(\lambda_x x) = \lambda_x^2 \cdot \BI$.
That is, the value of $\lambda_x$ is determined up to a sign by $N(\lambda_x x)$.
This sign has to be positive, since the first coordinate of $x$ must be positive.
We conclude that for every $y\in \RR^{2^d-1}$ there exists at most one $x\in \Spun{d}_+$ such that $\pi_1(x)= y$.

Set
\begin{align*}
G_d &= \{r \cdot \gamma \in C\ell_d^0 :\ r\in \RR\setminus\{0\} \text{ and } \gamma\in \Spin{d}\}, \\
J_d &= \{r \cdot x \in Z_d^0 :\ r\in \RR\setminus\{0\} \text{ and } x\in \Spun{d}\}.
\end{align*}
Note that $G_d$ is a group under the product operation of $C\ell_d^0$.
Similarly, $J_d$ is a group under the product operation of $Z_d^0$.
By studying these groups, we will obtain information about $\eta_d$ and about the structure of $\Spun{d}$.

The following lemma provides a consistent form for writing elements of $G_d$.
Below we will rely on this lemma to prove various claims by induction on $d$.

\begin{lemma} \label{le:GJReps}
(a) For every element $g \in G_{d}$ there exists $h\in G_{d-1}$ that satisfies the following.
Either $g= he_{d-1}e_{d}$ or there exists $u\in \SSS^{d-1}\setminus \{i^{-1}(-e_{d})\}$ such that $g = h\left(e_{d}i(u)-\BI\right)$. \\
(b) For every $z \in J_d$ there exist $v\in \RR^d$ and $g\in G_d$ such that $z = g \left(\BI- \frac{1}{2}e_{d+1}e_{d+2}i(v)\right)$.
\end{lemma}
\begin{proof}
(a) By definition, for every $g \in G_{d}$ there exists $r\in \RR\setminus\{0\}$ such that $g/r \in \Spin{d}$.
This implies that $(g/r)^{-1} = \overline{g/r}$, so $(g/r) \overline{(g/r)} = 1$.
That is, $g^{-1} =\overline{g}/r^2$.
We define the group action of $g$ on $v\in \RR^{d}$ to be
\[ i^{-1}(g i(v) g^{-1}) = i^{-1}\left(\frac{g}{r} i(v) \overline{\left(\frac{g}{r}\right)}\right) = i^{-1}\left(\frac{g}{r} i(v) \left(\frac{g}{r}\right)^{-1}\right). \]
Since this is the action of $g/r \in \Spin{d}$ on $v$, it is a rotation of $\SO{d}$.
Thus, the action of $g$ maps some point $u\in \SSS^{d-1}$ to $i^{-1}(e_{d})$.

We first assume that $u\neq i^{-1}(-e_{d})$.
We write $s=\|u+(0,\ldots,0,1)\|$ and note that $s\neq 0$.
Since $i^{-1}(e_{d}),\frac{u + i^{-1}(e_{d})}{s} \in \SSS^{d-1}$, we get that $x = e_{d}\left(e_{d}+i(u)\right)/s \in \Spin{d}$.
Since $u\in \SSS^{d-1}$, we note that the vectors $u+i^{-1}(e_{d}),u-i^{-1}(e_{d}) \in \RR^d$ are orthogonal.
By Lemma \ref{le:AntiCommute} we have
\begin{align*}
xi(u)x^{-1} &= \frac{e_d\left( e_d+ i(u)\right) i(u) \left( e_d+ i(u)\right)e_d}{s^2} = \frac{e_d\left( i(u)+ e_d\right) \left(\frac{i(u)+e_d}{2} + \frac{i(u)-e_d}{2}\right) \left( e_d+ i(u)\right)e_d}{s^2} \\[2mm]
&= \frac{e_d\left( i(u)+ e_d\right)^2 \left((i(u)+e_d) - (i(u)-e_d)\right)e_d}{2s^2} = \frac{- e_d \cdot 2e_d \cdot e_d}{2} = e_d.
\end{align*}

The above implies that $gx^{-1}$ is in the stabilizer of $i^{-1}(e_{d})$.
We observe that the stabilizer of $i^{-1}(e_{d})$ is $G_{d-1}$.
Setting $h = (g\cdot x^{-1}/s) \in G_{d-1}$, we get that
\[ g = g x^{-1} x = h e_d\left( e_d+ i(u)\right) = h(e_{d}i(u)-\BI). \]

The above completes the proof of the case where $u\neq i^{-1}(-e_{d})$.
We now assume that $u= i^{-1}(-e_{d})$.
That is, that $g e_d g^{-1} = -e_d$.
Let $h= - ge_{d-1} e_d$ and note that $h^{-1} = - e_d e_{d-1}g^{-1}$.
This implies that $h e_d h^{-1} = e_d$.
As before, since $h$ is in the stabilizer of $e_d$ we have $h\in G_{d-1}$.
We get that $g = - g e_{d-1}e_de_{d-1}e_d = he_{d-1}e_d$, as asserted.

(b) Since $z \in J_{d}$, there exists $r\in \RR$ such that $z/r \in \Spun{d}$.
By Lemma \ref{le:SpunDesc}, there exist $\gamma\in \Spin{d}$ and $u\in \RR^d$ such that $z/r = \gamma\left(\BI+e_{d+2}e_{d+1}i(u)\right)$.
The assertion of the lemma is obtained by setting $g = r \gamma$ and $v=-2u$.
\end{proof}

The following two lemmas will help us to show that the restriction of $\eta_d$ to $\Spun{d}_+$ is injective.

\begin{lemma} \label{le:GdTwoVectors}
If $g,g'\in G_{d}$ have the same nonzero first coordinate and the same 2-terms, then $g = g'$.
\end{lemma}
\begin{proof}
We prove the lemma by induction on $d$.
For the induction basis, note that the claim is trivial when $d\le 3$.
We now assume that the claim holds for $G_{d-1}$ and prove it for $G_d$.

Consider $g,g' \in G_{d}$ that satisfy the assumption of the lemma.
As in the proof of Lemma \ref{le:GJReps}(a), if $g(-e_{d})g^{-1} = i^{-1}(e_{d})$ then there exists $h\in G_{d-1}$ such that $g = h e_{d}e_{d-1}$.
This contradicts $g$ having a nonzero first coordinate, so we must have $g(-e_{d})g^{-1} \neq i^{-1}(e_{d})$.
A symmetric argument implies that $g'(-e_{d})(g')^{-1} \neq i^{-1}(e_{d})$.
By Lemma \ref{le:GJReps}(a), there exist $h,h' \in G_{d-1}$ and $u,u' \in S^{d-1}\setminus \{-e_{d}\}$ such that $g = h\left(e_{d}i(u)-\BI\right)$ and $g' = h'\left(e_{d}i(u')-\BI\right)$.

We write $h = r\cdot \BI + h_{2}+h_+$ such that $r\in \RR$, every term of $h_2$ is a 2-term, and $h_+$ contains no 0-term and 2-terms.
That is, we have
\[ g = h\left(e_{d}i(u)-\BI\right)  = h e_{d}i(u) - r\cdot \BI - h_{2}-h_+. \]

Let $u_{j}$ be the $j$'th coordinate of $u$, and set $u_* = u - (0,\ldots,0,u_{d})$.
Then
\[ g= h e_{d}i(u_*) - r(1+u_{d})\cdot \BI - h_{2}(1+u_{d})-h_+(1+u_{d}). \]
A symmetric argument gives
\[ g'= h' e_{d}i((u')_*) - r'(1+u'_{d})\cdot \BI - h'_{2}(1+u'_{d})-h'_+(1+u'_{d}). \]

By the assumption on $u$ and $u'$, we have that $u_{d}\neq -1$ and $u'_{d}\neq -1$.
Since $g$ and $g'$ have nonzero first coordinates, we have that $r \neq 0$ and $r' \neq 0$.
Since these first coordinates are identical, $r(1+u_{d}) = r'(1+u'_{d})$.
By the assumption on the 2-terms of $g$ and $g'$, we have that $(1+u_{d})h_{2}= (1+u'_{d})h'_{2}$ (the expressions $h e_{d}i(u_*)$ and $h' e_{d}i((u')_*)$ may also contain 2-terms, but these all involve $e_d$ and thus do not affect the terms of $h_{2},h'_{2}\in C\ell_{d-1}^0$).

By setting $\ell = (1+u_{d})/(1+u'_{d})$ we get that $r' = \ell r \neq 0$ and $h'_{2} = \ell h_{2}$.
We may thus apply the induction hypothesis on $h,\ell h' \in G_{d-1}$, to obtain that $h' = \ell h$.
That is,
\begin{align*}
g = h e_{d}i(u_*) - h(1+u_d) \quad \text{ and } \quad g' = \ell h e_{d}i((u')_*) - \ell h(1+u'_{d}).
\end{align*}

We write $h_{2} = \sum_{1\le j<k\le d-1} \lambda_{j,k}e_je_k$, where the coefficients $\lambda_{j,k}$ are in $\RR$.
Consider the terms of the form $e_je_d$ for some $1\le j \le d-1$.
By the assumption about 2-terms in $g$ and $g'$, we have
\begin{align*}
r e_{d} i(u_*) + \sum_{1\le j<k\le d-1} \lambda_{j,k}e_j&e_k(u_{j}e_de_j + u_{k}e_de_k)\\
& = \ell r_x e_{d} i((u')_*) + \ell \sum_{1\le j<k\le d-1} \lambda_{j,k}e_je_k(u'_{j}e_de_j + u'_{k}e_de_k).
\end{align*}
Simplifying, we have
\begin{align*}
r e_{d} i(u_*) + \sum_{1\le j<k\le d-1} \lambda_{j,k}&(-u_{j}e_ke_d + u_{k}e_je_d)\\
& = \ell r e_{d} i((u')_*) + \ell \sum_{1\le j<k\le d-1} \lambda_{j,k}(-u'_{j}e_ke_d + u'_{k}e_je_d).
\end{align*}

This leads to the following system of linear equations.
\begin{align*}
\begin{pmatrix}
r & \lambda_{1,2} & \lambda_{1,3} & \cdots & \lambda_{1,d-1}\\
-\lambda_{1,2} & r & \lambda_{2,3} & \cdots & \lambda_{2,d-1}\\
-\lambda_{1,3} & -\lambda_{2,3}& r & \cdots & \lambda_{3,d-1}\\
\vdots & \vdots & \vdots& \ddots & \vdots\\
-\lambda_{1,d-1} & -\lambda_{2,d-1} & -\lambda_{3,d-1} & \cdots & r
\end{pmatrix}
\begin{pmatrix} u_{1} - \ell u'_{1} \\ u_{2} - \ell u'_{2} \\ u_{3} - \ell u'_{3} \\ \vdots \\ u_{d-1} - \ell u'_{d-1} \end{pmatrix}
= 0.
\end{align*}

After placing zeros in every cell of the main diagonal, the above matrix becomes skew-symmetric.
Recall that the eigenvalues of a skew-symmetric matrix are pure imaginary, and that adding a constant $c$ to every element of the main diagonal adds $c$ to every eigenvalue.
Since $r$ is a nonzero real number, we get that the above matrix has no zero eigenvalues, and is thus invertible.
This implies that the only solution to the above system is $u_{j} = \ell u'_{j}$ for every $1\le j \le d-1$.

By recalling that $\ell = (1+u_{d})/(1+u'_{d})$ we get
\[ u_{d}^2 = (\ell(1+u'_{d}) -1)^2 = \ell^2(1+2u'_{d} + (u'_{d})^2) - 2\ell(1+u'_{d}) +1. \]
Combining the above with $u,u' \in \SSS^{d-1}$ leads to
\begin{align*}
1 = \|u\|^2 = \sum_{j=1}^d u_{j}^2 =  \sum_{j=1}^{d-1} \ell^2 (u'_{j})^2 + \ell^2(1+2u'_{d} &+ (u'_{d})^2) - 2\ell(1+u'_{d}) +1 \\
&= 2\ell^2 +2\ell^2 u'_{d}  - 2\ell -2\ell u'_{d} +1.
\end{align*}

Tidying up the above gives $\ell +\ell u'_{d} = 1 + u'_{d}$, so $\ell = 1$.
We thus get that $h=h'$ and $u=u'$, and conclude that $g=g'$.
\end{proof}

\begin{lemma} \label{le:JdTwoVectorsRule}
If $x,y\in J_{d}$ have the same nonzero first coordinate and the same 2-terms, then $x = y$.
\end{lemma}
\begin{proof}
By Lemma \ref{le:GJReps}(b), there exist $g,h \in G_d$ and $u_x,u_y \in \RR^d$ such that $x = g \left(\BI- \frac{1}{2}e_{d+1}e_{d+2}i(u_x)\right)$ and $y = h \left(\BI- \frac{1}{2}e_{d+1}e_{d+2}i(u_y)\right)$.
We write $g = r_x\cdot \BI + g_2+g'$ where $r_x\in \RR$, every term of $g_2$ is a 2-term, and $g'$ contains no 0-term and no 2-terms.
We symmetrically write $h = r_y\cdot \BI + h_2 + h'$.
That is, we have
\begin{align*}
x &= g \left(\BI- \frac{1}{2}e_{d+1}e_{d+2}i(u_x)\right) = (r_x+g_2 + g') \left(\BI- \frac{1}{2}e_{d+1}e_{d+2}i(u_x)\right) \\
&\hspace{80mm} = r_x \cdot \BI + g_2 + g' - \frac{1}{2}g e_{d+1}e_{d+2}i(u_x),\\
y &= h \left(\BI- \frac{1}{2}e_{d+1}e_{d+2}i(u_y)\right) = (r_y+h_2 + h') \left(\BI- \frac{1}{2}e_{d+1}e_{d+2}i(u_y)\right) \\
&\hspace{80mm} = r_y \cdot \BI + h_2 + h' - \frac{1}{2}h e_{d+1}e_{d+2}i(u_y).
\end{align*}

Since $x$ and $y$ have the same first coordinate, we have that $r_x=r_y$.
Since $\eta_d(x)=\eta_d(y)$, we have $g_2=h_2$.
By lemma \ref{le:GdTwoVectors}, we get that $g = h$.
We thus have
\begin{align*}
x = g - \frac{1}{2}g e_{d+1} e_{d+2}i(u_x), \quad \text{ and } \quad y = g- \frac{1}{2}g e_{d+1} e_{d+2}i(u_y).
\end{align*}

We write $g_2 = \sum_{1\le j<k\le d} \lambda_{j,k}e_je_k$, where the coefficients $\lambda_{j,k}$ are in $\RR$.
Also, let $u_{x,j}$ denote the $j$'th coordinate of $u_x$.
We now consider the terms of the form $e_{j}e_{d+1}e_{d+2}$ with $1 \le j \le d$.
Since $x$ and $y$ have the same 2-terms, we have
\begin{align*}
r_xe_{d+1}e_{d+2}i(u_x) + \sum_{1\le j<k\le d} &\lambda_{j,k}e_je_k(u_{x,j}e_{d+1}e_{d+2}e_j + u_{x,k}e_{d+1}e_{d+2}e_k)\\
&= r_xe_{d+1}e_{d+2}i(u_y) + \sum_{1\le j<k\le d} \lambda_{j,k}e_je_k(u_{y,j}e_{d+1}e_{d+2}e_j + u_{y,k}e_{d+1}e_{d+2}e_k).
\end{align*}
Simplifying, we have
\begin{align*}
r_xe_{d+1}e_{d+2}i(u_x) + \sum_{1\le j<k\le d} &\lambda_{j,k}(-u_{x,j}e_ke_{d+1}e_{d+2} + u_{x,k}e_je_{d+1}e_{d+2})\\
&= r_xe_{d+1}e_{d+2}i(u_y) + \sum_{1\le j<k\le d} \lambda_{j,k}(-u_{y,j}e_ke_{d+1}e_{d+2} + u_{y,k}e_je_{d+1}e_{d+2}).
\end{align*}

This leads to the following system of linear equations.
\begin{align*}
\begin{pmatrix}r_x &  \lambda_{1,2} &  \lambda_{1,3} & \cdots & \lambda_{1,d}\\
-\lambda_{1,2} & r_x & \lambda_{2,3} & \cdots & \lambda_{2,d}\\
-\lambda_{1,3} & -\lambda_{2,3}& r_x & \cdots & \lambda_{3,d}\\
\vdots & \vdots & \vdots& \ddots & \vdots\\
-\lambda_{1,d} & -\lambda_{2,d} & -\lambda_{3,d} & \cdots & r_x
\end{pmatrix}
\begin{pmatrix} u_{x,1} - u_{y,1} \\ u_{x,2} - u_{y,2} \\ u_{x,3} - u_{y,3} \\ \vdots \\ u_{x,d} - u_{y,d}\end{pmatrix}
= 0.
\end{align*}

By repeating the eigenvalues argument from the proof of Lemma \ref{le:GdTwoVectors}, we get that the only solution to this system is $u_{x,j} = u_{y,j}$ for every $1\le j\le d$.
Since $u_x=u_y$, we conclude that $x=y$.
\end{proof}

\begin{corollary} \label{co:EtaDInjective}
The restriction of $\eta_d$ to $\Spun{d}_+$ is injective.
\end{corollary}
\begin{proof}
Consider two elements $x,y\in \Spun{d}_+$ such that $\eta_d(x) = \eta_d(y)$.
Let $x_1$ be the first coordinate of $x$ and let $y_1$ be the first coordinate of $y$.
We set $x' = x/x_1$ and $y'=y/y'$, and note that $x',y' \in H_1$.
Moreover, we have that $\eta_d(x) = \eta_d(x') = \pi'\circ\pi_1(x')$ and $\eta_d(y) = \eta_d(y') = \pi'\circ\pi_1(y')$.
This also implies that $\eta_d(x') = \eta_d(y')$, which in turn implies that $x'$ and $y'$ have the same 2-terms.
Since $x'$ and $y'$ also have the same first coordinate, Lemma \ref{le:JdTwoVectorsRule} states that $x'=y'$.
By the definition of $\Spun{d}_+$, there is a unique $r\in \RR$ such that $r\cdot x' \in \Spun{d}_+$.
We thus conclude that $x=y$.
\end{proof}

We next show that the restriction of $\eta_d$ to $\Spun{d}_+$ is surjective in a similar manner.

\begin{lemma} \label{le:SpinSurject}
Consider $r\in \RR$ and elements $\lambda_{j,k} \in \RR$ for every $1\le j < k \le d$, such that $r\neq 0$.
Then there exists $g \in G_{d}$ such that the first coordinate of $g$ is $r$ and the coefficient of the term $e_je_k$ in $g$ is $\lambda_{j,k}$.
\end{lemma}
\begin{proof}
We prove the lemma by induction on $d$.
For the induction basis, note that the claim is trivial when $d=1$.
For the induction step, we assume that the claim holds for $G_{d-1}$ and consider the case of $G_d$.

By the induction hypothesis, there exists $h \in G_{d-1}$ with first coordinate $r$ and the term $\lambda_{j,k}e_je_k$ for every $1\leq j <k \leq d-1$.
We set $g=h - he_{d}i(u) \in G_{d}$, for some $u\in \RR^{d-1}$ that will be determined below.
Note that the first coordinate of $g$ is $r$ and the coefficient of the term $e_je_k$ in $g$ is $\lambda_{j,k}$, for $1\le j < k \le d-1$.
We now consider the terms of the form $e_je_{d}$ in $g$, and observe that these are all in $-he_{d}i(u)$.
Let $u_{j}$ denote the $j$'th coordinate of $u$.
Since for every $1\le j \le d-1$ we would like $g$ to contain the term $\lambda_{j,d} e_je_{d}$, we get the following system of linear equations.
\begin{align*}
\begin{pmatrix}r & -\lambda_{1,2} & -\lambda_{1,3} & \cdots & -\lambda_{1,d-1}\\
\lambda_{1,2} & r & -\lambda_{2,3} & \cdots &-\lambda_{2,d-1}\\
\lambda_{1,3} & \lambda_{2,3}& r & \cdots & -\lambda_{3,d-1}\\
\vdots & \vdots & \vdots& \ddots & \vdots\\
\lambda_{1,d-1} & \lambda_{2,d-1} & \lambda_{3,d-1} & \cdots & r
\end{pmatrix}
\begin{pmatrix} u_1 \\ u_2 \\ u_3 \\ \vdots \\ u_{d-1} \end{pmatrix}
= \begin{pmatrix} \lambda_{1,d} \\ \lambda_{2,d} \\ \lambda_{3,d} \\ \vdots \\ \lambda_{d-1,d} \end{pmatrix}.
\end{align*}

By repeating the eigenvalues argument from the proof of Lemma \ref{le:GdTwoVectors}, we get that the above matrix is invertible. Thus, there exists a choice of $u_1, \ldots, u_{d-1}$ such that the above system holds.
That is, there exists $u\in \RR^{d-1}$ such that $g$ satisfies the assertion of the lemma.
\end{proof}

\begin{lemma} \label{le:JdSurjective}
Consider $r\in \RR$ and elements $\lambda_{j,k} \in \RR$ for every $1\le j < k \le d+1$, such that $r\neq 0$.
Then there exists $g \in J_{d}$ such that the first coordinate of $g$ is $r$ and the coefficient of the term $e_je_k$ in $g$ is $\lambda_{j,k}$ (when $k=d+1$ we consider the term $e_je_{d+1}e_{d+2}$ instead).
\end{lemma}
\begin{proof}
By lemma \ref{le:SpinSurject}, there exists $h \in G_d$ such that the first coordinate of $h$ is $r$ and the coefficient of the term $e_je_k$ is $\lambda_{j,k}$, for every $1\le j < k \le d$.
We set $ g= h  - h e_{d+1}e_{d+2}i(u) \in J_{d}$, for a vector $u\in \RR^{d}$ that will be determined below.
Note that the first coordinate of $g$ is $r$ and the coefficient of the term $e_je_k$ in $g$ is $\lambda_{j,k}$, for $1\le j < k \le d$.
We now consider the terms of the form $e_je_{d+1}e_{d+2}$ in $g$, and observe that these are all in $-h e_{d+1}e_{d+2}i(u)$.
Let $u_{j}$ denote the $j$'th coordinate of $u$.
Since for every $1\le j \le d$ we would like $g$ to contain the term $\lambda_{j,d+1} e_je_{d+1}e_{d+2}$, we get the following system of linear equations.

\begin{align*}
\begin{pmatrix}
r & -\lambda_{1,2} & -\lambda_{1,3} & \cdots & -\lambda_{1,d}\\
\lambda_{1,2} & r & -\lambda_{2,3} & \cdots & -\lambda_{2,d}\\
\lambda_{1,3} & \lambda_{2,3}& r & \cdots & -\lambda_{3,d}\\
\vdots & \vdots & \vdots& \ddots & \vdots\\
\lambda_{1,d} & \lambda_{2,d} & \lambda_{3,d} & \cdots & r
\end{pmatrix}
\begin{pmatrix} u_1 \\ u_2 \\ u_3 \\ \vdots \\ u_d \end{pmatrix}
= \begin{pmatrix} \lambda_{1,d+1} \\ \lambda_{2,d+1} \\ \lambda_{3,d+1} \\ \vdots \\ \lambda_{d,d+1} \end{pmatrix}.
\end{align*}

By repeating the eigenvalues argument from the proof of Lemma \ref{le:GdTwoVectors}, we get that the above matrix is invertible. Thus, there exists a choice of $u_1, \ldots, u_{d}$ such that the above system holds.
That is, there exists $u\in \RR^{d}$ such that $g$ satisfies the assertion of the lemma.
\end{proof}

\begin{theorem} \label{th:etaBijectRd}
The map $\eta_d:\Spun{d}_+ \to \RR^{\binom{d+1}{2}}$ is a bijection.
\end{theorem}
\begin{proof}
By Corollary \ref{co:EtaDInjective} the restriction of $\eta_d$ to $\Spun{d}_+$ is injective.
It remains to show that this restriction is surjective on $\RR^{\binom{d+1}{2}}$.
Consider $v \in \RR^{{d+1}\choose{2}}$.
By Lemma \ref{le:JdSurjective}, there exists $g \in J_{d}$ such that $\eta_d(g) = v$ and the first coordinate of $g$ is 1.
By the definition of $J_d$, there exists $r\in \RR\setminus\{0\}$ such that $r g \in \Spun{d}$.
We have that $\eta_d(rg) = \eta_d(g) = v$.
Thus, the restriction of $\eta_d$ to $\Spun{d}_+$ is surjective on $\RR^{{d+1}\choose{2}}$.
\end{proof}

Now that we established that the restriction of $\eta_d$ to $\Spun{d}_+$ is a bijection, we move to study the image of $T_{ap}\cap \Spun{d}_+$ under $\eta_d$.
In particular, we will show that this image is a ${{d}\choose{2}}$-flat.
Similarly to $\Spun{d}_+$, let $\Spin{d}_+$ be the set of elements of $\Spin{d}$ where the term $\BI$ has a positive coefficient.
We also recall the definition of $F_{ap}$ from \eqref{eq:FapDefD}.

\begin{lemma} \label{le:TapProjection}
For $a,p\in \RR^d$, we have $\eta_{d}(T_{ap} \cap \Spun{d}_+) =  \eta_d\left(F_{ap} \setminus H_0 \right)$.
\end{lemma}
\begin{proof}
By Lemma \ref{le:TapStructureRd} we have $T_{ap} \cap \Spun{d}_+ \subseteq F_{ap} \setminus H_0$, which implies that $\eta_{d}(T_{ap} \cap \Spun{d}_+) \subseteq  \eta_d\left(F_{ap} \setminus H_0 \right)$.
For the other direction, we consider $z \in F_{ap} \setminus H_0$.
To complete the proof, we will show that there exists $x \in T_{ap} \cap \Spun{d}_+$ such that $\eta_d(z) = \eta_d(x)$.

We recall that $\left(\BI-\frac{1}{2}e_{d+1}e_{d+2}i(p)\right)\left(\BI+\frac{1}{2}e_{d+1}e_{d+2}i(p)\right) = \BI$.
Since $z\in F_{ap}$, we have
\begin{align*}
\left(\BI-\frac{1}{2}e_{d+1}e_{d+2}i(p)\right) z \left(\BI+\frac{1}{2}e_{d+1}e_{d+2}i(a)\right) \in C\ell_d^0.
\end{align*}

Since $C\ell_d^0$ is contained in $Z^0_d$, we can also think of $G_d$ as contained in $Z^0_d$.
Then, by Lemma \ref{le:SpinSurject} there exists $\gamma \in \Spin{d}_+$ such that
\[ \eta_d(\gamma) = \eta_d\left(\left(\BI-\frac{1}{2}e_{d+1}e_{d+2}i(p)\right) z \left(\BI+\frac{1}{2}e_{d+1} e_{d+2}i(a)\right) \right). \]

Thus, there exists $\lambda \in \RR\setminus\{0\}$ such that $\left(\BI-\frac{1}{2}e_{d+1}e_{d+2}i(p)\right) z \left(\BI+\frac{1}{2} e_{d+1}e_{d+2}i(a)\right) - \lambda \gamma$ contains no 0-term and no 2-terms.
By Lemma \ref{le:TermsSize}, multiplying from the left by $\left(\BI+\frac{1}{2}e_{d+1}e_{d+2}i(p)\right)$ and from the right by $\left(\BI-\frac{1}{2}e_{d+1} e_{d+2}i(a)\right)$ cannot create any 0-terms and 2-terms.
That is, setting $y=\lambda\left(\BI+\frac{1}{2}e_{d+1}e_{d+2}i(p)\right) \gamma \left(\BI-\frac{1}{2}e_{d+1}e_{d+2}i(a)\right)$, the expression $z - y$ contains no 0-terms and 2-terms.
Equivalently, $z$ and $y$ have the same the same 0-terms and 2-terms.
Note that $y$ has a nonzero first coordinate, so $\eta_d \left(y\right) = \eta_d(z)$.

Set $x= y/\lambda$.
By Lemma \ref{le:TapBasicRd}, we have
\[ x = \left(\BI+\frac{1}{2}e_{d+1}e_{d+2}i(p)\right) \gamma \left(\BI-\frac{1}{2}e_{d+1}e_{d+2}i(a)\right) = \gamma + \frac{1}{2}e_{d+1}e_{d+2}\left(i(p)\gamma - \gamma i(a)\right) \in T_{ap}.\]
Since $\gamma \in \Spin{d}_+$, we get that $x \in \Spun{d}_+$.
Finally,  $\eta_d(x) = \eta_d \left(y\right) = \eta_d(z)$, as required.
\end{proof}

Note that the map $\eta_d(x)$ is well-defined for every point $x\in \RR^{2^{d}}\setminus H_0$.
Additionally, when we restrict the domain of $\eta_d$ to $H_1$ it is a linear map.
Let $\eta'_d: \RR^{2^d} \to \RR^{\binom{d+1}{2}}$ be the standard projection that keeps only the coordinates corresponding to basis elements of $Z_d^0$ that are 2-terms.
We can think of $\eta'_d$ as a linear extension of the restricted $\eta_d$ to $\RR^{2^d}$.

\begin{lemma} \label{le:flatDimension}
The projection $\eta_d(T_{ap} \cap \Spun{d}_+)$ is a $\binom{d}{2}$-flat.
\end{lemma}
\begin{proof}
Since $F_{ap}$ is ruled by lines that are incident to the origin, we get that $\eta_d(F_{ap} \setminus H_0) = \eta_d(F_{ap}\cap H_1)$.
Since $F_{ap}\cap H_1$ is a flat and the restriction of $\eta_d$ to $H_1$ is a linear map, Lemma \ref{le:TapProjection} implies that $\eta_{d}(T_{ap} \cap \Spun{d}_+)$ is a flat in $\RR^{\binom{d+1}{2}}$.
It remains to establish the dimension of this flat.

From the definition of $F_{ap}$ in \eqref{eq:FapDefD} we notice that $F_{ap}\cap H_1 \neq \emptyset$ (for example, by taking the element $\BI$ from $C\ell_d^0$ in this definition).
We also note that every $v\in F_{ap}\cap H_1$ satisfies $F_{ap}\cap H_1 = v+ (F_{ap}\cap H_0)$.
For such a $v$ we have
\[ \eta_d(F_{ap}\setminus H_0) = \eta_d(F_{ap}\cap H_1) =\eta'_d(v+F_{ap}\cap H_0) = \eta_d(v) + \eta'_d(F_{ap}\cap H_0).\]

Combining the above with Lemma \ref{le:TapProjection} implies that
\[ \dim (\eta_d(T_{ap} \cap \Spun{d}_+)) = \dim (\eta'_d(F_{ap}\cap H_0)) = \dim(F_{ap}\cap H_0) - \dim \left(\ker \eta'_d \cap F_{ap}  \cap H_0\right). \]

Since $\dim F_{ap} = \dim (C\ell_d^0) = 2^{d-1}$ and $F_{ap}$ properly intersects the hyperplane $H_0$, we get that $\dim(F_{ap}\cap H_0)=2^{d-1}-1$.
Note that the elements of $\ker \eta'_d \cap F_{ap}  \cap H_0$ do not have 2-terms and 0-terms.
Let $\tau_{ap}:Z_d^0 \to Z_d^0$ be the map defined by $\tau_{ap}(x) = (\BI+\frac{1}{2}e_{d+1}e_{d+2}i(p)) x (\BI-\frac{1}{2}e_{d+1}e_{d+2}i(a))$.
By Lemma \ref{le:TermsSize}, we have that $\ker \eta'_d \cap F_{ap}  \cap H_0$ is the subspace generated by
\[ \left\{\tau_{ap}(x) :\ x \text{ is an element of the standard basis of } C\ell_d^0 \text{ that is an } m\text{-term for some } m\geq 4\right\}. \]

Since $\tau_{ap}^{-1}(x)=\left(\BI-\frac{1}{2}e_{d+1}e_{d+2}i(p)\right) x \left(\BI+\frac{1}{2}e_{d+1}e_{d+2}i(a)\right)$, we note that $\tau_{ap}$ is a linear bijection.
This implies that the above generating set is linearly independent, so $\dim \left(\ker \eta'_d \cap F_{ap}  \cap H_0\right) = 2^{d-1}-\binom{d}{2}-1$.
We conclude that
\[\dim (\eta_d(T_{ap} \cap \Spun{d}_+)) = 2^{d-1}-1 - \left(2^{d-1}-\binom{d}{2}-1\right) = \binom{d}{2}. \]
\end{proof}

\parag{Studying the flats in $\RR^{\binom{d+1}{2}}$.}
Let $L_{ap} = \eta_d(T_{ap}\setminus H_0)$ be the $\binom{d}{2}$-flat in $\RR^{\binom{d+1}{2}}$ that corresponds to $T_{ap}$.
Given points $a,p,b,q\in \RR^d$, we now study what happens to $L_{ap}$ and $L_{bq}$ when $T_{ap}\cap T_{bq}\neq \emptyset$.
This part is mostly identical to the case of $\RR^3$ that was presented in Section \ref{sec:DDR3}.
In particular, the proofs of Lemma \ref{le:EmptyIntTCond}, Lemma \ref{le:FapIntersectionChar}, Lemma \ref{le:ContainingFlat}, and Corollary \ref{co:EmptyTIntersection} easily extend to $\RR^d$ (by changing $e_{4}e_{5}$ to $e_{d+1}e_{d+2}$ and other such straightforward revisions).

The proof of Lemma \ref{le:LapbqDim} does not immediately extend to $\RR^d$.
Instead of that lemma, we rely on the three following ones.
Let $T_{ap+}$ be the set of points of $T_{ap}$ that have a positive first coordinate.

\begin{lemma} \label{le:TauIntProject}
If $T_{ap}\cap T_{bq} \not \subseteq H_0$ then $L_{ap}\cap L_{bq} = \eta_d\left(F_{ap} \cap F_{bq}\cap H_1\right)$.
\end{lemma}
\begin{proof}
By Lemma \ref{le:TapStructureRd} we have $T_{ap+}\cap T_{bq+} \subseteq \left(F_{ap}\cap F_{bq}\right)\setminus H_0$.
This implies that
\[ \eta_d(T_{ap+}\cap T_{bq+}) \subseteq \eta_d\left(\left(F_{ap}\cap F_{bq}\right)\setminus H_0\right) = \eta_d\left(F_{ap}\cap F_{bq}\cap H_1\right). \]

By Lemma \ref{le:TapProjection} we have that $\eta_d(T_{ap+}) = \eta_d(F_{ap}\setminus H_0) = \eta_d(F_{ap}\cap H_1)$,
and symmetrically $\eta_d(T_{bq+}) = \eta_d(F_{bq}\cap H_1)$.
Combining this with Theorem \ref{th:etaBijectRd} implies that
\begin{align*}
\eta_d(T_{ap+}\cap T_{bq+}) = \eta_d(T_{ap+})\cap \eta_d(T_{bq+}) = \eta_d(F_{ap}\cap H_1)\cap \eta_d(F_{bq}\cap H_1) \supseteq \eta_d\left(F_{ap}\cap F_{bq}\cap H_1\right).
\end{align*}

Combining the above, we conclude that
\[ \eta_d\left(F_{ap}\cap F_{bq}\cap H_1\right) = \eta_d(T_{ap+}\cap T_{bq+}) = L_{ap} \cap L_{bq}, \]
as asserted.
\end{proof}

In Lemma \ref{le:LapbqDimRd} below, we will study $L_{ap}\cap L_{bq}$ when $T_{ap}\cap T_{bq} \not\subseteq H_0$.
Handling the case where $T_{ap}\cap T_{bq}\neq \emptyset$ and $T_{ap}\cap T_{bq} \subseteq H_0$ is more difficult.
The following lemma shows that this problematic case cannot happen too often.

\begin{lemma} \label{le:FewBadPairs}
Assume that $T_{ap}\cap T_{bq}\neq \emptyset$.
Then $T_{ap}\cap T_{bq} \subseteq H_0$ if and only if $a-b = q-p$.
\end{lemma}
\begin{proof}
By the (straightforward) extension of Lemma \ref{le:FapIntersectionChar}  to $\Spun{d}$, we have
\begin{equation} \label{FIntersectionNonZero}
F_{ap} \cap F_{bq} = \left(\BI + \frac{1}{2}e_{d+1}e_{d+2}i(p)\right)\beta C\ell_{d-1}^0\alpha \left(\BI - \frac{1}{2}e_{d+1}e_{d+2}i(a)\right),
\end{equation}
for any $\alpha,\beta \in \Spin{d}$ that satisfy $\alpha \frac{i(b-a)}{\|b-a\|}\alpha^{-1} = e_{d}$ and $\beta e_{d} \beta^{-1} = \frac{i(q-p)}{\|q-p\|}$.
Combining this with Lemma \ref{le:TermsSize} implies that
\begin{equation} \label{eq:NotInH0Cond}
\left(\BI +\frac{1}{2}e_{d+1}e_{d+2}i(p)\right) \beta C\ell_{d-1}^0 \alpha \left(\BI-\frac{1}{2}e_{d+1}e_{d+2}i(a)\right)\subseteq H_0
\end{equation}
if and only if $\beta C\ell_{d-1}^0 \alpha \subseteq H_0$. By lemma 5.1 we have that $x\in H_0$ if and only if $\beta^{-1} x \beta\in H_0$, so $\beta C\ell_{d-1}^0 \alpha \subseteq H_0$ if and only if $C\ell_{d-1}^0 \alpha\beta \subseteq H_0$.

Assume that $a-b = q-p$.
For an arbitrary $\beta \in \Spin{d}$ such that $\beta e_{d} \beta^{-1} = \frac{i(q-p)}{\|q-p\|}$,
set $\gamma = e_{d-1}e_d$ and $\alpha = \gamma \beta^{-1}$.
Since $\gamma$ is the product of two elements from $i(\SSS^{d-1})$, we have that $\gamma\in \Spin{d}$, which in turn implies that $\alpha\in \Spin{d}$.
We get that
\begin{align} \label{eq:longAlphaConj}
\alpha \left(\frac{i(b-a)}{\|b-a\|}\right)\alpha^{-1} = \gamma \beta^{-1}\left(\frac{i(b-a)}{\|b-a\|}\right)\beta \gamma^{-1} = -\gamma \beta^{-1}\left(\frac{i(q-p)}{\|q-p\|}\right)\beta \gamma^{-1} = -\gamma e_d \gamma^{-1} = e_d.
\end{align}
We can thus use these $\alpha$ and $\beta$ in \eqref{eq:NotInH0Cond}.
This implies that $C\ell_{d-1}^0 \alpha\beta = C\ell_{d-1}^0 e_{d-1}e_d \subseteq H_0$, which in turn implies that \eqref{eq:NotInH0Cond} is false.
Combining this with \eqref{FIntersectionNonZero} and with Lemma \ref{le:TapStructureRd} implies $T_{ap}\cap T_{bq} \subseteq H_0$.

Next, assume that $a-b\neq q-p$. For an arbitrary $\beta \in \Spin{d}$ such that $\beta e_{d} \beta^{-1} = \frac{i(q-p)}{\|q-p\|}$,
set $B = \beta^{-1}\frac{i(b-a)}{\|b-a\|}\beta$.
We have that $-e_d\neq B$, so we may set $\gamma = \frac{1}{\left\|e_d+B\right\|}e_d\left(e_d+B\right)$ and $\alpha = \gamma \beta^{-1}$.
Since $\gamma$ is the product of two elements from $i(\SSS^{d-1})$, we have that $\gamma\in \Spin{d}$, which in turn implies that $\alpha\in \Spin{d}$.
Performing a calculation similar to the one in the proof of lemma 5.4, we have that $\alpha \left(\frac{i(b-a)}{\|b-a\|}\right)\alpha^{-1} = e_d$.
We can thus use these $\alpha$ and $\beta$ in \eqref{FIntersectionNonZero}.

Let
\begin{align} \label{eq:xDef1Tap}
x = \left(\BI+\frac{1}{2}e_{d+1}e_{d+2}i(p)\right)\beta \BI \gamma \beta^{-1} \left(\BI-\frac{1}{2}e_{d+1}e_{d+2}i(a)\right).
\end{align}
Note that $x\in F_{ap}\cap F_{bq}$.
By Lemmas \ref{le:mTermsConj}, \ref{le:TermsedAed}, and \ref{le:TermsSize} we have that $x \notin H_0$.
Since $x$ is a product of elements of $\Spun{d}$, we have that $x\in \Spun{d}$.
Lemma \ref{le:TapStructureRd} implies $T_{ap}\cap T_{bq} = F_{ap}\cap F_{bq}\cap \Spun{d}$, so $x \in T_{ap}\cap T_{bq}$.
We conclude that $T_{ap}\cap T_{bq}\not\subseteq H_0$, which completes the proof.
\end{proof}

\begin{lemma} \label{le:LapbqDimRd}
If $T_{ap} \neq T_{bq}$ and $T_{ap}\cap T_{bq} \not \subseteq H_0$, then $\dim \left(L_{ap}\cap L_{bq}\right) = \binom{d-1}{2}$.
\end{lemma}
\begin{proof}
By the assumption $L_{ap} \cap L_{bq} \neq \emptyset$.
Let $v\in L_{ap} \cap L_{bq}$, and note that it suffices to prove that $\dim \left((L_{ap}-v)\cap(L_{bq}-v)\right) = \binom{d-1}{2}$.
By Lemma \ref{le:TauIntProject},
\begin{align*}
(L_{ap}-v)\cap(L_{bq}-v) + v &= L_{ap}\cap L_{bq} = \eta_d\left(F_{ap} \cap F_{bq}\cap H_1\right) = \eta'_d\left(F_{ap} \cap F_{bq}\cap H_0\right) + v.
\end{align*}

By the extension of Lemma \ref{le:FapIntersectionChar} to $\RR^d$, we have that $\dim\left(F_{ap} \cap F_{bq}\right) = \dim(C\ell_{d-1}) =  2^{d-2}$.
The assumption $T_{ap}\cap T_{bq} \not \subseteq H_0$ implies that $F_{ap} \cap F_{bq}$ properly intersects $H_0$.
This in turn implies $\dim\left(F_{ap} \cap F_{bq}\cap H_0\right) = 2^{d-2}-1$.
It remains to show that $\dim\left(F_{ap} \cap F_{bq}\cap H_0\cap \ker\left(\eta'_d\right)\right) = 2^{d-2}-\binom{d-1}{2}-1$.

For an arbitrary $\beta\in \Spin{d}$ that satisfies $\beta e_{d} \beta^{-1} = \frac{i(q-p)}{\|q-p\|}$, set $B= \beta^{-1}\frac{i(b-a)}{\|b-a\|}\beta$.
By Lemma \ref{le:FewBadPairs}, the assumption $T_{ap}\cap T_{bq}\not\subseteq H_0$ implies $a-b \neq q-p$,
which in turn implies that $B\neq -e_d$.
Let $\gamma = \frac{1}{\left\|e_n+B\right\|}e_n\left(e_n+B\right)$ and let $\alpha = \gamma \beta^{-1}$.
Since $\gamma$ is the product of two elements from $i(\SSS^{d-1})$, we have that $\gamma\in \Spin{d}$, which in turn implies that $\alpha\in \Spin{d}$.
By repeating the argument in \eqref{eq:longAlphaConj}, we get  that  $\alpha\frac{b-a}{||b-a||}\alpha^{-1} =  e_{d}$.
By the extension of Lemma \ref{le:FapIntersectionChar} to $\RR^d$, we have
\begin{equation} \label{eq:FIntForKernelBasis}
F_{ap}\cap F_{bq} = \left(\BI+\frac{1}{2}e_{d+1}e_{d+2}i(p)\right) \beta C\ell_{d-1}^0 \gamma \beta^{-1} \left(\BI-\frac{1}{2}e_{d+1}e_{d+2}i(a)\right).
\end{equation}

Consider the map $\tau_{ap}:Z_d^0 \to Z_d^0$ defined by
\[ \tau_{ap}(x)= \left(\BI+\frac{1}{2}e_{d+1}e_{d+2}i(p)\right)\beta x \alpha \left(\BI-\frac{1}{2}e_{d+1}e_{d+2}i(a)\right).\]
Since $\tau_{ap}^{-1}(x)=\beta^{-1}\left(\BI-\frac{1}{2}e_{d+1}e_{d+2}i(p)\right) x \left(\BI+\frac{1}{2}e_{d+1}e_{d+2}i(a)\right)\alpha^{-1}$, we note that $\tau_{ap}$ is a linear bijection.

We claim that $F_{ap}\cap F_{bq} \cap H_0 \cap \ker\left(\eta'_d\right)$ is generated by
\begin{align}
\bigg\{\tau(f) :\ f \mbox{ is an element of the standard basis of } C\ell_{d-1}^0 \text{ and an } m\text{-term for some } m\geq 4\bigg\}. \label{eq:KernelGeneratorM4}
\end{align}
Indeed, for any such $m$-term $f$, Lemmas \ref{le:mTermsConj}, \ref{le:TermsedAed}, and \ref{le:TermsSize} imply that
\[ \tau(f) = \left(\BI+\frac{1}{2}e_{d+1}e_{d+2}i(p)\right)\beta f \alpha \left(\BI-\frac{1}{2}e_{d+1}e_{d+2}i(a)\right)\in H_0\cap \ker(\eta'_d). \]
By \eqref{eq:FIntForKernelBasis}, this expression is also in $F_{ap}\cap F_{bq}$.

If $f \in Z_d^0$ contains a 0-term or a 2-term, then Lemmas \ref{le:mTermsConj}, \ref{le:TermsedAed}, and \ref{le:TermsSize} imply that $\tau_{ap}(f) \notin F_{ap}\cap F_{bq} \cap H_0\cap \ker(\eta'_d)$.
That is, if $g\in F_{ap}\cap F_{bq} \cap H_0\cap \ker(\eta'_d)$ then $\tau_{ap}^{-1}(g)$ contains no 0-term or 2-terms.
We conclude that \eqref{eq:KernelGeneratorM4} generates $F_{ap}\cap F_{bq} \cap H_0\cap \ker(\eta'_d)$.
Since $\tau(f)$ is a bijection, the set \eqref{eq:KernelGeneratorM4} is a linearly independent subset of $F_{ap}\cap F_{bq} \cap H_0\cap \ker(\eta'_d)$.
This implies that $\dim(F_{ap}\cap F_{bq} \cap H_0\cap \ker(\eta'_d)) = 2^{d-2} - \binom{d-1}{2} - 1$, which completes the proof.
\end{proof}

We are now ready to state the connection between the distinct distances problem and the flats $L_{ap}$.
Let $Q'$ be the set of quadruples $(a,p,b,q)\in \pts^4$ such that $T_{ap}\cap T_{bq}\not\subseteq H_0$.
In particular, note that $(a,p,b,q)\in Q'$ implies that $T_{ap}\cap T_{bq} \neq \emptyset$.

\begin{corollary} \label{co:TapTbqIntCharRd}
We have that $Q'\subset Q$ and $|Q'|\ge |Q|/2$.
\end{corollary}
\begin{proof}
Recall that a quadruple $(a,p,b,q)\in \pts^4$ is in $Q$ if and only if $T_{ap} \cap T_{bq} \neq \emptyset$.
Since $T_{ap}\cap T_{bq}\not\subseteq H_0$ implies $T_{ap}\cap T_{bq} \neq \emptyset$, we have that $Q'\subseteq Q$.
It remains to show that at least half of the quadruples of $Q$ are also in $Q'$.
Consider $T_{ap} \neq T_{bq}$ such that $T_{ap} \cap T_{bq} \subseteq H_0$.
By Lemma \ref{le:FewBadPairs} we have that $a-b = q-p$.
This implies that $b-a \neq p-q$, so $T_{bp}\cap T_{aq} \not \subseteq H_0$ (since $|ab| = |pq|$, we get that $T_{bp}\cap T_{aq} \neq \emptyset$).
That is, for every quadruple $(a,p,b,q)\in Q$ not in $Q'$ there exists a distinct quadruple $(b,p,a,q)$ that is in $Q'$.
\end{proof}

\parag{Flats in $\RR^{\binom{d+1}{2}}$ and in $\RR^{2d-1}$.}
We set
\[ \lines = \{L_{ap} :\ a,p\in \pts \text{ and } a\neq p \}.\]
Note that $\lines$ is a set of $\Theta(n^2)$ flats of dimension $\binom{d}{2}$ in $\RR^{\binom{d+1}{2}}$.
By Corollary \ref{co:TapTbqIntCharRd}, to get an asymptotic upper bound for the number of quadruples in $Q$ it suffices to derive an upper bound for the number of quadruples $(a,p,b,q)\in \pts^4$ such that $T_{ap}\cap T_{bq} \not\subseteq H_0$.
By Lemma \ref{le:LapbqDimRd} every such quadruple satisfies $\dim L_{ap}\cap L_{bq} = \binom{d-1}{2}$.
On the other hand, when $T_{ap}\cap T_{bq} \subseteq H_0$ we have that $L_{ap}\cap L_{bq}=\emptyset$.
Thus, it remains to derive an upper bound on the number of pairs of flats of $\lines$ that intersect (in a $\binom{d-1}{2}$-flat).

The proof of the following lemma is identical to the proof of Lemma \ref{le:3flatsRest}.

\begin{lemma} \label{le:3flatsRestRd}
(a) Every point of $\RR^{\binom{d+1}{2}}$ is contained in at most $n$ flats of $\lines$. \\
(b) Every hyperplane in $\RR^{\binom{d+1}{2}}$ contains at most $n$ flats of $\lines$.
\end{lemma}

Note that $\binom{d+1}{2}-\binom{d-1}{2} = 2d-1$ and that $\binom{d}{2}-\binom{d-1}{2} = d-1$.
Let $H_g$ be a generic $(2d-1)$-flat in $\RR^{\binom{d+1}{2}}$, in the sense that:
\begin{itemize}[topsep=0pt,itemsep=-1ex]
\item Every $\binom{d}{2}$-flat of $\lines$ intersects $H_g$ in a $(d-1)$-flat.
 \item Every $\binom{d-1}{2}$-flat of the form $L_{ap}\cap L_{bq}$ (with $a,b,p,q\in \pts$) intersects $H_g$ at a single point.
\end{itemize}

Let $\flats = \{L_{ap} \cap H_g :\, L_{ap}\in \lines\}$ and consider $H_g$ as $\RR^{2d-1}$.
Note that $\flats$ is a set of $\Theta(n^2)$ distinct $(d-1)$-flats.
Every two $(d-1)$-flats of $\flats$ are either disjoint or intersect in a single point.
By Lemma \ref{le:3flatsRestRd}, every point of $\RR^{2d-1}$ is incident to at most $n$ of the flats of $\flats$ and every hyperplane in $\RR^{2d-1}$ contains at most $n$ of the flats of $\flats$.

For every integer $k\ge 2$, let $m_k$ denote the number of points of $\RR^{2d-1}$ that are contained in exactly $k$ of the $(d-1)$-flats of $\flats$.
Similarly, let $m_{\ge k}$ denote the number of points of $\RR^{2d-1}$ that are contained in \emph{at least} $k$ of the $(d-1)$-flats of $\flats$.
Then $|Q'|$ is the number of pairs of intersecting $(d-1)$-flats of $\flats$, and
\[ |Q| \le 2|Q'|= 2\sum_{k=2}^n m_k \cdot 2\binom{k}{2} < 2\sum_{k=2}^n k^2 m_k = O\left( \sum_{k=1}^{\log n} 2^{2k} m_{\ge 2^k}\right). \]

If we had the bound $m_{\ge k} = O\left(\frac{n^{(4d-2)/d}}{k^{2+\eps}}\right)$ for some $\eps>0$, then the above would imply  $|Q|=O(n^{(4d-2)/d})$. This  would in turn imply that the points of $\pts$ span $\Omega\left(n^{2/d}\right)$ distinct distances.

An incidence result of Solymosi and Tao \cite{ST12} implies that the number of incidences between $m$ points and $n$ flats of dimension $d-1$ in $\RR^{2d-1}$, with every two flats intersecting in at most one point, is $O(m^{2/3+\eps'}n^{2/3}+m+n)$ (for any $\eps'>0$).
Every incidence bound of this form has a dual formulation involving $k$-rich points (for example, see \cite[Chapter 1]{ShefferBook}).
In this case, the dual bound is: Given $n^2$ flats of dimension $d-1$ in $\RR^{2d-1}$ such that every two intersect in at most one point, for every $k\ge 2$ the number of $k$-rich points is $O\left(\frac{n^{4/(1-\eps')}}{k^{3/(1-\eps')}}+\frac{n^2}{k}\right)$.
By taking $\eps'$ to be sufficiently small with respect to $\eps$, we obtain the bound $m_{\ge k} = O\left(\frac{n^{4+\eps}}{k^{3}}+\frac{n^2}{k}\right)$ for the number of $k$-rich points.
This bound is stronger than the required bound when $k=\Omega(n^{2/d+\eps})$.
That is, it remains to consider the case where $k=O(n^{2/d+\eps})$.

\section{The structure of the flats $L_{ap}$} \label{sec:structure}

In this section we study the structure of the $\binom{d}{2}$-flats $L_{ap}$ in $\RR^{\binom{d+1}{2}}$.
In particular, we derive the equations that define such a flat.
This structure is useful for deriving additional properties of the flats, which may be required for solving the incidence problem in Theorem \ref{th:DDreduction}.

Recall that we think of every coordinate of $\RR^{\binom{d+1}{2}}$ as corresponding to a 2-term in the standard basis of $Z_d^0$.
We denote the coordinate corresponding to $e_je_k$ as $x_{j,k}$, for every $1\le j < k \le d$.
Similarly, we denote the coordinate corresponding to $e_je_{d+1}e_{d+2}$ as $x_{j,d+1}$.
For $a\in \RR^d$, we denote by $a_j$ the $j$'th coordinate of $a$.

\begin{theorem} \label{th:FlatStructureRd}
Given $a,p \in \RR^d$, the flat $\eta_d(T_{ap}\cap \Spun{d}_+)$ is defined by the following system of $d$ equations in the coordinates of $\RR^{\binom{d+1}{2}}$.
\begin{align*}
a_1 - p_1 = & (a_2+p_2)x_{1,2}+(a_3+p_3)x_{1,3} + \cdots + (a_d+p_d)x_{1,d} + 2x_{1,d+1},\\
a_2 - p_2 = & -(a_1+p_1)x_{1,2}+(a_3+p_3)x_{2,3} + \cdots + (a_d+p_d)x_{2,d} + 2x_{2,d+1},\\
\vdots& \\
a_d - p_d = &-(a_1+p_1)x_{1,n}-(a_2+p_2)x_{2,d} - \cdots - (a_{d-1}+p_{d-1})x_{d-1,d} + 2x_{d,d+1}.
\end{align*}
\end{theorem}
\begin{proof}
In the following proof, every reference to orthogonal elements is with respect to the standard inner product $\langle \cdot,\cdot\rangle$ of $\RR^{2^d}$.
For a vector $v\in \RR^{2^d}$, we denote the dual of $v$ as $v^*$.
That is, $v^*$ is the map $v^*(u) = \langle v,u\rangle$.
For linear maps $f,g : \RR^{2^d} \to \RR^{2^d}$, we denote by $f^t(g)(v)$ the transpose $g(f(v))$.

Consider the linear map $\tau_{ap}: Z_d^0 \to Z_d^0$ defined by
\[ \tau_{ap}(x) = \left(\BI+\frac{1}{2}e_{d+1}e_{d+2}i(p)\right)x\left(\BI-\frac{1}{2}e_{d+1}e_{d+2}i(a)\right). \]
We also observe that
\begin{equation} \label{eq:TauAPInverse}
\tau^{-1}_{ap}(x) = \left(\BI-\frac{1}{2}e_{d+1}e_{d+2}i(p)\right)x\left(\BI+\frac{1}{2}e_{d+1}e_{d+2}i(a)\right).
\end{equation}
Thus, $\tau_{ap}$ is a linear bijection.

\begin{lemma} \label{le:OrthConditionTau}
For $u,w\in Z_d^0$, we have that $u$ is orthogonal to $\tau_{ap}(w)$ if and only if $u = \left((\tau_{ap}^{-1})^t \circ v^*\right)^*$ for some $v\in \RR^{2^d}$ orthogonal to $w$.
\end{lemma}
\begin{proof}
Let $v\in \RR^{2^d}$ be orthogonal to $w$.
We have that\footnote{Strictly speaking, $(u^*)^*$ is not equal to $u$.
With a slight abuse of notation, we apply here the natural isomorphism between the space $\left({\RR^{2^d}}^*\right)^*$ and $\R^{2^d}$. }
\begin{align*}
\langle\left((\tau_{ap}^{-1})^t \circ v^*\right)^*,\tau_{ap}(w)\rangle =\left(\left((\tau_{ap}^{-1})^t\circ v^*\right)^*\right)^*\left(\tau_{ap}(w)\right) &=\left((\tau_{ap}^{-1})^t \circ v^*\right)\left(\tau_{ap}(w)\right) \\
&= \left(v^*\circ \tau_{ap}^{-1}\right)\left(\tau_{ap}(w)\right) = v^*(w)=\langle v,w\rangle = 0.
\end{align*}
That is, $\left((\tau_{ap}^{-1})^t \circ v^*\right)^*$ is orthogonal to $\tau_{ap}(w)$, as required.

For the other direction, assume that $u$ is orthogonal to $\tau_{ap}(w)$ and note that
\begin{align*}
u = \left(u^*\right)^* = \left((\tau_{ap}^{-1})^t \left( ((\tau_{ap}^{-1})^t)^{-1}\circ u^*\right)\right)^*.
\end{align*}
That is, $u = \left((\tau_{ap}^{-1})^t \circ v^*\right)^*$ for $v=\left( ((\tau_{ap}^{-1})^t)^{-1} \circ u^*\right)^*$.
We also have that
\begin{align*}
\langle v,w \rangle = \left\langle \left(((\tau_{ap}^{-1})^t)^{-1}\circ u^*\right)^*,w\right\rangle &= \left(\left(((\tau_{ap}^{-1})^t)^{-1} \circ u^*\right)^*\right)^*(w) =\left(((\tau_{ap}^{-1})^t)^{-1}\circ u^*\right) (w)\\
&=\left(\tau_{ap}^t(u^*)\right) (w) =u^* \left(\tau_{ap}(w)\right) = \langle u,\tau_{ap}(w)\rangle = 0.
\end{align*}
\end{proof}

Let $V_d^0$ be the orthogonal complement of $C\ell_d^0$ in $Z_d^0$.
Note that every term of every element of $V_d^0$ contains $e_{d+1}e_{d+2}$.
Lemma \ref{le:OrthConditionTau} implies that $\left((\tau_{ap}^{-1})^t \circ \left(V_d^0\right)^*\right)^*$ is the orthogonal complement of $\tau_{ap}\left(C\ell_d^0\right)$.
Let $I_{2^{d-1}}$ be the $2^{d-1}\times 2^{d-1}$ identity matrix.
We can express $(\tau_{ap}^{-1})^t$ as a $2^d\times 2^d$ matrix of the form\footnote{To write this matrix, we must choose a specific ordering of the dual elements of the standard basis of $Z_d^0$.
As long as the elements dual to the basis elements involving $e_{d+1}e_{d+2}$ come after those dual to those that do not, the details of the ordering do not matter.}
\begin{equation} \label{eq:MatrixTauAP}
 \matwo{I_{2^{d-1}}}{C}{0}{I_{2^{d-1}}}.
\end{equation}
Indeed, recall that taking the transpose of a linear transformation corresponds to taking the transpose of the matrix of this transformation.
Note that the columns of \eqref{eq:MatrixTauAP} with index greater than $2^{d-1}$ form a basis of $(\tau_{ap}^{-1})^t \circ (V_d^0)^*$.

We denote the coordinates of $Z_d^0 \cong \RR^{2^d}$ as $y_1,\ldots,y_{2^d}$.
Let $(v_1, \ldots,v_{2^d})^* \in \left(Z_d^0\right)^*$ be one of the basis vectors of $(\tau_{ap}^{-1})^t \circ (V_d^0)^*$ that are columns of \eqref{eq:MatrixTauAP}.
We associate with this vector the equation $v_1y_1+\hdots + v_{2^d}y_{2^d} = 0$.
\ignore{For any particular values $x_1,\hdots, x_{2^n}$ of the variables, this is equivalent to the equation
\begin{align*}
\begin{pmatrix}v_1&\cdots&v_{2^n}\end{pmatrix}\vecthree{x_1}{\vdots}{x_{2^n}} = 0
\end{align*}
}
Let $S_{ap}$ be the system of $2^{d-1}$ homogeneous linear equations that are obtained in this way from the column vectors of \eqref{eq:MatrixTauAP} with index greater than $2^{d-1}$.
Since $\left((\tau_{ap}^{-1})^t \circ \left(V_d^0\right)^*\right)^*$ is the orthogonal complement of $\tau_{ap}\left(C\ell_d^0\right)$, the set of solutions to $S_{ap}$ is $\tau_{ap}\left(C\ell_d^0\right)$.

We construct a system of homogeneous linear equations $S'_{ap}$ by taking a subset of the equations of $S_{ap}$, as follows.
Let $v_1y_1+\hdots + v_{2^d}y_{2^d} = 0$ be an equation of $S_{ap}$.
We add this equation to $S'_{ap}$ if for every nonzero coefficient $v_j$ the variable $y_j$ corresponds either to a 0-term or to a 2-term.
Let $F'_{ap}$ be the set of solutions to the system $S'_{ap}$.

\begin{lemma} \label{le:SubsystemSuffices}
$\eta_d(F'_{ap}\setminus H_0) = \eta_d(\tau_{ap}(C\ell_d^0)\setminus H_0)$.
\end{lemma}
\begin{proof}
As stated above, the set of solutions to $S_{ap}$ is $\tau_{ap}(C\ell_d^0)$.
Since $S'_{ap} \subset S_{ap}$, we get that $\tau_{ap}(C\ell_d^0)\subset F'_{ap}$.
This immediately implies $\eta_d(\tau_{ap}(C\ell_d^0)\setminus H_0) \subseteq \eta_d(F'_{ap}\setminus H_0)$.
It remains to prove that $\eta_d(F'_{ap}\setminus H_0) \subseteq \eta_d(\tau_{ap}(C\ell_d^0)\setminus H_0)$.

For a linear equation $w_1y_1+\hdots + w_{2^d}y_{2^d} = 0$, we set $w = (w_1,\ldots,w_{2^d})^*$ and $u = (u_1,\ldots,u_{2^d})^* = \tau_{ap}^t \circ w$.
If $z\in Z_d^0$ is a solution to $w_1y_1+\hdots + w_{2^d}y_{2^d} = 0$ then $w^*$ is orthogonal to $z$, which in turn implies that $\left(\tau_{ap}^t \circ w\right)^*$ is orthogonal to $\tau_{ap}^{-1}(z)$.
That is, $\tau_{ap}^{-1}(z)$ is a solution to $u_1y_1+\hdots +u_{2^d}y_{2^d} = 0$.
Conversely, if $z\in Z_d^0$ is a solution to $u_1y_1+\hdots +u_{2^d}y_{2^d} = 0$ (that is, $u^*$ is orthogonal to $z$) then  $\tau_{ap}(z)$ is orthogonal to $\left((\tau_{ap}^{-1})^t \circ u\right)^* = \left((\tau_{ap}^t)^{-1} \circ u\right)^* = w^*$.
We conclude that $\tau_{ap}^{-1}$ is a bijection from the solutions to $w_1y_1+\hdots + w_{2^d}y_{2^d} = 0$ to the solutions to $u_1y_1+\hdots +u_{2^d}y_{2^d} = 0$.

Recall that every equation of $S_{ap}$ is defined by a dual vector $v\in (\RR^d)^*$ of the form $(\tau_{ap}^{-1})^t \circ  (\gamma e_{d+1}e_{d+2})^*$, where $\gamma e_{d+1}e_{d+2}$ is a basis vector of $V^0_d$ (that is, $\gamma$ is in the standard basis of $C\ell_d^1$).
Every non-zero term of such a vector corresponds to a 0-term or a 2-term if and only if $v^*\in \RR^d$ is orthogonal to every vector corresponding to an $m$-term for some $m \ge 4$.
Let $w\in \RR^d$ be a vector that corresponds to such an $m$-term.
If $\gamma = e_j$ for some $1\le j \le d$, then Lemma \ref{le:TermsSize} implies that $\tau^{-1}_{ap}(w)$ is orthogonal to $\gamma e_{d+1}e_{d+2}$.
Lemma \ref{le:OrthConditionTau} states that $((\tau^{-1}_{ap})^t \circ (e_j e_{d+1}e_{d+2})^*)^*$ is orthogonal to $\tau_{ap}((\tau_{ap}^{-1})^t(w))= w$.
That is, when $\gamma = e_j$ the equation defined by $v$ is in $S'_{ap}$.

Next, assume that $\gamma e_{d+1}e_{d+2}$ is an $m$-term with $m \ge 4$, and write $u=\gamma e_{d+1}e_{d+2}$.
Lemma \ref{le:TermsSize} implies that $\tau_{ap}(u)$ contains neither 2-terms nor a 0-term.
If $((\tau_{ap}^{-1})^t \circ u^*)^*$ is orthogonal to $\tau_{ap}(u)$ then by (the other direction of) Lemma \ref{le:OrthConditionTau} we get that $u$ is orthogonal to $u$.
This contradiction implies that $((\tau_{ap}^{-1})^t \circ u^*)^*$ is not orthogonal to $\tau_{ap}(u)$, so in this case the equation defined by $v$ is not in $S'_{ap}$.

Combining the two preceding paragraphs implies that the equations of $S'_{ap}$ are determined by the vectors $(\tau_{ap}^{-1})^t \circ \left((e_je_{d+1}e_{d+2})^*\right)$ for $1\le j\le d$.
It follows that the equations of $S'_{00}$ are obtained from those defining $S'_{ap}$ by applying $\tau_{ap}^t$ to the coefficient vectors.
By the second paragraph of this proof, for every $v\in F'_{ap}$ we have that $\tau_{ap}^{-1}(v)\in F'_{00}$.

When $a= p=0$, we have that \eqref{eq:MatrixTauAP} is the identity matrix.
Thus, each equation of $S_{00}$ consists of a single term.
This in turn implies that $F'_{00}$ is the subspace defined by having 0 in every coordinate that corresponds to a 2-term of the form $e_je_{d+1}e_{d+2}$ (where $1\le j \le d$).
For $v\in F'_{ap}\setminus H_0$, we obtain that $\tau_{ap}^{-1}(v)$ contains no terms of the form $e_je_{d+1}e_{d+2}$.
By Lemmas \ref{le:JdTwoVectorsRule} and \ref{le:JdSurjective}, there is a unique $x \in J_d$ with the property that $x - \tau_{ap}^{-1}(v)$ contains no 0-term and no 2-terms.
Note that $x$ also contains no terms of the form $e_je_{d+1}e_{d+2}$, so Lemma \ref{le:SpinSurject} implies that $x \in G_d$.
Since $x\in C\ell_d^0$, we have that $\tau_{ap}(x) \in \tau_{ap}\left(C\ell_d^0\right)$.
By Lemma \ref{le:TermsSize}, the expression $\tau_{ap}\left(x- \tau_{ap}^{-1}(v)\right)$ also contains no 0-term and no 2-terms, so $\eta_d(\tau(x)) = \eta_d(v)$.
That is, there exists $\tau_{ap}(x) \in \tau_{ap}\left(C\ell_d^0\right)$ such that $\eta_d(\tau_{ap}(x)) = \eta_d(v)$.
Since $v\notin H_0$, we have that $\tau_{ap}^{-1}(v)\notin H_0$, which in turn implies that $x\notin H_0$ and that $\tau(x) \notin H_0$.
This implies that $\eta_d(F'_{ap}\setminus H_0) \subseteq \eta_d(\tau_{ap}(C\ell_d^0)\setminus H_0)$ and completes the proof.
\end{proof}

By Lemma \ref{le:SubsystemSuffices}, to complete the proof of Theorem \ref{th:FlatStructureRd} it suffices to study $\eta_d(F'_{ap}\setminus H_0)$.
We move from the coordinate system $y_j$ to the coordinate system $x_{j,k}$, as described before the statement of the theorem.
We denote by $x_1$ the coordinate corresponding to the coefficient of $\BI$ (that is, $y_1$).

We now study the equations of $S'_{ap}$.
As discussed in the proof of Lemma \ref{le:SubsystemSuffices}, these equations correspond to the dual vectors $(\tau_{ap}^{-1})^t \circ (e_je_{d+1}e_{d+2})^*$ for $1\le j\le d$.
If $e_j e_{d+1}e_{d+2}$ is the $k$'th element in our ordering of the basis of $Z^0_d$, then $(\tau_{ap}^{-1})^t \circ (e_je_{d+1}e_{d+2})^*$ is the $k$'th column of the matrix \eqref{eq:MatrixTauAP}.
Since the transpose of a linear transformation corresponds to the transpose of the matrix of the transformation, the above is also the $k$'th row of the matrix of $\tau_{ap}^{-1}$.
To get this row, we apply $\tau_{ap}^{-1}$ to the basis vectors of $Z^0_d$ and then keep the coefficient of $e_j e_{d+1}e_{d+2}$ (recall that $\tau_{ap}^{-1}$ is defined in \eqref{eq:TauAPInverse}).
The only basis vectors of $Z^0_d$ for which this coefficient is nonzero are $\BI$ and 2-terms involving $e_j$.
Repeating this process for every $1\le j\le d$ leads to the following system.
\begin{align*}
(a_1 - p_1)x_1 = & (a_2+p_2)x_{1,2}+(a_3+p_3)x_{1,3} + \cdots + (a_d+p_d)x_{1,d} + 2x_{1,d+1},\\
(a_2 - p_2)x_1 = & -(a_1+p_1)x_{1,2}+(a_3+p_3)x_{2,3} + \cdots + (a_d+p_d)x_{2,d} + 2x_{2,d+1},\\
\vdots& \\
(a_d - p_d)x_1 = &-(a_1+p_1)x_{1,d}-(a_2+p_2)x_{2,d} - \cdots - (a_{d-1}+p_{d-1})x_{d-1,d} + 2x_{d,d+1}.
\end{align*}

Recall from the beginning of Section \ref{sec:DistancesRd} that $\eta_d = \pi' \circ \pi_1$.
Since the above is a system of homogeneous linear equations, $F'_{ap}$ is spanned by lines that are incident to the origin.
This implies that $\pi(F'_{ap}\setminus H_0) = \pi(F'_{ap}\cap H_1)$.
Thus, $\pi(F'_{ap}\setminus H_0)$ is the set of solutions to the system obtained by setting $x_1 = 1$:
\begin{align} \
(a_1 - p_1) = & (a_2+p_2)x_{1,2}+(a_3+p_3)x_{1,3} + \cdots + (a_d+p_d)x_{1,d} + 2x_{1,d+1}, \nonumber \\
(a_2 - p_2) = & -(a_1+p_1)x_{1,2}+(a_3+p_3)x_{2,3} + \cdots + (a_d+p_d)x_{2,d} + 2x_{2,d+1}, \nonumber \\
\vdots&  \label{eq:SystemF'ap} \\
(a_d - p_d) = &-(a_1+p_1)x_{1,d}-(a_2+p_2)x_{2,d} - \cdots - (a_{d-1}+p_{d-1})x_{d-1,d} + 2x_{d,d+1}.  \nonumber
\end{align}

Since none of the variables $x_{j,k}$ correspond to elements of $\RR^{2^d-1}$ that are in the kernel of $\pi'$, we get that $\eta_d(F'_{ap})$ is the solution set of \eqref{eq:SystemF'ap}.
\end{proof}

\section{Properties of the 2-flats in $\RR^5$} \label{sec:FlatsR5}

In this section we study the 2-flats in $\RR^5$ that are obtained from our reduction of the three-dimensional distinct distances problem.
In particular, we show how one can bound the number of 2-flats contained in constant-degree three- and four-dimensional varieties.
We also show how one can bound the number of 2-flats that have a one-dimensional intersection with a constant-degree two-dimensional variety.
Deriving these results requires several definitions and tools from Algebraic Geometry, and these are described in Section \ref{ssec:VarPrelim}.

\subsection{Algebraic Geometry preliminaries} \label{ssec:VarPrelim}

In the following, $\FF$ could be taken to be either $\CC$ or $\RR$.
The \emph{variety} defined by the polynomials $f_1,\ldots,f_k\in \FF[x_1,\ldots,x_d]$ is
\[ \vb(f_1,\ldots,f_k) = \left\{(a_1,\ldots,a_d)\in \FF^d  : f_j(a_1,\ldots,a_d)=0  \text{ for all }  1\le j \le k\right\}. \]

There are several non-equivalent definitions for the degree of a variety in $\RR^d$.
For our purposes, we say that a variety $U\subset \RR^d$ is \emph{defined at degree} $D$ if
\begin{equation*}
D = \min_{f_1,\ldots,f_k \in \RR[x_1,\ldots,x_d] \atop \vb(f_1,\ldots,f_k) = U} \max_{1 \le i \le k} \deg f_i.
\end{equation*}
In other words, $D$ is the minimum integer such that $U$ can be defined by a finite set of polynomials of degree at most $D$.
When using any reasonable notion for the degree of $U$, this degree is bounded by a constant if and only if $D$ is bounded by a constant.
In light of this, we say that a real variety $U$ is a \emph{constant-degree variety} when the degree at which $U$ is defined is bounded by a constant.

A variety $U\subseteq \FF^d$ is \emph{reducible} if there exist two proper subvarieties $U', U'' \subset U$
such that $U = U' \bigcup U''$. Otherwise, $U$ is \emph{irreducible}.
An \emph{irreducible component} of $U$ is an irreducible variety that is contained in $U$, and not contained in any other irreducible subvariety of $U$.

\begin{lemma} \label{le:BoundedNumComponents}
Let $U\subset \RR^d$ be a variety defined at degree $k$.
Then the number of irreducible components of $U$ is $O_{d,k}(1)$.
\end{lemma}

Intuitively, we say that a variety $U\subset \RR^d$ has dimension $k$ if there exists a subset of $U$ that is homeomorphic to the open $k$-dimensional cube, but no subset of $U$ is homeomorphic to an open cube of a larger dimension.
For more information about varieties in $\RR^d$ and a more precise definition of dimension, see for example \cite{BCR98}.

\parag{Singular points, regular points, and tangent flats.}
The \emph{ideal} of a variety $U\subseteq \RR^d$, denoted $\ib(U)$, is the set of polynomials in $\RR[x_1,\ldots,x_d]$ that vanish on every point of $U$.
We say that a set of polynomials $f_1,\ldots,f_\ell \in \RR[x_1,\ldots,x_d]$ \emph{generate} $\ib(U)$ if every element of $\ib(U)$ can be written as $\sum_{j=1}^\ell f_j g_j$ for some $g_1,\ldots,g_\ell \in \RR[x_1,\ldots,x_d]$.
We also write $\langle f_1,\ldots,f_\ell \rangle = \ib(U)$ to state that $f_1,\ldots,f_\ell$ generate $\ib(U)$.

The \emph{Jacobian matrix} of a set of polynomials $f_1,\ldots,f_k \in \RR[x_1,\ldots,x_d]$ is

\[ \jb_{f_1,\ldots,f_k} = \left( \begin{array}{cccc}
\frac{\partial f_1}{\partial x_1} & \frac{\partial f_1}{\partial x_2} & \cdots & \frac{\partial f_1}{\partial x_d} \\[2mm]
\frac{\partial f_2}{\partial x_1} & \frac{\partial f_2}{\partial x_2} & \cdots & \frac{\partial f_2}{\partial x_d} \\[2mm]
\cdots & \cdots & \cdots & \cdots \\[2mm]
\frac{\partial f_k}{\partial x_1} & \frac{\partial f_k}{\partial x_2} & \cdots & \frac{\partial f_k}{\partial x_d} \end{array}\right)\]

Consider a variety $U\subset\RR^d$ of dimension $k$, and let $f_1,\ldots,f_\ell \in \RR[x_1,\ldots,x_d]$ satisfy $\langle f_1,\ldots,f_\ell \rangle = \ib(U)$.
We say that $p\in U$ is a \emph{singular} point of $U$ if $\rank \jb(p) <d-k$.
A point of $U$ that is not singular is said to be a \emph{regular} point of $U$.
We denote the set of singular points of $U$ as $U_\text{sing}$, and the set of regular points of $U$ as $U_\text{reg}$.
A $k$-dimensional variety has a unique well-defined tangent $k$-flat at every regular point.
We denote the tangent $k$-flat at $p\in U$ as $T_p U$, and think of it as a linear subspace (that is, as incident to the origin).
At singular points of a real variety, a unique well-defined tangent flat may or may not exist.

\begin{theorem} \label{th:Singular}
Let $U\subset\RR^d$ be a variety defined at degree $k$ and dimension $d'$.
Then $U_\text{sing}$ is a variety of dimension smaller than $d'$ and is defined at degree $O_{k,d}(1)$.
\end{theorem}

References for the above claims and additional information can be found, for example, in \cite{BCR98}.

\parag{Complexification.}
Given a variety $U \subset \RR^d$, the \emph{complexification} $U^*\subset \CC^d$ of $U$ is the smallest complex variety that contains $U$, in the sense that any other complex variety that contains $U$ also contains $U^*$ (for example, see \cite{RV02,Whit57}).
As shown in \cite[Lemma 6]{Whit57}, such a complexification always exists, and $U$ is precisely the set of real
points of $U^*$.

As shown in \cite[Section 10]{Whit57}, there is a bijection between the
irreducible components of $U$ and the irreducible components of $U^*$,
such that each real component is the real part of its corresponding complex component.
In particular, the complexification of an irreducible variety is irreducible.
The real dimension of a real irreducible component in $\RR^d$ is equal to the complex dimension of the corresponding complex component in $\CC^d$.

\parag{Constructible sets, semi-algebraic sets, and projections.}
As before, $\FF$ could be taken to be either $\CC$ or $\RR$.
If $U\subset \FF^d$ is a set, the \emph{Zariski closure} $\overline{U}$ is the smallest variety in $\FF^d$ that contains $U$.
A set $X\subset\FF^d$ is \emph{constructible} if there exist non-empty varieties $X_1,\ldots,X_\ell \subset \FF^d$ such that $\dim X_{j+1}<\dim X_j$ for every $1\le j < \ell$ and
\begin{equation}\label{eq:Const}
 X=\Big(\big((X_1\backslash X_2) \cup X_3\big)\backslash X_4\ldots\Big).
\end{equation}

We define $\dim(X) = \dim(\overline{X})=\dim(X_1)$.
We define the \emph{complexity} of $X$ to be $\min(\deg(X_1)+\deg(X_2)+\ldots+\deg(X_\ell)),$ where the minimum is taken over all representations of $X$ of the form \eqref{eq:Const}.
This definition is not standard.
However, since we are interested only in constructible sets of bounded complexity, any reasonable definition of complexity would work equally well.
For further details, see for example \cite[Chapter 3]{Harris92}.

A \emph{semi-algebraic set} in $\RR^d$ is the set of points in $\RR^d$ that satisfy a given finite Boolean combination
of polynomial equations and inequalities in $d$ coordinates.
Every constructible set in $\RR^d$ is semi-algebraic.
On the other hand, a semicircle in $\RR^2$ is semi-algebraic but not constructible.

Let $S\subset\RR^d$ be a semi-algebraic set defined by a Boolean combination of equations and inequalities involving the polynomials $f_1,\ldots,f_k\in \RR[x_1,\ldots,x_d]$ (that is, the $j$'th equation or inequality has zero on one side and $f_j$ on the other).
The dimension of $S$ is the dimension of the real variety $\overline{S}$.
The complexity of $S$ is $\min\{\deg f_1 + \cdots + \deg f_k\}$, where the minimum is taken over all Boolean combinations that define $S$.
Note that the degree at which the variety $\overline{S}$ is defined is at most the complexity of $S$.

One can also include the quantifiers $\forall$ and $\exists$ in the definition of a semi-algebraic set, each quantifying an additional variable that is not a coordinate of the points in the set.
For every definition of a semi-algebraic set using quantifiers, there exists a definition that does not use quantifiers.
For example, the formula $\forall t : (t> 0) \vee (x+y>t)$ defines an open half-plane in $\RR^2$, which can easily be defined without the $\forall$ quantifier.
In the above definition of the complexity, one may only use definitions of $S$ that do not include such quantifiers.
For the following, see for example \cite[Section 11.3]{BPR07}.

\begin{lemma} \label{le:QuantElim}
Let $S\subset \RR^d$ be a semi-algebraic set using $k$ quantified variables, and $s$ polynomials of degree at most $D$.
Then $S$ is of complexity $O_{k,s,D,d}(1)$.
\end{lemma}

Both in $\RR^d$ and in $\CC^d$, the projection of a variety is not necessarily a variety.
In $\RR^d$, the projection of a constructible set is not necessarily constructible.
The following result states properties that are satisfied by every projection.
For part (a), see for example \cite[Theorem 3.16]{Harris92} (this reference only says that the projection of a constructible set is constructible. However, the proof is constructive and thus gives us a bound on the complexity of the projection.)
Part (b) is implied by Lemma \ref{le:QuantElim}, noting that the projection of a semi-algebraic set can be obtained by adding $\exists$ quantifiers to its definition.

\begin{theorem} \label{th:Projection} $\quad$ \\
(a) Let $X\subset \CC^d$ be a constructible set of dimension $d'$ and complexity $k$.
Let $\pi:\CC^d \to \CC^e$ be a projection on $e$ out of the $d$ coordinates of $\CC^d$.
Then $\pi(X)$ is a constructible set of dimension at most $d'$ and complexity $O_{k,d}(1)$. \\
(b) Let $U\subset \RR^d$ be a semi-algebraic set of dimension $d'$ and complexity $k$.
Let $\pi:\RR^d \to \RR^e$ be a projection on $e$ out of the $d$ coordinates of $\RR^d$.
Then $\overline{\pi(U)}$ is a variety of dimension at most $d'$ and is defined at degree $O_{k,d}(1)$.
\end{theorem}

For $d> d'$, let $X\subset\CC^{d}$ be a constructible set and let $Y\subset\CC^{d^\prime}$ be an irreducible variety.
Let $\pi\colon \CC^{d}\to\CC^{d^\prime}$ be a projection onto $d^\prime$ coordinates.
We say that $\pi\colon X\to Y$ is \emph{dominant} if $\overline{\pi(X)}=Y$.
The following is a corollary of Chevalley's upper semi-continuity theorem (for example, see \cite[Corollary 11.13]{Harris92} and the paragraph following it;  for the claim that the set is constructible, see also \cite[Theorem 3.16]{Harris92}).

\begin{theorem} \label{th:upperSemiContThm}
Let $X \subset \CC^d$ and $Y \subset \CC^{d'}$ be irreducible varieties each of degree at most $k$, and suppose $\pi : X \to Y$ is dominant.
Then exists a variety $Y'\subset \CC^d$ of degree $O_k(1)$ such that $\dim Y'< \dim Y$ and for every $y\in Y\setminus Y'$ we have that $\pi^{-1}(y)$ is a constructible set of dimension $\dim X -\dim Y$ and complexity $O_{k,d}(1)$.
\end{theorem}

We require a real variant of Theorem \ref{th:upperSemiContThm}.

\begin{corollary} \label{co:RealUpperSemi}
Let $U$ be a variety of dimension $d$ in $\RR^6$, let $\pi:\RR^6\to \RR^3$ be the projection on the first three coordinates, and let $U_3=\overline{\pi(U)}$ be of dimension $d_3$.
Then there exists a variety $W\subset \RR^3$ defined at degree $O_k(1)$ such that $\dim W< d_3$ and for every $u\in U_3\setminus W$ we have that $\pi^{-1}(u)$ is a constructible set of dimension at most $d -d_3$ and complexity $O_{k}(1)$.
\end{corollary}
\begin{proof}
Consider the complexification $U^*$ of $U$ and the complexification $U_3^*$ of $U_3$.
Note that $U^*$ is of dimension $d$ and that $U^*_3$ is of dimension $d_3$.
We extend the projection $\pi: \RR^6 \to \RR^3$ to $\pi: \CC^6 \to \CC^3$.
As before, this is the projection on the first three coordinates.

Set $U' = \overline{\pi(U^*)}$.
By Theorem \ref{th:Projection}(a), the variety $U'$ has degree $O_k(1)$.
Since $U^*_3$ is the smallest complex variety containing $U_3$, we have that $U^*_3 \subseteq U'$.
Consider the cylindrical variety $C=\pi^{-1}(U^*_3) \subset \CC^6$, and note that $U$ is contained in the real part of $C$.
Since $U^*$ is the smallest variety in $\CC^6$ that contains $U$, we have that $U^*\subseteq C$.
This in turn implies that $U'\subseteq U^*_3$, so $U' = U^*_3$.
In particular, $\dim U' = d_3$.

By Theorem \ref{th:upperSemiContThm}, there exists a variety $W^*\subset \CC^3$ of degree $O_k(1)$ such that $\dim W^*< d_3$ and for every $u\in U'\setminus W^*$ we have that $\pi^{-1}(u)$ is a constructible set of dimension $d -d_3$ and complexity $O_k(1)$.
We set $W\subset \RR^3$ to be the real part of $W^*$, and note that $\dim W <d_3$.
Since $(U_3\setminus W) \subset (U'\setminus W^*)$, for every $u\in U_3\setminus W$ we have that $\pi^{-1}(u)\subset \RR^6$ is a constructible set of dimension at most $d -d_3$ and complexity $O_k(1)$.
\end{proof}

\subsection{Flats in $\RR^5$.}

Theorem \ref{th:FlatStructureRd} implies the following for the case of distinct distances in $\RR^3$.
Given two points $a=(a_1,a_2,a_3)$ and $p=(p_1,p_2,p_3)$ in $\RR^3$,
the corresponding 3-flat $L_{ap}\subset \RR^6$ is defined by
\begin{align}
-(a_1+p_1) x_2  - (a_2+p_2) x_3  + 2x_6  &= a_3  - p_3, \nonumber \\
-(a_1+p_1) x_1+ (a_3+p_3) x_3  + 2x_5  &= a_2  - p_2, \nonumber \\
(a_2+p_2) x_1  + (a_3+p_3) x_2   + 2x_4  &= a_1-p_1. \label{eq:3FlatDef}
\end{align}
Note that $\{L_{ap} :\ a,p\in\RR^3 \}$ is a six-dimensional family of 3-flats in $\RR^6$.

Let $\pts$ be a set of $n$ points in $\RR^3$, and let $H$ be a hyperplane in $\RR^6$, chosen generically with respect to $\pts$.
For $a,p\in\RR^3$ we write $F_{ap} = L_{ap}\cap H$.
We consider the sets
\[ \flats = \{F_{ap} :\ a,p\in\RR^3 \} \qquad \text{ and } \qquad \flatss = \{F_{ap} :\ a,p\in\pts \}. \]
Since $H$ is chosen generically, $\flatss$ is a set of $n^2$ distinct 2-flats in $H$.
We think of $H$ as $\RR^5$, so $\flatss$ becomes a set of 2-flats in $\RR^5$.
As shown in Section \ref{sec:DDR3}, every pair of flats in $\flatss$ intersect in at most one point.

\begin{theorem} \label{th:PlanesIn3d}
Let $U$ be an irreducible three-dimensional variety in $\RR^5$ defined at degree $k$.
Then either $U$ contains $O_k\left(n^{2/3}\right)$ flats of $\flatss$ or there exists a curve in $\RR^3$ defined at degree $O_k(1)$ that contains $\Omega_k\left(n^{2/3}\right)$ points of $\pts$.
\end{theorem}
It is not difficult to show that $n^{2/3}$ points on a constant-degree curve in $\RR^3$ span $\Omega\left(n^{2/3}\right)$ distinct distances.
This is exactly the conjectured number of distances in $\RR^3$, so we may assume that no constant-degree curve in $\RR^3$ contains $n^{2/3}$ points of $\pts$.
Then, Theorem \ref{th:PlanesIn3d} implies that every constant-degree three-dimensional variety in $\RR^5$ contains $O(n^{2/3})$ flats of $\flatss$.
\begin{proof}[Proof of Theorem \ref{th:PlanesIn3d}.]
For any $a,p\in \RR^3$, by \eqref{eq:3FlatDef} we have the parametrization
\begin{align}
L_{ap} = \bigg\{(s,t,r,\big(a_1-p_1  - (a_2+&p_2) s  - (a_3+p_3) t\big)/2 , \big(a_2  - p_2 + (a_1+p_1) s - (a_3+p_3) r\big)/2, \nonumber \\
&\big(a_3  - p_3 + (a_1+p_1) t  + (a_2+p_2) r\big)/2 )\in \RR^6 \ :\ s,t,r\in \RR \bigg\}. \label{eq:LapParam}
\end{align}
To parameterize $L_{ap}\cap H$, we isolate $x_3$ in the linear equation defining $H$ and use this to eliminate the parameter $r$ (since $H$ is generic, its defining equation contains $x_3$).
This parametrization of $L_{ap}\cap H$ consists of five linear functions in the two variables $s,t\in \RR$, with coefficients that are polynomials of degree at most two in the coordinates of $a$ and $p$.

We identify $H$ with $\RR^5$.
Equivalently, let $\pi_H: H \to \RR^5$ be a the map that takes $H$ to $\RR^5$.
Since $\pi_H$ can be seen as a translation followed by a rotation and a projection, we can write $\pi_H$ as five linear polynomials in $x_1,\ldots,x_6$.
Combining this with the above parametrization, we obtain a parametrization of $F_{ap} = \pi_H(L_{ap}\cap H)$ using five linear functions in the two variables $s,t\in \RR$ and coefficients that are polynomials of degree at most two in the coordinates of $a$ and $p$.

Let $f\in \RR[x_1,\ldots, x_5]$ be a polynomial of degree $2k$ such that $\vb(f) = U$ (if $U$ is defined as $\vb(f_1,\ldots,f_m)$ where each $f_j$ is of degree at most $k$, then we take $f= f_1^2 +\cdots + f_m^2$).
We think of $f\mid_{\pi_H(L_{ap}\cap H)}$ as a polynomial of degree at most $2k$ in $\RR[s,t]$ and coefficients that depend on the coordinates of $a$ and $p$.
Note that $F_{ap} \subset U$ if and only if $f|_{\pi_H(L_{ap}\cap H)}$ is identically zero.
That is, if and only if the coefficient of every monomial of $f|_{\pi_H(L_{ap}\cap H)}$ is zero.
There are $O_k(1)$ such monomials, and the coefficient of each is a polynomial of degree at most $4k$ in the coordinates of $a$ and $p$.
This implies that the set
\[ \flatsu = \left\{(a,p)\in \RR^6 :\ F_{ap} \subset U\right\} \]
is a variety defined at degree $O_k(1)$.
We use the notation $\flatsu$ to refer both to the above set of points in $\RR^6$ and to the set of corresponding 2-flats in $\RR^5$.

Let $u$ be a regular point of $U$, and let $F_{ap},F_{a'p'}\subset U$ be 2-flats of $\flats$ incident to $u$.
Since any pair of 2-flats of $\flats$ intersect in at most one point, we have $F_{ap}\cap F_{a'p'} = \{u\}$, so $T_u F_{ap}$ and $T_u F_{a'p'}$ span a 4-flat in $\RR^5$.
This is impossible, since both $T_u F_{ap}$ and $T_u F_{a'p'}$ are contained in the 3-flat $T_u U$.
This contradiction implies that every regular point of $U$ is incident to at most one 2-flat of $\flatsu$.
By Theorem \ref{th:Singular}, the set of singular points $U_\text{sing}$ is a two-dimensional variety defined at degree $O_k(1)$.
Thus, the number of 2-flats of $\flats$ contained in $U_\text{sing}$ is $O_k(1)$, and in particular there are $O_k(1)$ flats of $\flatss$ in $U_\text{sing}$.
Every 2-flat of $\flatss$ that is not contained in $U_\text{sing}$ intersects $U_\text{sing}$ in a variety of dimension at most one.
That is, excluding $O_k(1)$ flats, every flat of $\flatsu$ intersects $U_\text{reg}$ in a constructible set of dimension two.
If $\flatsu$ is of dimension at least two then $U$ contains a two-dimensional union of disjoint two-dimensional constructible sets, which in turn implies that $U$ is of dimension at least four.
This contradicts the assumption that $U$ is three-dimensional, so $\flatsu$ is of dimension at most one.

Let $\pi_1: \RR^6 \to \RR^3$ be the projection on the first three coordinates and let $\pi_2: \RR^6 \to \RR^3$ be the projection on the last three coordinates.
That is, for points $a,p\in \RR^3$ we have $\pi_1(a,p) = a$ and $\pi_2(a,p) = p$.
By Theorem \ref{th:Projection}(b), the variety $\gamma_1 = \overline{\pi_1(\flatsu)} \subset \RR^3$ is defined at degree $O_k(1)$ and of dimension at most one.
We symmetrically define $\gamma_2 = \overline{\pi_2(\flatsu)}$.

Set $\pts_1 = \pts\cap \gamma_1$ and $\pts_2 = \pts\cap \gamma_2$.
If $|\pts_1| = \Omega\left(n^{2/3}\right)$ then we are done, since we found a constant-degree curve in $\RR^2$ containing many points of $\pts$.
We may thus assume that $|\pts_1| = O\left(n^{2/3}\right)$, and symmetrically that $|\pts_2| = O\left(n^{2/3}\right)$.
If $\gamma_1$ is of dimension zero, then by Lemma \ref{le:BoundedNumComponents} it is a set of $O_k(1)$ points.
Since $|\pts_2| = O\left(n^{2/3}\right)$, we get that $O_k\left(n^{2/3}\right)$ flats of $\flatss$ are contained in $U$.
This completes the proof, so we may assume that $\gamma_1$ is of dimension one.

By Corollary \ref{co:RealUpperSemi}, excluding $O_k(1)$ exceptional points, for every $a\in \gamma_1$ there are $O_k(1)$ points $w\in \flatsu$ such that $\pi_1(w) = a$.
Since $|\pts_2| = O\left(n^{2/3}\right)$, the exceptional points correspond to $O_k\left(n^{2/3}\right)$ flats of $\flatss$ in $U$.
Since $|\pts_1| = O\left(n^{2/3}\right)$, the non-exceptional points also correspond to $O_k\left(n^{2/3}\right)$ flats of $\flatss$ in $U$.
We conclude that $|\pts^2\cap \flatsu| = O_k\left(n^{2/3}\right)$, which completes the proof.
\end{proof}

We now study the number of 2-flats of $\flatss$ in a four-dimensional constant-degree variety in $\RR^5$.

\begin{theorem}
Let $\pts$ be a set of $n$ points in $\RR^3$ and let $U$ be an irreducible four-dimensional variety in $\RR^5$ defined at degree $k$.
Then either $U$ contains $O_k\left(n^{4/3}\right)$ flats of $\flatss$ or there exists a surface in $\RR^3$ that contains $\Omega_k\left(n^{2/3}\right)$ points of $\pts$.
\end{theorem}
\begin{proof}
The case where $U$ is a hyperplane was already handled in Section \ref{sec:SpunD}, so we may assume that $U$ is not a hyperplane.
We begin by imitating the proof of Theorem \ref{th:PlanesIn3d}.
As in that proof, we define
\[ \flatsu = \left\{(a,p)\in \RR^6 :\ F_{ap} \subset U\right\}, \]
and note that $\flatsu$ is a variety definted at degree $O_k(1)$.

By Theorem \ref{th:Singular}, the set of singular points $U_\text{sing}$ is a three-dimensional variety defined at degree $O_k(1)$.
By Lemma \ref{le:BoundedNumComponents}, the set $U_\text{sing}$ consists of $O_k(1)$ irreducible components.
We apply Theorem \ref{th:PlanesIn3d} to each of these components, obtaining that either there exists a one-dimensional variety containing $\Omega\left(n^{2/3}\right)$ points of $\pts$, or that the total number of flats from $\flatsu$ contained in $U_\text{sing}$ is $O\left(n^{2/3}\right)$.
We may assume that we are in the latter case, since otherwise we are done.

Let $w$ be a regular point of $U$, and let $\flats_w$ be a set of 2-flats of $\flats$ that are contained in $U$ and incident to $w$.
Note that for every $F_{ap} \in\flats_w$ we have $F_{ap} \in U \cap (w+ T_w U)$.
Since $U$ is not a hyperplane, the intersection $C = U \cap (w+ T_w U)$ is a variety defined at degree at most $k$ and dimension at most three.
Since every pair of 2-flats of $\flats_w$ intersect in at most one point, every point of $C\setminus\{w\}$ is incident to at most one such flat.
It is not difficult to verify that the points in $\flatsu$ that correspond to flats of $\flats_w$ form a variety.
If this variety is of dimension at least two, then $C$ contains a two-dimensional family of 2-flats that intersect only at $w$, so $\dim C =4$.
This contradicts the above claim that $\dim C\le 3$, so the points of $\flatsu$ that correspond to 2-flats in $\flats_w$ form a subvariety of $\flatsu$ of dimension at most one.

Every 2-flat of $\flatsu$ that is not contained in $U_\text{sing}$ intersects $U_\text{reg}$ in a constructible set of complexity $O_k(1)$ and dimension two.
Since $U$ is four-dimensional and every point of $U_\text{reg}$ is incident to a family of 2-flats of dimension at most one, we conclude that $\flatsu$ is of dimension at most three.

Let $\pi_1: \RR^6 \to \RR^3$ be the projection on the first three coordinates and let $\pi_2: \RR^6 \to \RR^3$ be the projection on the last three coordinates.
As in the proof of Theorem \ref{th:PlanesIn3d}, we set $\gamma_1 = \overline{\pi_1(\flatsu)}$ and $\gamma_2 = \overline{\pi_2(\flatsu)}$.
These are two varieties of dimension at most three defined at degree $O_k(1)$.
We partition the rest of the analysis according to the dimension of $\gamma_1$.

If $\dim \gamma_1 = 0$ then by Lemma \ref{le:BoundedNumComponents} it consists of $O_k(1)$ points.
Each such point can participate in at most $n$ points of $\pts^2 \cap \flatss$, and this sums up to a total of $O_k(n)$ flats of $\flatss$ in $U$.

If $\dim \gamma_1 = 1$ then we may assume that $\gamma_1$ contains $O\left(n^{2/3}\right)$ points of $\pts$, since otherwise we are done.
By Corollary \ref{co:RealUpperSemi}, excluding $O_k(1)$ exceptional points, for every $a\in \gamma_1$ the set of points $w\in \flatsu$ satisfying $\pi_1(w) = a$ is contained in a variety of dimension two defined at degree $O_k(1)$.
Since the set of exceptional points is zero-dimensional, it can be handled as in the case of $\dim \gamma_1 = 0$.
Consider a non-exceptional point $a\in \gamma_1$ and set $\gamma_a = \overline{\pi_2(\pi_1^{-1}(a))}$.
Note that $\gamma_a\subset \RR^3$ is a variety of dimension at most two defined at degree $O_k(1)$.
If $|\gamma_a\cap \pts| = \Omega\left(n^{2/3}\right)$ then we are done.
We may thus assume that every non-exceptional point $a\in \gamma_1 \cap \pts$ satisfies $|\gamma_a\cap \pts| = O\left(n^{2/3}\right)$.
This gives a total of $O_k\left(n^{4/3}\right)$ flats of $\flatss$ in $U$.

If $\dim \gamma_1 = 2$ then we may assume that $\gamma_1$ contains $O\left(n^{2/3}\right)$ points of $\pts$, since otherwise we are done.
By Corollary \ref{co:RealUpperSemi}, there exists a variety $W\subset \RR^3$ of dimension at most one defined at degree $O_k(1)$ such that for every $a\in \gamma_1\setminus W$ we have that $\pi^{-1}(a)$ is a constructible set of dimension at most one and complexity $O_{k}(1)$.
Since $\dim W \le 1$, points on $W$ can be handled as in the case of $\dim \gamma_1 = 1$.
For a point $a\in \gamma_1\setminus W$ we set $\gamma_a = \overline{\pi_2(\pi_1^{-1}(a))}$.
Note that $\gamma_a\subset \RR^3$ is a variety of dimension at most one defined at degree $O_k(1)$.
If $|\gamma_a\cap \pts| = \Omega\left(n^{2/3}\right)$ then we are done.
We may thus assume that every non-exceptional point $a\in \gamma_1 \cap \pts$ satisfies $|\gamma_a\cap \pts| = O\left(n^{2/3}\right)$.
This gives a total of $O_k\left(n^{4/3}\right)$ flats of $\flatss$ in $U$.

Finally, consider the case where $\dim \gamma_1 = 3$.
By Corollary \ref{co:RealUpperSemi}, there exists a variety $W\subset \RR^3$ of dimension at most two defined at degree $O_k(1)$ such that for every $a\in \gamma_1\setminus W$ we have that $\pi^{-1}(a)$ is a set of $O_{k}(1)$ points.
Since $\dim W \le 2$, it can be handled as in the cases of $\dim \gamma_1 \le 2$.
For every non-exceptional $a\in \gamma_1$, we have that $\flatsu$ contains $O_k(1)$ points of $\{a\}\times\pts$.
By summing this over every $a\in \pts\setminus W$ we get a total of $O_k\left(n^{4/3}\right)$ flats of $\flatss$ in $U$.
\end{proof}

We conclude this section by studying the number of 2-flats of $\flatss$ that have a one-dimensional intersection with a given two-dimensional variety.

\begin{theorem}
Let $\pts$ be a set of $n$ points in $\RR^3$ and let $U$ be an irreducible two-dimensional variety in $\RR^5$ defined at degree $k$.
Then either $U$ has a one-dimensional intersection with $O_k\left(n^{4/3}\right)$ flats of $\flatss$ or there exists a two dimensional variety defined at degree $O_k(1)$ in $\RR^3$ that contains $\Omega_k\left(n^{2/3}\right)$ points of $\pts$.
\end{theorem}
\begin{proof}
Set
\[ \flatsu = \left\{(a,p)\in \RR^6 :\ \dim (F_{ap} \cap U) = 1\right\}. \]

By Lemma \ref{le:BoundedNumComponents}, if a 2-flat in $\RR^5$ has a zero-dimensional intersection with $U$, then this intersection consists of $O_k(1)$ points.
Denote this maximum number of intersection points as $\alpha_k$.
A 2-flat in $\RR^5$ intersects $U$ in a variety of dimension at least one if and only if this intersection consists of at least $\alpha_k+1$ points.
That is, $(a,p)\in \flatsu$ if and only if $F_{ap}\neq U$ and there exist $\alpha_k+1$ distinct points of $\RR^5$ that are contained in $U\cap F_{ap}$.
This is a semi-algebraic condition, so $\flatsu$ is semi-algebraic.
By Lemma \ref{le:QuantElim}, the complexity of $\flatsu$ is $O_k(1)$.

By Theorem \ref{th:Singular}, the set of singular points $U_\text{sing}$ is a one-dimensional variety defined at degree $O_k(1)$.
Thus, for a 2-flat to have a one-dimensional intersection with $U_\text{sing}$, the 2-flat must contain a one-dimensional component of $U_\text{sing}$.
Since any pair of 2-flats of $\flats$ intersect in at most one point, the number of 2-flats of $\flats$ that have a one-dimensional intersection with $U_\text{sing}$ is $O_k(1)$.

Let $w$ be a regular point of $U$, and let $\flats_w$ be a set of 2-flats of $\flatsu$ such that $w$ is contained in a one-dimensional component of their intersection with $U$.
Note that for every $F_{ap} \in\flats_w$ we have that $F_{ap} \cap (w+ T_w U)$ is a line (or equal to $F_{ap}$).
Since every pair of 2-flats of $\flats_w$ intersect in at most one point, every point of $(w+ T_w U)\setminus\{w\}$ is incident to at most one such line.
Thus, $\flats_w$ is of dimension at most one.
Since $U$ is two-dimensional and every regular point of $U$ is incident to a one-dimensional subset of flats of $\flatsu$, we conclude that $\flatsu$ is of dimension at most two.

Let $\pi_1: \RR^6 \to \RR^3$ be the projection on the first three coordinates and let $\pi_2: \RR^6 \to \RR^3$ be the projection on the last three coordinates.
As in the preceding proofs, let $\gamma_1 = \overline{\pi_1(\flatsu)}$ and $\gamma_2 = \overline{\pi_2(\flatsu)}$.
By Theorem \ref{th:Projection}, both $\gamma_1$ and $\gamma_2$ are varieties defined at degree $O_k(1)$ and dimension at most two.
If $|\gamma_1\cap \pts|=\Omega(n^{2/3})$ or $|\gamma_2\cap \pts|=\Omega(n^{2/3})$, then we are done.
We may thus assume that $|\gamma_1\cap \pts|=O(n^{2/3})$ and $|\gamma_2\cap \pts|=O(n^{2/3})$.
We partition the rest of the analysis according to the dimension of $\gamma_1$.

If $\dim \gamma_1 = 0$ then by Lemma \ref{le:BoundedNumComponents} it consists of $O_k(1)$ points.
Since $|\gamma_2\cap \pts|=O\left(n^{2/3}\right)$, every point of $\gamma_1$ corresponds to $O\left(n^{2/3}\right)$ elements of $\flatsu$.
By summing this over every point of $\gamma_1$, we get $O\left(n^{2/3}\right)$ flats of $\flatss$ that have a one-dimensional intersection with $U$.

If $\dim \gamma_1 = 1$, then we apply Corollary \ref{co:RealUpperSemi} to it.
We obtain that, excluding $O_k(1)$ exceptional points, for every $a\in \gamma_1$ we have that $\pi^{-1}(a)$ is contained in a variety of dimension at most one defined at degree $O_{k}(1)$.
We denote such a variety as $\gamma_a$, and set $\gamma'_a = \overline{\pi_2(\gamma_a)}$.
By Theorem \ref{th:Projection}, $\gamma'_a \subset \RR^3$ is a variety of dimension at most one and defined at degree $O_k(1)$.
If $|\gamma'_a \cap \pts|=\Omega\left(n^{2/3}\right)$ then we are done.
It remains to handle the case where for every non-exceptional $a$ we have $|\gamma'_a \cap \pts|=O\left(n^{2/3}\right)$.
Recalling also that $|\gamma_1 \cap \pts|=O\left(n^{2/3}\right)$, we get that $O\left(n^{4/3}\right)$ flats of $\flatss$ have a one-dimensional intersection with $U$.

If $\dim \gamma_1 = 2$, we again apply Corollary \ref{co:RealUpperSemi} to it.
This implies that there exists a variety $W$ of dimension at most one defined at degree $O_k(1)$, such that for every $a\in \gamma_1\setminus W$ we have that $\pi^{-1}(a)$ consits of $O_{k}(1)$ points.
Since $\dim W \le 1$, it can be handled as the cases of $\dim \gamma_1 \le 1$.
Recalling that $|\gamma_1\cap \pts| = O\left(n^{2/3}\right)$, we conclude that $O\left(n^{2/3}\right)$ flats of $\flatss$ have a one-dimensional intersection with $U$.
\end{proof}

\end{document}